\documentclass[oneside,10pt,reqno]{amsart}
\usepackage{amssymb,amsmath,amsthm,bbm,enumerate,mdwlist,url,multirow,hyperref,amsthm,paralist,mathrsfs}
\usepackage[pdftex]{graphicx}
\usepackage[shortlabels]{enumitem}
\allowdisplaybreaks

\def\EE{\mathcal{E}}
\def\BB{\mathcal{B}}
\def\PP{\mathbb{P}}
\def\RR{\mathbb{R}}
\def\QQ{\mathbb{Q}}
\def\pr{\mathrm{pr}}

\def\dd{\mathrm{d}}
\def\ee{\mathsf{e}}
\def\sgn{\mathrm{sgn}}
\def\GG{\mathcal{G}}
\def\AA{\mathcal{A}}
\def\MM{\mathcal{M}}
\def\HH{\mathcal{H}}
\def\FF{\mathcal{F}}
\def\JJ{\mathcal{J}}
\def\LLL{\mathrm{L}}
\def\Cl{\mathrm{Cl}}
\def\at{\mathrm{at}}
\def\leb{\mathscr{L}}
\def\BBB{\mathsf{B}}
\def\FFF{\mathbb{F}}
\def\acc{\mathrm{acc}}
\def\WW{\mathbb{W}}

\theoremstyle{definition}
\newtheorem{definition}{Definition}
\theoremstyle{theorem}
\newtheorem{proposition}[definition]{Proposition}
\newtheorem{lemma}[definition]{Lemma}
\newtheorem{theorem}[definition]{Theorem}
\newtheorem{corollary}[definition]{Corollary}
\numberwithin{equation}{section}
\numberwithin{definition}{section}
\theoremstyle{remark}
\newtheorem{remark}[definition]{Remark}
\newtheorem{example}[definition]{Example}

\makeatletter
\@namedef{subjclassname@2020}{%
  \textup{2020} Mathematics Subject Classification}
\makeatother
\begin{document}
\title[Extending the noise of splitting and stability of  Brownian  maxima]{Extending the noise of splitting to its completion and stability of  Brownian  maxima}

\author{Matija Vidmar}
\address{Department of Mathematics, Faculty of Mathematics and Physics, University of Ljubljana, Slovenia}
\email{matija.vidmar@fmf.uni-lj.si}

\author{Jon Warren}
\address{Department of Statistics, University of Warwick, United Kingdom}
\email{j.warren@warwick.ac.uk}

\begin{abstract}
The stochastic noise of splitting, defined initially on the (basic) algebra of finite unions of intervals of the real line, is extended to a largest class of domains. The $\sigma$-fields of this largest extension constitute the completion, in the sense of noise-type Boolean algebras, of the range of the unextended (basic)  noise. The basic noise extends to a given  measurable domain precisely when a certain stability property is met: the times at which a  Brownian motion has local maxima which fall \emph{inside} the domain must remain unaffected under resampling of the Brownian increments \emph{outside} the domain; together with the same being true for the complement of the domain. A set that is equal to an open set modulo a Lebesgue negligible one, with the same holding of its complement, has  this stability property, but others have it too: the extension is non-trivial. Some domains are totally unstable with respect to the indicated resampling, and to them the extension cannot be made. 
\end{abstract}

\thanks{MV acknowledges funding from the Slovenian Research and Innovation Agency (ARIS) under programme No. P1-0448. JW acknowledges financial support of the LMS research in pairs scheme.}



\date{\today}

\maketitle


\section{Introduction}
\subsection{Motivation and agenda}\label{subsection:motivation}

According to Tsirelson  \cite[Abstract]{tsirelson}, whose phrasing  it seems hardly possible to better, ``a [stochastic] noise is a kind of homomorphism from a Boolean algebra of domains to the lattice of $\sigma$-fields.'' In the one-dimensional case \cite{vershik-tsirelson,tsirelson-nonclassical,picard2004lectures} the domains are finite unions of intervals of the  real line (that we think of as the time axis) to which the homomorphism assigns sub-$\sigma$-fields of the probability space that are independent over disjoint time sets. Classical one-dimensional noises are engendered by (maybe infinite-dimensional) L\'evy processes, the $\sigma$-field associated to a given domain being generated by the increments of the L\'evy process within that domain. And while their construction is typically  much more involved, there are,  in fact, many nonclassical noises~\cite{warren,warren-sticky,spectra-harris,watanabe,Jan08thenoise,yan-raimond,schramm,ellis,directed-landscape}. 

Now, a classical noise always admits a sequentially monotonically continuous extension  to all Borel  sets \cite[Lemma~6.18]{picard2004lectures} \cite[Section~3]{tsirelson-arxiv-5} and such  families of $\sigma$-fields, indexed by the measurable subsets of a Borel space, were first studied already by Feldman \cite[Definition~1.1]{feldman}. 
But, and to quote again Tsirelson \cite[top of p.~2]{tsirelson-arxiv-1}: ``What happens to a nonclassical [\ldots] noise? One may hope that it extends naturally to the greatest class  [of domains] acceptable for the given noise. For now, nothing like that is proved, nor even conjectured.'' He splits the problem in two: (a) enlarging the given set of $\sigma$-fields (irrespective of their relation to the domains); (b) extending the given correspondence between the domains and the $\sigma$-fields. 

Item~(a) is treated in \cite{tsirelson} (see also the preprints  \cite{tsirelson-arxiv-1,tsirelson-arxiv-2,tsirelson2011noise}, containing some results which the published version \cite{tsirelson} does not). Namely, Tsirelson recognizes that the image of a noise is a Boolean algebra of $\sigma$-fields, which he calls a  noise(-type) Boolean algebra. Then the largest extension $\overline{\BBB}$ of a noise Boolean algebra $\BBB$ is defined and named the noise(-type) completion of $\BBB$. 
A benefit of this abstract approach is that $\overline{\BBB}$ can be introduced in full generality. However, even when $\BBB$ came about as the range of a noise, the indexation is lost in the process; a priori there is no association of the $\sigma$-fields of $\overline{\BBB}$ back  to domains. 

 Our original motivation was to attempt Item~(b) of Tsirelson's agenda, i.e. to index $\overline{\BBB}$ as a(n extended) noise, for those $\BBB$, which are the images of a given (basic) noise. In this regard, it is notable that a nonclassical noise  cannot be extended sequentially monotonically continuously to all Borel sets, which is a consequence of a seminal result of Tsirelson \cite[Theorem~1.5]{tsirelson}  \cite[Theorem~3.2]{tsirelson-arxiv-5}, answering a long-standing question of Feldman \cite[Problem~1.9]{feldman}. It is this that makes the noise extension problem highly non-trivial  and it has proved too difficult in general for now. Settling for a humbler goal we achieve (b), apparently for the first time in the nonclassical case, for the one-dimensional noise of splitting,  introduced by Warren \cite{warren} and later studied by Watanabe \cite{watanabe-splitting}, see also  \cite[Paragraph~4.4.3]{yan-raimond} (erratum \cite{yan-raimond-erratum}).  Additionally, in constructing and characterizing the extension, we find a series of results concerning the stability of the times of Brownian local maxima, that we find interesting in their own right. 

\subsection{The  landscape in broad strokes}\label{the-landscape}
Before presenting the results of our paper in Subsection~\ref{subsection:contributions}, we are obliged to sketch the concepts involved, and to introduce some notation.

\subsubsection{Continuous products and noises}\label{subsub:cts-products-noises} A continuous product of probability spaces  is a separable probability space equipped with a two-parameter  family of complete sub-$\sigma$-algebras $\bigl({\mathcal F}_{s,t}\bigr)_{s< t}$ indexed by pairs of ordered (extended-real) times  and  having the following two properties for all times $s< t< u$:
 \begin{equation}\label{intro:indep}
 {\mathcal F}_{s,t} \mbox{ is independent of } {\mathcal F}_{t,u};
 \end{equation}
 \begin{equation}\label{intro:join}
 {\mathcal F}_{s,u}={\mathcal F}_{s,t} \vee {\mathcal F}_{t,u}
 \end{equation}
 (also, $\FF_{-\infty,\infty}=\lor_{n\in \mathbb{N}}\FF_{-n,n}=\PP^{-1}([0,1])$ is the domain of the ambient probability measure $\PP$).   Then, by a one-dimensional noise, we mean a continuous product of probability spaces, which is homogeneous in time, so that for each real $h$  one may shift time  by $h$ and this carries ${\mathcal F}_{s,t}$ to ${\mathcal F}_{s+h,t+h}$. To bring it in line with the homomorphism perspective of (the first paragraph of) Subsection~\ref{subsection:motivation} one has  only to consider $\FF_E:=\FF_{s,t}$ as being attached to the intervals $E$ with left endpoint $s$, right endpoint $t$ (the inclusion/exclusion of the endpoints is without significance) and extend this correspondence to the algebra 
 \begin{equation}\label{algebra-elementary}
 \EE:=\{\text{finite unions of intervals of the real line}\}
 \end{equation} by taking joins of $\sigma$-fields (the empty join is the $\PP$-trivial $\sigma$-field $\FF_\emptyset=\PP^{-1}(\{0,1\}$)), obtaining thus  $\FF:=(\EE\ni E\mapsto \FF_E)$, a homomorphism in a sense to be explained in Paragraph~\ref{paragraph:closure-completion}. 

\subsubsection{Stability and sensitivity}\label{paragraph:stability-sensitivity}
Notions of stability and sensitivity  are central to the theory of stochastic  noises.   They concern the behaviour of random variables under infinitesimal resampling. The continuous product $(\FF_{s,t})_{s<t}$ allows us to decompose the whole of the (one-dimensional) noise into independent parts corresponding to finite partitions of $\mathbb{R}$ into  intervals. Each of the parts is replaced with an independent copy of itself with some probability $p>0$, left the same with probability $1-p$; in the limit as the partition becomes finer and finer we get for each (complex-valued) $X\in \LLL^2(\FF_{-\infty,\infty})$ a copy $X_p$. 
Then $X$ is called sensitive if $\lim_{p\downarrow 0}\mathbb{E}[\overline{X}X_p]=0$ and stable if $\lim_{p\downarrow 0}\mathbb{E}[\vert X-X_p\vert^2]=0$. 	The stable  random variables belong to the largest classical subnoise and form the $\LLL^2$-space of a unique complete, so-called stable, $\sigma$-field $\FF_{\mathrm{stb}}$. The orthogonal complement  $\LLL^2(\FF_\mathbb{R})\ominus \LLL^2(\FF_{\mathrm{stb}})$  consists of the sensitive random variables.

\subsubsection{Boolean algebras of $\sigma$-fields, closure and completion}\label{paragraph:closure-completion}
As  indicated already in Subsection~\ref{subsection:motivation}, $\BBB:=\{\FF_E:E\in \EE\}$ is a noise Boolean algebra. 
It means that $\BBB$ is a distributive lattice of complete sub-$\sigma$-fields of $\PP$ for the operations of intersection as meet and the usual join of $\sigma$-fields, which is  also a Boolean algebra for independent complements. Here,  $y$ being an independent complement of $x$ means  that $x\lor y=\sigma(x\cup y)=\FF_\mathbb{R}=\PP^{-1}([0,1])$  (= the unit $1$ of $\BBB$) and that $x$ and $y$ are independent (and thus $x\land y=x\cap y=\FF_\emptyset=\PP^{-1}(\{0,1\})$  (= the $0$ of $\BBB$)); distributivity of  $\BBB$ ensuring that such a $y$ is then unique  \emph{within} $\BBB$ for a given $x\in\BBB$.\footnote{In keeping (and for consistency) with \cite{tsirelson} we  allow ourselves to use the letters $x$, $y$ to denote $\sigma$-fields.\label{footnote:x-y-for-sigma-fields}}  Taking the sequential monotone closure of $\BBB$ gives $\Cl(\BBB)$, which we refer to simply as the closure of $\BBB$. Then the (noise) completion $\overline{\BBB}$ of $\BBB$ consists of  those elements of $\Cl(\BBB)$ that are independently complemented in $\Cl(\BBB)$. As it happens, $\overline{\BBB}$ is a noise Boolean algebra in turn, and is hence the largest one of its kind contained in $\Cl(\BBB)$ and containing $\BBB$. A largest extension of $\BBB$! Note that in defining the completion it is quite natural to stay within the closure, since  it is only the $\sigma$-fields of $\Cl(\BBB)$ that we may think of as being somehow determined by~$\BBB$. When a noise was previously described as a ``homomorphism,'' this referred to  the fact that $\FF:\EE\to \BBB$ is  a homomorphism of Boolean algebras.

\subsubsection{Tanaka's SDE and the splitting noise}
A fairly general recipe for constructing nonclassical noises is to let ${\mathcal F}_{s,t}$ be generated by the evolution of a stochastic flow between times $s$ and $t$, using  a flow which  is associated with a stochastic differential equation (SDE) having a weak solution not measurable with respect to  the driving  Brownian motion \cite{yan-raimond,tsirelson-nonclassical}.  
The noise of splitting arises in this way by considering Tanaka's SDE.  The latter,
 \begin{equation}\label{eq:tanaka-1}
 X_t= x+ \int_0^t \mbox{sgn}(X_s) \dd B_s\qquad\qquad \left(\mbox{sgn}:=\mathbbm{1}_{(0,\infty)}-\mathbbm{1}_{(-\infty,0]}\right),
 \end{equation}  
is  unusual in that it is possible to describe very explicitly how the solution  $X$, which is itself distributed as a Brownian motion,  fails  to be measurable with respect  to the driving Wiener process $B$. The  initial segment of the path up to hitting zero, during which $X$ simply follows the increments of $\sgn(x)B$, is trivial. So, we may as well take $x=0$.  Then, applying Tanaka's formula \cite[Theorem~VI.1.2]{revuz-yor} and Skorokhod's lemma \cite[Lemma~VI.2.1]{revuz-yor},
\begin{equation}\label{eq:tanaka-2}
\vert X_t\vert= B_t+ L_t^X=B_t-\inf\{ B_s: s\in [0,t]\},
\end{equation}
where $L_t^X$ is the $\vert X\vert$-measurable local time accrued by $X$ at $0$. Thus, via \eqref{eq:tanaka-2}, $\vert X\vert$ can be constructed from $B$, and $B$ recovered from $\vert X\vert$. Furthermore,  the signs of the excursions of $X$  from zero are  independent equiprobable random signs,  independent of $ \vert X\vert$, therefore of $B$. Notice that as time passes these random signs  manifest themselves at the instants excursions of $X$ from zero begin, at which times $B$ makes excursions above its running minimum. They are in particular   the times of (certain)  local minima of  $B$.  It is possible to define a stochastic flow  of maps $\big ( X_{s,t}, s\leq t)$ so that each one-point motion $t\mapsto X_{s,t}(x)$ for $t \geq s$ solves \eqref{eq:tanaka-1}, albeit starting  from $x\in \mathbb{R}$ at time $s\in \mathbb{R}$,  all  driven  by the same two-sided Brownian motion $B=(B_t)_{t\in \mathbb{R}}$  \cite{yan-raimond,yan-raimond-tanaka}. Then the flow contains even more randomness than  a single solution to \eqref{eq:tanaka-1} does: $B$ is augmented with a family  of independent random signs, one associated to every time at which $B$ has a local minimum. In  Subsection~\ref{subsection:splitting-noise} we make precise the notion, which at first seems a little strange, of attaching a random sign  to each local minimum of $B$. For now we note that we will be able to construct for any random time $T$, which selects a local minimum of $B$, a corresponding random sign, which we denote $\epsilon_T$. The noise of splitting is then specified by taking,   for all $s< t$,  ${\mathcal F}_{s,t}$ to be generated by the $X_{u,v}$ for $s\leq u< v\leq t$, or equivalently, 
\begin{equation}\label{second-description}
\FF_{s,t}= \FF^{\mathrm{stb}}_{s,t}\lor \sigma(\epsilon_{T_i}:i\in \mathbb{N}), 
\end{equation}
where $\FF^{\mathrm{stb}}_{s,t}:=\sigma(\text{increments of  $B$ on $(s, t)$})$ and where $T=(T_i)_{i\in \mathbb{N}}$ is an(y) $\FF^{\mathrm{stb}}_{s,t}$-measurable enumeration of the times of local minima of $B$ in $(s,t )$. This second description \eqref{second-description} will become our definition of the noise  of splitting, with no need to refer to Tanaka's SDE. But, we prefer to work henceforth with local maxima rather than minima, which at the level of \eqref{eq:tanaka-1} amounts simply to replacing $\mbox{sgn}$ with $-\mbox{sgn}$. The presence of the random signs, which are sensitive, makes the splitting noise nonclassical.

\subsection{Highlights and discussion of results}\label{subsection:contributions}
 Consider the noise of splitting $\FF$, defined on the algebra $\EE$ of \eqref{algebra-elementary}, its range $\BBB=\FF(\EE)$, underlying Brownian motion $B$ and random signs $\epsilon:\{\text{times of local maxima of $B$}\}\to \{-1,1\}$. Noting that $\FF_{\mathrm{stb}}=\sigma(B)$, consider also the Wiener noise $\FF^{\mathrm{stb}}=(\FF^{\mathrm{stb}}_E)_{E\in \MM}$, indexed by the $\sigma$-field $\MM$ of all $\leb$-measurable sets ($\leb:=$ complete Lebesgue measure), generated by $B$:  it means that $\FF^{\mathrm{stb}}_E=\sigma(\int_0^t \mathbbm{1}_E(s) \dd B_s:t\in \mathbb{R})$  for   $E\in \MM$.   We may then outline the gist of our findings as follows.
\begin{enumerate}[(I)]
\item \label{gist:I} \emph{The largest extension of $\FF$ to a noise $\overline{\FF}:\overline{\EE}\to \Cl(\BBB)$, such that $\overline{\FF}_E\cap \FF_{\mathrm{stb}}=\FF^{\mathrm{stb}}_E$ for all $E\in \overline{\EE}$, exists. The extension $\overline{\FF}$ maps the subalgebra $\overline{\EE}$ of $\MM$ onto the completion $\overline{\BBB}$  sequentially monotonically  continuously.  For $E\in \overline{\EE}$  the times of the local maxima  of $B$ which fall in $E$ admit a countable enumeration $T=(T_i)_{i\in I}$ measurable w.r.t. $\FF^{\mathrm{stb}}_E$, and  $\overline{\FF}_E=\FF^{\mathrm{stb}}_E\lor \sigma(\epsilon_{T_i}:i\in I)$. }
\item\label{gist:II}\emph{An  $E\in \MM$ belongs to $\overline{\EE}$ iff  the times of the local maxima of  $B$ falling in $E$ remain unaffected under resampling of the  increments of $B$ on $\mathbb{R}\backslash E$, and vice versa. The extension $\overline\FF$ of $\FF$ is non-trivial in the sense that $\overline{\EE}$ contains domains other than those which are open mod-$\leb$ and whose complement is open mod-$\leb$.}
\item\label{gist:III}\emph{There are $E\in \MM$ (in fact closed $E$) with $\leb(E)>0$ for which the local maxima of $B$  belonging to $E$,  after resampling the increments of $B$ off $E$, are disjoint from those of $B$ prior to the resampling. Such $E$ do not belong to $\overline{\EE}$.} 
\end{enumerate}
Let us expand on \ref{gist:I}-\ref{gist:III}, pointing the reader along the way to where in the text the precise and more exhaustive statements are to be found. 

To start us off, we observe that when attempting to extend a general (one-dimensional) nonclassical noise $\FF=(\EE\ni E \mapsto {\mathcal F}_E)$, two problems manifest themselves. 

Firstly, different  approximations  to the same  domain  can give rise to  differing $\sigma$-algebras from the closure of $\FF(\EE)$. In  Appendix~\ref{appendix:ntba} we show this issue arises in a particularly egregious way: provided the stable part of $\FF$ is non-trivial, for all times $s<t$, there is a sequence  $(A_n)_{n\in \mathbb{N}}$ in $\EE$ such that $\liminf_{n\to\infty}A_n=[s,t]$ a.e.-$\leb$, but $\liminf_{n\to\infty}\FF_{A_n}={\mathcal F}_{\mathrm{stb}}\cap \FF_{[s,t]}\subsetneq \FF_{[s,t]}$. Actually, the latter is a consequence of a more general statement, Theorem~\ref{thm:is-in-closure}, asserting that the stable $\sigma$-field always belongs to the closure of a noise Boolean algebra.\footnote{We feel this is a significant finding in its own right, which gives insight into the structure of the closure. However, we do not directly use it in our study of the splitting noise, hence its presence in an appendix. } In any event, we see that, generically,  we have  non-uniqueness of the $\sigma$-algebra $\FF_E$ associated  with a measurable set $E$ by  sequential monotonic continuity (mod-$\leb$), even in the case that $E$ is an interval, which  dovetails with the fact,  mentioned in Subsection~\ref{subsection:motivation}, that a continuous extension of  $\FF$ to \emph{all} Borel sets is precluded.

Setting aside potential  non-uniqueness, the second issue that arises when extending the noise $\FF$ is that whilst if $E$  and $F$ are disjoint  then (any reasonably defined) ${\mathcal F}_E$ and ${\mathcal F}_{F}$ are always independent (corresponding to \eqref{intro:indep}), the generalisation of property \eqref{intro:join} does not necessarily hold: ${\mathcal F}_E$ and ${\mathcal F}_{F}$ need not generate  ${\mathcal F}_{E\cup F}$. 

In fact, both these problems involve the same phenomenon, which can be thought of as a form of loss of information. It is the nature of this loss of information in the noise of splitting that represents a pervading theme of our paper and we will show how it depends on the ``geometric'' structure of a domain and properties of Brownian motion.

Restricting our attention now to the splitting noise, spectral theory\footnote{An important technical device in  the analysis of a noise is the spectral structure of the commutative von Neumann algebra generated by the conditional expectation operators w.r.t. its $\sigma$-fields.} gives us a way of directly associating,  with every  $E\in \MM$, a definite  $ {\mathcal F}_E\in\Cl(\BBB)$, that agrees with the given $\FF_E$ when $E\in \EE$. This circumvents the issue, noted above, with attempting to define $\FF_E$ via some arbitrary approximation of $E$. But, it amounts only to a very abstract specification. We would like to describe the $\FF_E$, introduced via spectrum, using $B$ and $\epsilon$. 
To this end,  call  $E$ max-enumerable when the times of local maxima of $B$ that fall in $E$ admit an $\FF^{\mathrm{stb}}_E$-measurable countable enumeration. Then it turns out that, whenever  $E$ and  $\mathbb{R}\backslash E$ are both max-enumerable: (1) ${\mathcal F}_E$ can be rendered explicitly as $\FF_E=\FF^{\mathrm{stb}}_E\lor \sigma(\epsilon_{T_i}:i\in I)$, where $T=(T_i)_{i\in I}$ is a(ny) countable enumeration of the times of local maxima of $B$ falling in $E$, measurable w.r.t. $\FF^{\mathrm{stb}}_E$; (2) not only are $\FF_E$ and $\FF_{\mathbb{R}\backslash E}$ independent (which is true always), but also $\FF_E\lor \FF_{\mathbb{R}\backslash E}=\FF_\mathbb{R}$, whence
${\mathcal F}_E\in\overline{\BBB}$. 
  Conversely, $$\overline{\BBB}\subset \{\FF_E:E\in \MM\}.$$ Furthermore, we are able to show that  for no $E\in \MM$  for which, or for the complement of which, the property of max-enumerability fails, does $\FF_E$ belong to $\overline{\BBB}$. Thus, for an $E\in \MM$, the statements
\begin{enumerate}[(i)]
\item\label{knickers:i} $E\in \overline{\EE}:=\{F\in \MM:\FF_F\in \overline{\BBB}\}$ and
\item\label{knickers:ii} $E$ and $\mathbb{R}\backslash E$ are  max-enumerable
\end{enumerate}
are equivalent. In proving that \ref{knickers:i} implies \ref{knickers:ii} we are required  to extract relatively concrete objects (suitably measurable enumerations of the times of local maxima) out of the  abstract definition of $\overline{\BBB}$, which is one of the subtler arguments of this paper. 
We are also able to establish that the resulting indexation $\overline{\FF}:=(\overline{\EE}\ni E\mapsto \FF_E)$ of $\overline{\BBB}$ is temporally homogeneous and that 
it is a homomorphism of Boolean algebras, so that we may rightfully call it a noise. This noise is even sequentially monotonically continuous: for every sequence $(A_n)_{n\in \mathbb{N}}$  in $\overline{\EE}$ that is $\downarrow$ (resp. $\uparrow$) to an $A\in \overline{\EE}$, one has that $\land_{n\in \mathbb{N}}\FF_{A_n}= \FF_A$  (resp. $\lor_{n\in \mathbb{N}}\FF_{A_n}= \FF_A$);  and it is insensitive to $\leb$-negligible sets:  if $E\in \overline{\mathcal{E}}$, $F\in \MM$  and  $F=E$ a.e.-$\leb$, then $F\in \overline{\mathcal{E}}$; and, for $\{E,F\}\subset \overline{\mathcal{E}}$,  $\FF_E=\FF_F$ iff $E=F$ a.e.-$\leb$.  Lastly, $\overline{\FF}$ is the largest extension of $\FF$ within  $\Cl(\BBB)$ that sends complements to independent complements and which restricts to $\FF^{\mathrm{stb}}$ on intersection with $\FF_{\mathrm{stb}}$. This is the content of Theorems~\ref{theorem:extension} and~\ref{theorem:characterization}; compare \ref{gist:I} above.

The preceding is quite agreeable. But it is also specific to the noise of splitting. While we believe the techniques that we develop could be applied to extending some other nonclassical noises, the mileu of  noises is so rich and diverse that significant new ideas would likely be required to achieve the extension to the completion in generality. This applies in particular to black noises, whose stable part is trivial. 

Returning to reporting on the positive outcomes of our work, it emerges that max-enumerability of $E$ is equivalent to the times of local maxima of  $B$ which belong to $E$  being also   times of local maxima of the $E$-censored Brownian motion $\int_0^\cdot \mathbbm{1}_E(s) \dd B_s$. This is a type of stability property, that we call max-total stability, for which we offer several characterizations in Theorem~\ref{thm:stability}, one of them having been set out in \ref{gist:II} above. Whether or not this property holds  depends on the sizes of the gaps of $E$ at small scales, and we  give in Proposition~\ref{proposition:density-of-sets}\ref{towards-stable} a sufficient condition for it in  terms of  the speed of convergence of the local density of $E$.   Thereafter we can provide explicit examples of non-trivial closed $E$ meeting this criterion by constructing them out of the ranges of infinite activity subordinators with positive drifts (Example~\ref{example:stable}\ref{subo:A}). Their complements being open and hence also max-totally stable, the extended indexation $\overline{\FF}$ is defined on them as observed above.

At the other end of the spectrum we have sets $E$, of positive Lebesgue measure, for which the times of local maxima of $B$ are disjoint with the times of local maxima of $\int_0^\cdot \mathbbm{1}_E(s) \dd B_s$. We refer to this as max-total instability of $E$, providing equivalent descriptions for it in Theorem~\ref{theorem:max-unstable} that match up nicely with those of Theorem~\ref{thm:stability}, and one of which was developed in \ref{gist:III} above. Such sets do not belong to the domain of definition of $\overline{\mathcal{F}}$. Remarkably,  Theorem~\ref{thm:is-in-closure} is  already sufficient to ensure existence of ``large'' (in the sense of Lebesgue measure) closed max-totally unstable sets (Example~\ref{example:unstable}). But, and in parallel to the results in the max-totally stable case, we can also give a criterion on the local density of $E$, which ensures its max-total instability (Proposition~\ref{proposition:density-of-sets}\ref{towards-unstable}), and we can elicit examples of closed max-totally unstable sets coming out of the ranges of subordinators (Example~\ref{example:stable}\ref{subo:B}). 

We leave the reader  with the following striking observation that drops out of the abstract results of Appendix~\ref{appendix:ntba} on noise Boolean algebras, and which (it will be seen) relates to the distinction between the max-totally stable and unstable sets. For the maximizer $\tau$ of $B$ on  the interval $[0,1]$ there exists a closed nowhere dense  $E\subset [0,1]$ of $\leb$-measure arbitrarily small (resp. arbitrarily close to $1$) such that $\PP(\tau=\tau_E)>0$ (resp. $=0$). Here $\tau_E$ is the maximizer on $[0,1]$ of $B_E$ that results from $B$ by replacing its increments off $E$ with an independent copy. For a precise rendering of this, see Items~\ref{dichotomy:B}-\ref{dichotomy:A} on p.~\pageref{dichotomy:B}. We stress that it concerns a phenomenon of stability/sensitivity quite distinct from that of Paragraph~\ref{paragraph:stability-sensitivity}. Indeed, like all random variables, so too $\tau$ is stable for the Wiener noise. That it can be sensitive in the preceding sense  (when $\PP(\tau=\tau_E)=0$) is due to the different nature of the perturbation: for the noise version we resample independently uniformly with vanishing probability, whereas here it is  with probability one on  a set of positive $\leb$-measure, however small this set may be.

\subsection{Article structure} The organization of the remainder of this paper is as follows. Section~\ref{sec:preliminaries} lays the groundwork, importantly the splitting noise is put forward in formal terms. Section~\ref{sec:extension} pins down the ``largest extension'' of this noise. We turn to stability/sensitivity of the Brownian maxima and to how this is connected to the nature of the domains to which the splitting noise extends/fails to extend in Section~\ref{section:stability-sensitivity}. The two Appendices~\ref{appendix:ntba}-\ref{appendix:coupling} are stand-alone items that can be read independently of the body of the text (save for Subsections~\ref{miscellaneous}-\ref{subsection:lattice}).

\section{Preliminaries}\label{sec:preliminaries}
We commence herewith a formal exposition.

\subsection{Miscellaneous general notation, vocabulary, conventions}\label{miscellaneous} 
For a measure $\nu$ defined on a $\sigma$-field $\mathcal{N}$: the push-forward of $\nu$ along an $\mathcal{N}$-measurable map $f$ is denoted $f_\star\nu:=(A\mapsto \nu(f\in A))$, the domain of $f_\star \nu$ (which is to say, the $\sigma$-field on the codomain of $f$) being understood from context; for a numerical $\mathcal{N}$-measurable $f$,  the definite (resp. indefinite) integral of $f$ against $\nu$ is written $\nu[f]:=\int f\dd\nu$ (resp. $f\cdot \nu:=(\mathcal{N}\ni A\mapsto \nu [f;A])$, where $\nu[f;A]:=\int_A f\dd\nu$ for $A\in \mathcal{N}$) whenever this is significant (is well-defined). In particular, if $\mathbf{P}$ is a probability, then $\mathbf{P}[X]=\mathbf{E}[X]$ (resp. $\mathbf{P}[X;A]=\mathbf{E}[X\mathbbm{1}_A]$) is the expectation of a numerical random element  $X$ under $\mathbf{P}$ (resp. on the event $A$); similarly, for further a sub-$\sigma$-field $\GG$, $\mathbf{P}[X\vert\GG]=\mathbf{E}[X\vert\GG]$ stands for the $\mathbf{P}$-conditional expectation of $X$ w.r.t. $\GG$, while $\mathbf{P}(A\vert \GG):=\mathbf{P}[\mathbbm{1}_A\vert\GG]$ is the $\mathbf{P}$-conditional probability of an event $A$ given $\GG$.

A measure $\nu$  defined on a measurable space $(\Omega,\Sigma)$ is restricted sometimes to a sub-$\sigma$-field $\Sigma'$ (i.e. a $\sigma$-field on $\Omega$ contained in $\Sigma$), sometimes to a measurable subset $A$ of its base ``sample space'' (i.e. an element of $\Sigma$), these restrictions to be  denoted indiscriminately by $\nu\vert_{\Sigma'}$ and $\nu\vert_A:=\nu\vert_{\Sigma\cap 2^A}$ respectively ($2^X$ is the power set of a set $X$).  If  $\nu$ is $\sigma$-finite, then by the $\nu$-essential union  of a family $\mathcal{V}\subset \Sigma$ we mean the a.e.-$\nu$ uniquely determined   $\text{$\nu$-ess-$\cup$}\mathcal{V}:=V\in \Sigma$ such that $\mathbbm{1}_V=\text{$\nu$-ess sup}\{\mathbbm{1}_V:V\in \mathcal{V}\}$ a.e.-$\nu$, the latter being just the usual essential supremum of a family of measurable (extended-real valued) maps. The notation $\mu\ll \nu$ (resp. $\mu\sim\nu$) is used for absolute continuity (resp. equivalence) of $\mu$ w.r.t. (resp. and)   $\nu$, $\mu$ and $\nu$ being measures defined on the same measurable space. 

We will write $\times_{i\in I}X_i$ for the cartesian product of a family of sets $(X_i)_{i\in I}$; as usual, the same symbol $\times $ will be used for the product of measures.

By an extension of a probability $\mathbf{P}$ we mean a probability $\mathbf{Q}$ together with  a measurable transformation $\mathsf{p}$ satisfying $\mathsf{p}_\star \mathbf{Q}=\mathbf{P}$. However, by a standard abuse of notation,  we would avoid making $\mathsf{p}$ explicit: a random element $Z$ under $\mathbf{P}$  is still written $Z$ under $\mathbf{Q}$, though really we mean  $Z\circ \mathsf{p}$; similarly, if $\GG$ is a sub-$\sigma$-field of $\mathbf{P}$ we transfer it to $\mathbf{Q}$ as $\mathsf{p}^{-1}(\GG)$, but continue to just write $\GG$ etc.

We intend to be pedantic about negligible sets: if an a.s./a.e. qualifier is missing but the claim is not true with certainty, then we are in error. Nevertheless, we shall indulge in the usual confusion between members of $\LLL^2$-spaces as functions on the measure space and equivalence classes thereof: whether the first or the second is intended will be clear from context (e.g. from the presence or absence of an a.e./a.s. qualifier), which should be sufficient to guard against fallacy.

The collection of the Borel sets of $\mathbb{R}$ is denoted $\mathcal{B}_\mathbb{R}$, while $\leb$ is the complete Lebesgue measure on $\mathbb{R}$.  A two-sided Brownian motion is for us a stochastic process $H=(H_t)_{t\in \mathbb{R}}$ for which $(H_t)_{t\in [0,\infty)}$ and $(H_{-t})_{t\in (-\infty,0]}$ are independent univariate Brownian motions null at zero. The indefinite Wiener integral of a locally square-integrable $\leb$-measurable map $f:\mathbb{R}\to \mathbb{R}$ against such an $H$ is  $f\cdot H:=(\mathbb{R}\ni t\mapsto \int_0^t f(s)\dd H_s)$ [$\int_0^tf(s)\dd H_s:=-\int_t^0f(s)\dd H_s$ for $t\in (-\infty,0)$], and we mean in the latter a(ny) continuous version of this process vanishing at zero (there is ambiguity [only] on a negliglible event). For a set $s$ we understand $\vert s\vert:=\infty$ when $s$ is infinite, and otherwise $\vert s\vert$ is the number of points in $s$.  The symbol $\uparrow$ (resp. $\uparrow\uparrow$) means nondecreasing (resp. strictly increasing); analogously we interpret $\downarrow$, $\downarrow\downarrow$.
 
 All Hilbert spaces appearing herein will be complex.  Products of ($\sigma$-finite) measures are automatically completed. 
  
 \subsection{Lattice of sub-$\sigma$-fields and noise(-type) Boolean algebras}\label{subsection:lattice}
 For an arbitrary probability $\mathbf{P}$ we 
 \begin{quote}
\normalsize denote  by $\hat{\mathbf{P}}$ the collection of all $\mathbf{P}$-complete sub-$\sigma$-fields of $\mathbf{P}$
 \end{quote} 
 (the probability $\mathbf{P}$ itself need not be complete, it just means that every element of $\hat{\mathbf{P}}$ contains $\mathbf{P}^{-1}(\{0\})$). $\hat{\mathbf{P}}$ is a bounded complete lattice for meet $\land=$ intersection, join $\lor=$ the $\sigma$-field generated by the union, bottom element 
 \begin{equation*}
 0_\mathbf{P}:=\mathbf{P}^{-1}(\{0,1\})\text{ ($=$ the  $\mathbf{P}$-trivial sets)}
  \end{equation*}
   and top element
   \begin{equation*}
    1_\mathbf{P}:=\mathbf{P}^{-1}([0,1]) \text{ ($=$ the domain of $\mathbf{P}$).}
    \end{equation*}
It must be emphasized that $\hat{\mathbf{P}}$ is in general not distributive \cite[Example~1.1]{vidmar_2019}, however there is distributivity over independent families \cite[Proposition~3.4]{vidmar_2019}.  In connection to the latter it is important to point out that actually
\begin{equation}\label{distributivity-indep-not-complete}
[\land_{n\in \mathbb{N}}(x_n\lor y_n)]\lor 0_\mathbf{P}= [(\land_{n\in \mathbb{N}}x_n)\lor (\land_{n\in \mathbb{N}}y_n)]\lor 0_\mathbf{P}
\end{equation}
 for arbitrary $\downarrow$ sequences  $(x_n)_{n\in \mathbb{N}}$ and $(y_n)_{n\in \mathbb{N}}$ of (not necessarily complete!) sub-$\sigma$-fields of $\mathbf{P}$, \emph{provided} $\lor_{n\in\mathbb{N}}x_n$ and $\lor_{n\in \mathbb{N}}y_n$ are independent. (The proof is the same as in the case of complete sub-$\sigma$-fields; we outline it  briefly for the sake of completeness. By decreasing martingale convergence  and a general property of independent conditioning of products w.r.t. joins \cite[Lemma~2.2]{vidmar_2019}, we compute, for bounded real, respectively  $\lor_{n\in \mathbb{N}}x_n$- and $\lor_{n\in \mathbb{N}}y_n$-measurable $f$ and $g$, 
 \begin{align*}
& \mathbf{P}[fg\vert \land_{n\in \mathbb{N}}(x_n\lor y_n)]=\lim_{n\to\infty}\mathbf{P}[fg\vert x_n\lor y_n]=\lim_{n\to\infty}\mathbf{P}[f\vert x_n]\mathbf{P}[g\vert y_n]\\&=\mathbf{P}[f\vert \land_{n\in \mathbb{N}}x_n]\mathbf{P}[g\vert \land_{n\in \mathbb{N}}y_n]=\mathbf{P}[fg\vert(\land_{n\in \mathbb{N}}x_n)\lor (\land_{n\in \mathbb{N}}y_n)]\text{ a.s.-$\mathbf{P}$};
 \end{align*} and conclude via monotone class. Q.E.D.) A very special case of the latter observation is then that 
\begin{equation}\label{triviality-in-intersection}
\land_{n\in \mathbb{N}}(x_n\lor 0_\mathbf{P})=(\land_{n\in\mathbb{N}}x_n)\lor 0_\mathbf{P}
\end{equation} for all $\downarrow$ sequences  $(x_n)_{n\in \mathbb{N}}$ of (again, not necessarily complete) sub-$\sigma$-fields of a given probability $\mathbf{P}$.\footnote{We may also mention in passing that  in \eqref{distributivity-indep-not-complete}, $0_\mathbf{P}$ depends on $\mathbf{P}$ only up to equivalence, which means that the condition of independence of $\lor_{n\in\mathbb{N}}x_n$ and $\lor_{n\in \mathbb{N}}y_n$ ensuring the validity of \eqref{distributivity-indep-not-complete} can be weakened to: $\lor_{n\in\mathbb{N}}x_n$ and $\lor_{n\in \mathbb{N}}y_n$ are independent under a  probability $\mathbf{Q}\gg \mathbf{P}$. Still, the intervention of the trivial sets of the probability in \eqref{distributivity-indep-not-complete} cannot in  general be dispensed with, not even on a product space \cite{steinicke}.}

  An independent complement of an $x\in\hat{\mathbf{P}}$ is a $y\in \hat{\mathbf{P}}$ such that $x\lor y=1_\mathbf{P}$ and $x$ is independent of $y$ (in which case automatically $x\land y=0_\mathbf{P}$).  Members $x$ and $y$ of the lattice $\hat{\mathbf{P}}$ are said to be commuting if their associated conditional expectations $\mathbf{P}[\cdot\vert x]$ and $\mathbf{P}[\cdot\vert y]$, acting on $\LLL^2( \mathbf{P})$, are \cite[Definition~3.1]{tsirelson}, in which case $\mathbf{P}[\cdot\vert x]\mathbf{P}[\cdot \vert y]=\mathbf{P}[\cdot\vert x\land y]=\mathbf{P}[\cdot\vert y]\mathbf{P}[\cdot \vert x]$.
 
A noise(-type) Boolean algebra under a probability $\mathbf{P}$ (which we now insist is essentially separable, i.e. $\LLL^2(\mathbf{P})$ is separable) is a distributive sublattice $\mathbf{B}$ of $\hat{\mathbf{P}}$ containing $\{0_\mathbf{P},1_\mathbf{P}\}$, every element of which admits an independent complement in $\mathbf{B}$ \cite[Definition~1.1]{tsirelson}. Such $\mathbf{B}$ is indeed a  Boolean algebra for the given lattice  meet, join, and constants $0_\mathbf{P}$, $1_\mathbf{P}$, independent complements playing the role of complementation $(\cdot)'$. True, in $\hat{\mathbf{P}}$ independent complements are in general not unique \cite[Example~1.3]{vidmar_2019}, but within $\mathbf{B}$ they are \cite[(just before) Definition~1.1]{tsirelson}. When using the notation $x'$ to denote the unique independent complement of an $x\in \mathbf{B}$ \emph{within} $\mathbf{B}$,  the $\mathbf{B}$ (relative to which we intend the independent complement) will be understood from context or will be spelled out. 

We put $\Cl(\mathbf{B})$ for the sequential monotone closure of $\mathbf{B}$ (i.e. $\Cl(\mathbf{B})$ is the smallest $C\subset \hat\PP$ such that $\mathbf{B}\subset C$ and such that $C$ is closed for $\uparrow$ joins and $\downarrow$ meets of sequences) and recall  \cite[Theorem~1.6]{tsirelson} that 
 \begin{equation}\label{eq:closure-identified}
 \Cl(\mathbf{B})=\left\{\liminf_{n\to\infty}x_n:x\text{ a sequence in $\mathbf{B}$}\right\}.
 \end{equation} Then the (noise-type) completion 
 \begin{equation}\label{eq:completion-definition}
 \overline{\mathbf{B}}:=\{x\in\Cl(\mathbf{B}):\exists y\in \Cl(\mathbf{B})\text{ such that }y\text{ is an independent complement of }x\}
 \end{equation} of $\mathbf{B}$ is the collection of those members of $\Cl(\mathbf{B})$ that are independently complemented in $\Cl(\mathbf{B})$, and it is a noise Boolean algebra in turn, the largest one containing $\mathbf{B}$ and contained in $\Cl(\mathbf{B})$ \cite[Theorem~1.7]{tsirelson}.  
 Like in any noise Boolean algebra, independent complements are unique also within $\overline{\mathbf{B}}$; for $x\in \overline{\mathbf{B}}$, without ambiguity (as long as $\mathbf{B}$ is understood from context), 
 \begin{quote}
 \normalsize $x'$ denotes the unique independent complement of $x$ in $\overline{\mathbf{B}}$.
 \end{quote}

\subsection{The noise of splitting}\label{subsection:splitting-noise}
Let $\WW$ be the  two-sided Wiener measure on the canonical space $(\Omega_0,\mathcal{B}_{\Omega_0})$ --- coordinate process $W$ --- of continuous paths mapping $\mathbb{R}$ to $\mathbb{R}$, vanishing at zero,  every one of which has countably infinitely many (\emph{times of strict}-, to be understood henceforth without emphasis) local maxima. The last property is not so ``canonical'' but it will save us from a number of a.s. qualifiers in what follows. We complete $\WW$ and use $\WW$ to denote this completion from now on. 

By a selection of a local maximum of $W$ (resp. belonging to/on some set $A\subset \mathbb{R}$) we shall mean a  random variable $\tau$ of $\WW$ taking values in $\mathbb{R}\cup \{\dagger\}$ (resp. $A\cup \{\dagger\}$) such that $\tau$ is a local maximum of $W$  on $\{\tau\ne \dagger\}$. Here $\dagger\notin \mathbb{R}$ is a separate adjoined point and we agree $ \dagger\pm r:=r\pm \dagger:=\dagger$, $\omega(\dagger):=0$ for $r\in\mathbb{R}$, $\omega:\mathbb{R}\to \mathbb{R}$. Unless explicitly stressed otherwise, when talking about a \emph{maximizer} of $W$ we shall always implicitly mean the time of the maximum of $W$ on some non-degenerate compact interval (this interval not depending on the sample path of $W$), setting it for definiteness equal to $\dagger$ in case of non-unique existence or else when it is not a local maximum of $W$ (which happens only with $\WW$-probability zero). For real $s<t$,  the notation for said maximizer of $W$ on $[s,t]$ shall be $\tau_{s,t}$. 

An enumeration of the local maxima of $W$ (resp. belonging to/on some set $A\subset \mathbb{R}$) is a sequence $S=(S_k)_{k\in \mathbb{N}}$ of selections of local maxima of $W$ such that for every $\omega\in \Omega_0$, $\{S_k(\omega):k\in \mathbb{N}\}\backslash \{\dagger\}$ are precisely the local maxima of $\omega$ (resp. belonging to $A$). Thus, for instance, the maximizers of $W$ on intervals with rational endpoints offer an enumeration of the local maxima of $W$. Sometimes we will speak about selections/enumerations of local maxima and maximizers  of  processes, indexed by 	[subintervals of] $\mathbb{R}$ and $\mathbb{R}$-valued, other than $W$: these will then carry the obvious analogous meanings. 

\phantomsection\label{enumeration-page}
Let $\mathsf{S}=(\mathsf{S}_k)_{k\in \mathbb{N}}$ be  an enumeration of the local maxima of $W$ such that, for all $\omega\in \Omega_0$,  $\mathsf{S}(\omega)=(\mathbb{N}\ni k\mapsto \mathsf{S}_k(\omega))$ is injective and real-valued (we do not allow  $\dagger$ for $\mathsf{S}$). Such an enumeration certainly exists because we have insisted that paths in $\Omega_0$ admit precisely denumerably many local maxima.  Put $$\Omega:=\{(\omega,f):\omega\in \Omega_0,\, f\in \{-1,1\}^{\{\text{local maxima of }\omega\}}\}$$ and define $\Theta:\Omega_0\times \{-1,1\}^\mathbb{N}\to \Omega$ by setting $$\Theta(\omega,p):=(\omega,p\circ \mathsf{S}(\omega)^{-1}),\quad (\omega,p)\in \Omega_0\times \{-1,1\}^\mathbb{N}.$$ Pushing forward $\WW\times (\frac{1}{2}\delta_{-1}+\frac{1}{2}\delta_1)^{\times \mathbb{N}}$ --- together with its $\sigma$-field 
 --- along $\Theta$ we get the probability $\PP:=\Theta_\star (\WW\times (\frac{1}{2}\delta_{-1}+\frac{1}{2}\delta_1)^{\times \mathbb{N}})$ on $\Omega$, which does not depend on the choice of the enumeration $\mathsf{S}$. 

The probability $\PP$ supports the one-parameter group $\theta=(\theta_t)_{t\in \mathbb{R}}$ of bimeasurable bijections of~$\Omega$, \phantomsection\label{levy-shift}
$$\theta_t(\omega,f):=(\underbrace{\omega(t+\cdot)-\omega(t)}_{\Delta_t(\omega)},f(t+\cdot)),\quad (\omega,f)\in \Omega,\, t\in \mathbb{R},$$
that leave $\PP$ invariant: $ {\theta_t}_\star\PP=\PP$  (in particular, the L\'evy shift $\Delta_t:\Omega_0\to\Omega_0$ leaves invariant $\WW$) for all $t\in \mathbb{R}$. The random elements $B$ and $\epsilon$ under $\PP$ will be the two coordinate projections: $$\text{$B(\omega,f):=\omega$ and $\epsilon(\omega,f):=f$ for $(\omega,f)\in \Omega$}.$$ We may thus think of the local maxima of the Brownian motion $B$ as being decorated by the random signs of $\epsilon$: for a selection $T$ of a local maximum of $W$, 
\begin{equation*}\phantomsection\label{random-signs-page}
\epsilon_T:=\epsilon(T(B))
\end{equation*}
 is the random sign ``attached'' to the local maximum $T(B)$; here we understand $\epsilon(\dagger):=0$ in case the selection allows for the state $\dagger$, so that $\epsilon_T=(\Omega\ni (\omega,f)\mapsto f(T(\omega))\mathbbm{1}_{\mathbb{R}}(T(\omega)))$. For definiteness we set $\Delta_\dagger(\omega)\equiv \mathsf{w}_0$ ($\mathsf{w}_0\in\Omega_0$ fixed, its choice will not matter) for all $\omega\in \Omega_0$. 
 
 It is immediate from the construction, and important to keep in mind, that for a family $(T_i)_{i\in I}$ of selections of local maxima of $W$, if for all distinct $i$ and $j$ from $I$, $T_i\ne T_j$ a.s.-$\WW$ on $\{T_i\ne\dagger,T_j\ne \dagger\}$, then the $\epsilon_{T_i}$, $i\in I$, are independent mean zero given $B$; if further $\WW(T_i=\dagger)=0$ for all $i\in I$, then $\PP(\epsilon_{T_i}=1\vert B) =\frac{1}{2}$ a.s.-$\PP$ for all $i\in I$, so that the $\epsilon_{T_i}$, $i\in I$, are also jointly independent of $B$ and mutually independent amongst themselves.

 A finite (possibly empty) union of (possibly degenerate) intervals of $\mathbb{R}$ shall be called an elementary set, 
 \begin{quote}\normalsize the algebra of all elementary sets shall be  denoted $\mathcal{E}$. 
 \end{quote}
 For an $A\in \EE$ define $\FF_A$ to be the $\sigma$-field generated by the $\PP$-trivial sets $0_\PP$, the increments of $B$ on $A$, and the random signs $\epsilon_{\mathsf{S}^A(k)}$, $k\in \mathbb{N}$, with $\mathsf{S}^A$  any enumeration of the local maxima of $W$  on $A$  measurable w.r.t. $\sigma(\text{increments of $W$ on $A$})\lor 0_\WW$. [We should have perhaps said in the preceding, to be completely unmabiguous, ``of the local maxima of $W$ \emph{that fall in} $A$'', but we will not emphasize this from now on.] Again one verifies at once that $\FF_A$ does not actually depend on the choice of $\mathsf{S}^A$. The factorization $\FF=(\FF_A)_{A\in \mathcal{E}}$, endowed with the time shifts $\theta$, viewed under the Lebesgue-Rokhlin probability $\PP$ is the  noise 
 \cite[Definitions~3.16 and~3.27, Remark~3.28]{picard2004lectures} of splitting in that: 
 \begin{align}\label{eq.noise1}
 \FF_\mathbb{R}&=1_\PP=\lor_{n\in \mathbb{N}}\FF_{[-n,n]};\\\label{eq.noise2}
 \FF_E\lor\FF_F&=\FF_{E\cup F},\quad \{E,F\}\subset \EE;\\\label{eq.noise3}
\text{$\FF_E$ is independent }&\text{of $\FF_{E\backslash \mathbb{R}}$},\quad E\in \EE;
\end{align}
 and (temporal homogeneity) 
 \begin{equation}\label{eq.noise4}
 \theta_t^{-1}(\FF_E)=\FF_{E+t},\quad E\in \EE,\, t\in \mathbb{R}.
 \end{equation} We do not make here explicit a certain technical issue  --- continuity of the group action of $\theta$ --- which makes sure that we are in fact dealing with a one-dimensional noise in the narrowest sense of \cite[Definition~3.27]{picard2004lectures}
  referring the reader for this instead to \cite[Subsection~7.2]{stationary}, where this is handled more generally for the noise attached to an arbitrary stationary local random countable set over the Wiener noise, of which the local maxima are a prime example. The designation ``factorization'' derives from the equality $\LLL^2(\PP)=\LLL^2(\PP\vert_{\FF_E})\otimes \LLL^2(\PP\vert_{\FF_{\mathbb{R} \backslash E}})$, holding true on account of \eqref{eq.noise1}-\eqref{eq.noise3} for all $E\in \EE$ and valid up to the natural unitary equivalence $$fg\leftrightarrow f\otimes g,\quad (f,g)\in \LLL^2(\PP\vert_{\FF_E})\times \LLL^2(\PP\vert_{\FF_{\mathbb{R} \backslash E}}).$$
  By general distributivity  over independent $\sigma$-fields \eqref{distributivity-indep-not-complete} it is automatic from \eqref{eq.noise2}-\eqref{eq.noise3} that 
  \begin{equation}\label{noise:meet}
 \FF_E\land\FF_F=\FF_{E\cap F},\quad \{E,F\}\subset \EE.
 \end{equation}
  Remark also that, like for any (one-dimensional) noise, the addition or subtraction of a finite number of points from an elementary set $E$ does not affect $\FF_E$ (just because e.g. \eqref{eq.noise2}-\eqref{eq.noise3} imply that the $\FF_{\{t\}}$, $t\in \mathbb{R}$, are independent, which, by the separability of $\LLL^2(\PP)$ and \eqref{eq.noise4} can only happen if $\FF_{\{t\}}=0_\PP$ for all $t\in \mathbb{R}$; now combine this with \eqref{eq.noise2}). 
 
 We write $$\BBB:=\{\FF_A:A\in \mathcal{E}\}\subset \hat\PP$$ for the range of $\FF$, which is a  noise Boolean algebra, moreover $\FF:\EE\to \BBB$ is an epimorphism of Boolean algebras. In particular, the pieces of notation $\Cl(\BBB)$, $\overline{ \BBB}$ and complementation $(\cdot)^\prime$ are defined as put forth in Subsection~\ref{subsection:lattice} (for  $\mathbf{P}=\PP$, $\mathbf{B}=\BBB$).

The $\PP$-complete $\sigma$-field generated by $B$ is equal to the so-called stable $\sigma$-field of $\FF$ \cite[Theorems~5.10 and~6.2]{picard2004lectures}, which we shall denote $\FF_{\mathrm{stb}}$. Then setting $$\FF^{\mathrm{stb}}_A:=\FF_A\land \FF_{\mathrm{stb}}=\sigma(\text{increments of $B$ on $A$})\lor 0_\PP,\quad A\in \mathcal{E},$$ gives a subnoise of $\FF$, the  linear/classical/stable part $\FF^{\mathrm{stb}}=(\FF^{\mathrm{stb}}_A)_{A\in \mathcal{E}}$ of $\FF$ \cite[Theorem~6.15, Remark~6.16]{picard2004lectures}, which, viewed under $\PP\vert_{\FF_{\mathrm{stb}}}$, is (isomorphic to) the classical Wiener noise. \phantomsection\label{for-absolute-consistency}(For absolute consistency with the literature \cite[bottom of p.~5]{picard2004lectures} we ought actually in the preceding to have passed to the quotient of $\PP$ w.r.t. $\FF_{\mathrm{stb}}$ to get a standard probability, but this is just a technical reservation that will be of no consequence here.) 

It is well-known  that  the indexation $\FF^{\mathrm{stb}}$ is continued uniquely to all $\leb$-measurable sets by insisting that the resulting family, which we shall continue to denote $\FF^{\mathrm{stb}}$, respects  sequential monotone limits (mod-$\leb$): for all $\leb$-measurable $A$, for every sequence $(A_n)_{n\in \mathbb{N}}$ of $\leb$-measurable sets that is $\uparrow$ (resp. $\downarrow$) to $A$ a.e.-$\leb$, $\lor_{n\in \mathbb{N}}\FF^{\mathrm{stb}}_{A_n}=\FF^{\mathrm{stb}}_A$ (resp. $\land_{n\in \mathbb{N}}\FF^{\mathrm{stb}}_{A_n}=\FF^{\mathrm{stb}}_A$).  The resulting extension is a homomorphism from the Boolean algebra of $\leb$-measurable sets into its noise Boolean algebra range $\BBB_{\mathrm{stb}}$, which satisfies the property
\begin{equation}\label{eq:wiener-noise-injective}
\text{ for all $\leb$-measurable $E$ and $F$: }\FF^{\mathrm{stb}}_E=\FF^{\mathrm{stb}}_F\text{ iff }E=F\text{ a.e.-$\leb$},
\end{equation}
 indeed $\FF_E^{\mathrm{stb}}=\sigma(\mathbbm{1}_E\cdot B)\lor 0_\PP$ for $\leb$-measurable $E$. Rather than under $\PP$, it will sometimes  be more convenient to consider the Wiener noise with reference to the probability $\WW$ and the accompanying L\'evy time-shifts   $(\Delta_t)_{t\in \mathbb{R}}$. We will  then denote it  $\FF^W$, the indexation being again over all $\leb$-measurable sets, so that 
 \begin{equation*}
\FF^W_A:=\sigma(\mathbbm{1}_A\cdot W)\lor 0_\WW\text{ for $\leb$-measurable $A$.}
\end{equation*}

\subsection{Spectrum of  a noise}\label{subsection:spectrum-of-a-noise}
The conditional expectations $\PP[\cdot\vert \FF_A]$, $A$ running over the elementary sets, acting on the separable Hilbert space $\LLL^2(\PP)$, generate the commutative von Neumann algebra $$\AA:=\{\PP[\cdot\vert\FF_A]:A\text{ an $\leb$-measurable set}\}''$$ (indeed $\PP[\cdot \vert \FF_E]\PP[\cdot\vert\FF_F]=\PP[\cdot\vert \FF_{E\cap F}]=\PP[\cdot\vert\FF_F]\PP[\cdot\vert\FF_E]$ for $\{E,F\}\subset \EE$), which admits a spectral resolution: there is a unitary isomorphism $\Psi$ between $\LLL^2(\PP)$ and a direct integral $\int^\oplus_S \HH_s\mu(\dd s)$ of a measurable field of separable Hilbert spaces, $\mu$ a (complete, finite) standard measure, which carries $\AA$ onto the algebra of diagonalizable operators (as a spatial isomorphism of von Neumann algebras). In particular, for each elementary set $A$, the projection $\PP[\cdot\vert \FF_A]$ corresponds via $\Psi$ to multiplication with $\mathbbm{1}_{S_A}$ for some a.e.-$\mu$ unique measurable subset $S_A$ of the spectral space $S$, called the spectral set of $\FF_A$, and then for $f\in \LLL^2(\PP)$, $\mu_f:=\Vert \Psi(f)\Vert^2\cdot \mu$ satisfies \phantomsection\label{relying-upon}  $$\mu_f(S_A)=\PP[\vert\PP[f\vert \FF_A]\vert^2].$$
Relying upon \cite[Theorem~2.3]{Tsirelson1998UnitaryBM} \cite[Theorem~9b1(a)]{tsirelson-nonclassical} we may and do represent $S$ as the compact subsets of the real line --- members of $S$ to be informally also referred to as ``spectral sets''\footnote{Context will always be sufficient to determine whether we mean a spectral set as a $\mu$-measurable set $S^x$ associated to some element  $x\in \hat\PP$ for which $\PP[\cdot\vert x]\in \AA$, so that $\PP[\cdot\vert x]$ corresponds to multiplication with $\mathbbm{1}_{S^x}$ via $\Psi$, or whether we intend it as an element of $S$.} --- with $S_A$ being those contained in  $A$ (in principle only a.e.-$\mu$, but we fix now once and for all this version) for all $A\in \EE$. The $\sigma$-field $$\Sigma^0:=\sigma_S(\{S_A:A\in \EE\})$$
turns $S$ into a standard measurable space \cite[Appendix~C]{molchanov2005theory} and $\mu$ is just the completion of $\mu\vert_{\Sigma^0}$. Having insisted on this, $\mu\vert_{\Sigma^0}$ is then unique up to equivalence \cite[bottom of p.~274]{tsirelson-nonclassical}, and consequently, by the temporal homogeneity \eqref{eq.noise4} of the noise $\FF$,
\begin{equation} \label{temporal-homogeneity-spectral}
(\cdot+h)_\star\mu\sim \mu,\quad h\in \mathbb{R},
\end{equation}
i.e. $\mu$ is quasi-invariant under translations.  Recall also that \cite[Eq.~(3.11)]{picard2004lectures}  \begin{equation}\label{eq:single-points-in-spectral-set}
\mu(\{s\in S:t\in s\})=0,\quad t\in \mathbb{R}.
\end{equation}

For $\mu$-measurable $V$ let us further set $$H(V):=\Psi^{-1}\left(\int^\oplus_V\HH_s\mu(\dd s)\right)\subset\LLL^2(\PP),$$ so that $H(S_A)=\LLL^2(\PP\vert_{\FF_A})$ for $A\in \EE$, and let us  write $$K(s):=\vert s\vert,\quad s\in S,$$ for the ``counting'' map $K:S\to \mathbb{N}_0\cup \{\infty\}$. The empty set is an atom of $\mu$, indeed the spectral set $S_\emptyset=\{\emptyset\}=\{K=0\}$ corresponds to the one-dimensional projection $\PP[\cdot\vert 0_\PP]=\PP[\cdot]$ (= the expectation operator) onto the space of constants.

For arbitrary $A\subset \mathbb{R}$, let us now put $$S_A:=\{s\in S:s\subset A\}=S\cap 2^A,$$   and let us define $\pr_A:S\to S_A$ by asking that $$\pr_A(s):=
\begin{cases}
s\cap A, & \text{ if }s\cap A\in S \\
\emptyset, &\text{ otherwise}
\end{cases},\quad s\in S.$$  With this notation in hand we remind  the reader that
\cite[p.~10]{Tsirelson1998UnitaryBM} \cite[p.~275]{tsirelson-nonclassical} 
 for a finite partition $P$ of $\mathbb{R}$ into elementary sets (where by \eqref{eq:single-points-in-spectral-set} we may ignore a finite set of points):  $s\cap p\in S$ for $\mu$-a.e. $s$ and 
 $\pr_p:S\to S_p$ is measurable for all $p\in P$; the map 
 $$\sqcup_P:=(\times_{p\in P}S_p\ni O\mapsto \cup_{p\in P}O_p\in S)$$ (sending $P$-tuples from $\times_{p\in P}S_p$  to their union)  is also measurable and 
\begin{equation}\label{equivalent-to-product}
\mu\sim {\sqcup_P}_\star \left(\times_{p\in P}((\pr_p)_\star\mu)\right)\sim {\sqcup_P}_\star \left(\times_{p\in P} \mu\vert_{S_p}\right).
\end{equation}
Besides, with $P$ as above,  for $\mu$-measurable $E_p$ ($p\in P$),
 \begin{equation}\label{tensor-product}
H(\cap_{p\in P}\pr_p^{-1}(E_p))=\otimes_{p\in P}H(E_p\cap S_p)
 \end{equation}
 up to the natural unitary equivalence of $\LLL^2(\PP)$ and $\otimes_{p\in P}\LLL^2(\PP\vert_{\FF_p})$ (\eqref{tensor-product} is true evidently when, for $p\in P$,  $E_p=S_{Q_p}$ with $Q_p\in \EE$, and extends to the general case by an application of Dynkin's lemma). Lest the reader be misled, let us emphasize that  \eqref{equivalent-to-product} does not mean existence of a probability equivalent to $\mu$ under which for all $A\in \EE$, $\pr_A$ and $\pr_{\mathbb{R}\backslash A}$ are independent; this can indeed happen only in the classical case \cite[Extended remark~5.28]{vidmar-fock} \cite[Theorem~9.17]{vidmar-noise}. 
 
 \subsection{Spectral structure of the splitting noise}\label{subsection:spectrum-splitting-gen}
The content of the preceding subsection holds true of any one-dimensional noise; apart from the specific representation of the spectral space as compact subsets of $\mathbb{R}$ and \eqref{temporal-homogeneity-spectral}-\eqref{eq:single-points-in-spectral-set} it actually holds true, mutatis mutandis, in the general context of noise Boolean algebras \cite{tsirelson,vidmar-noise}. A property specific to the noise of splitting \cite{watanabe-splitting,warren} \cite[Section~6.2, esp. Example~6.10]{picard2004lectures} \cite[Example~8b10]{tsirelson-nonclassical}  \cite[Subsection~7.3]{stationary} that we shall use is that 
\begin{quote}
\normalsize for $\mu$-a.e. $s$   the collection  $\acc(s)$  of the accumulation points of $s$ is finite,
\end{quote} 
 corresponding to the fact that the noise of splitting has no super-superchaoses, only the stable part and superchaoses of finite order graded by the ``supercounting'' map \phantomsection\label{super-cuonting-map}  
 $$K'(s):=\vert \acc(s)\vert,\quad s\in S,$$ $\{K'=0\}=\{\acc=\emptyset\}=\{K<\infty\}$ giving the stable subspace  $$H_{\mathrm{stb}}:=H(\{K<\infty\})=\LLL^2(\PP\vert_{\FF_{\mathrm{stb}}}).$$ In particular, the latter equality ensures  that
 \begin{equation}\label{eq:stable-in-vNa}
 \PP[\cdot\vert \FF_\mathrm{stb}]\in \AA\text{ with spectral set equal to }\{K<\infty\}\text{ a.e.-$\mu$}.
 \end{equation}

We denote the first chaos \cite[pp.~67-68]{picard2004lectures} and first superchaos \cite[p.~71]{picard2004lectures} of $\FF$ by $$H^{(1)}:=H(\{K=1\})=\{f\in \LLL^2(\PP):f=\PP[f\vert \FF_A]+\PP[f\vert \FF_{\mathbb{R}\backslash A}]\text{ for all }A\in \EE\}$$ and 
\begin{align*}
H^{(1)\prime}&:=H(\{K'=1\}) \\
&=\{f\in \LLL^2(\PP):f=\PP[f\vert \FF_A\lor\FF_{\mathrm{stb}}]+\PP[f\vert \FF_{\mathbb{R}\backslash A}\lor\FF_{\mathrm{stb}}]\text{ for all }A\in \EE\}
\end{align*} respectively. It is known and not difficult to see that  $$H^{(1)}=\overline{\mathrm{lin}(\{\text{increments of $B$}\})}=\left\{\int_{-\infty}^\infty f(s)\dd  B_s:f\in \LLL^2(\leb)\right\},$$ while $$H^{(1)\prime}=\overline{\mathrm{lin}\{g\epsilon_T:g\in  \LLL^2(\PP\vert_{\FF_{\mathrm{stb}}}),\, T \text{ a selection of  a local maximum of $W$}\}}.$$ The orthogonal complement of the stable subspace $H_{\mathrm{stb}}$ is the so-called sensitive subspace $$H_{\mathrm{sens}}:=\LLL^2(\PP)\ominus H_{\mathrm{stb}}=H(\{K=\infty\}).$$  By \eqref{tensor-product}, writing $\{0<K<\infty\}=\cup_P \cap_{p\in P}\{ K(\pr_p)=1\}$ (resp. $\{K=\infty\}=\cup_P \cap_{p\in P}\{ K'(\pr_p)=1\}$ a.e.-$\mu$ [$\because$ \eqref{eq:single-points-in-spectral-set}, $\mu(K'=\infty)=0$]), where the union is over all finite partitions $P$ of $\mathbb{R}$ into intervals with rational endpoints (say), we see that random variables of the form $\prod_{p\in P}f_p$ with $f_p\in \LLL^2(\PP\vert_{\FF_p})\cap H^{(1)}$ (resp. $f_p\in \LLL^2(\PP\vert_{\FF_p})\cap H^{(1)\prime}$)  for $p\in P$, as $P$ runs over the finite partitions of $\mathbb{R}$ into intervals,   are total in $H_{\mathrm{stb}}\ominus\{\text{constants}\}$ (resp. in $H_{\mathrm{sens}}$).

We draw attention next to the following two general facts:\phantomsection\label{draw-attention}
\vspace{1mm}

($\bullet$) A non-zero $\sigma$-finite measure $m$ on $(\RR,\BB_\RR)$ that is translation invariant up to equivalence (is quasi-invariant under translations), i.e. one which satisfies $(\cdot+h)_\star m\sim m$ for all $h\in \mathbb{R}$, is equivalent to $\leb\vert_{\BB_\mathbb{R}}$ \cite[Proposition~VII.1.11]{bourbaki2004integration}.
\vspace{1mm}

($\dagger$) If a $\sigma$-finite measure $\nu$ on the finite subsets of $\mathbb{R}$ (to be precise, on $\{K<\infty\}$ endowed with the trace  $\Sigma^0\vert_{\{K<\infty\}}$) satisfies $\nu(\{\emptyset\})>0$, $\nu\vert_{\{K=1\}}\sim \leb\vert_{\BB_\mathbb{R}}$ (where we identify a singleton with its only element) and $\nu\sim {\sqcup_P}_\star \left(\times_{p\in P}((\pr_p)_\star\nu)\right)$ on $\{K<\infty\}$ for all finite partitions $P$ of $\mathbb{R}$ into elementary sets, then $\nu$ is equivalent to the law $\Pi$  of a(ny) Poisson point process with finite intensity measure equivalent to $\leb$. Proof: We have only to note that $\{K<\infty\}=\{K<\infty\}\cap \left[\{\emptyset\}\cup  \left(\cup_P \cap_{p\in P}\{ K(\pr_p)=1\}\right)\right]$, where $\cup_P$ is over all finite partitions $P$ of $\mathbb{R}$ into intervals with rational endpoints, and that $\nu$ and $\Pi$ are evidently equivalent on restriction to each member of the denumerable union inside $[\cdots]$. Q.E.D.

\vspace{1mm}
Using $(\bullet$) and ($\dagger$) we are in a position to divulge a little more concerning the nature of the spectral measure $\mu$. Namely, identifying a singleton with its only element,  the measure $\mu$ restricted to   $\{K=1\}$ is equivalent to $\leb$  by $(\bullet)$, translation quasi-invariance \eqref{temporal-homogeneity-spectral} of $\mu$ (which restricts to $\{K=1\}$) and since the first chaos $H^{(1)}=H(\{K=1\})$ is non-void for the noise of splitting. Moreover, from \eqref{equivalent-to-product} (which restricts to $\{K<\infty\}$), $\mu(\{\emptyset\})>0$ and $(\dagger)$, it follows that
\begin{quote}
\normalsize\phantomsection\label{Poisson-character}$\mu\vert_{\{K<\infty\}}$ is equivalent to the law of a(ny) Poisson point process on $\mathbb{R}$ with finite  intensity measure equivalent to $\leb$ 
\end{quote}
 (cf. \cite[Example~9b9]{tsirelson-nonclassical}). Similarly, the measure $\mu$ restricted to sets with a single accumulation point, pushed forward along the map which sends members of $\{K'=1\}$ to their single accumulation points, is also equivalent to $\leb$, where, in order to reach this conclusion, just like before we appeal to  $(\bullet)$, \eqref{temporal-homogeneity-spectral} (which restricts to $\{K'=1\}$, then pushes forward along $\acc\vert_{\{K'=1\}}$) and now the fact that the first superchaos $H^{(1)\prime}=H(\{K'=1\})$ is non-empty for the noise of splitting. In turn, from  \eqref{equivalent-to-product} (which restricts to $\{K'<\infty\}$ and then, noting \eqref{eq:single-points-in-spectral-set}, pushes forward along $\acc$), $\mu(\acc=\emptyset)=\mu(K<\infty)>0$ and $(\dagger)$, we deduce that 
 \begin{quote}
\normalsize  $\acc_\star(\mu\vert_{\{K'<\infty\}})$  is equivalent to the law  of a(ny) Poisson point process on $\mathbb{R}$ with finite  intensity measure equivalent to $\leb$ (of course in the present case $\mu(K'=\infty)=0$ anyway). 
  \end{quote}
  We shall refer to the two findings in display just above as the Poisson character of $\mu$.

\subsection{Some  properties of Brownian motion}\label{subsection:properties-of-BM}
In the  arguments of Section~\ref{section:stability-sensitivity} we shall appeal to a number of, for the most part technical, observations concerning the Wiener process, that are not so immediate, but are nevertheless either easily accessible in the literature or straightforward consequences of well-established results. We gather them here so as to disturb less the natural flow of the text later. The content  of this subsection the reader may skip without it affecting his/her understanding of the results of Section~\ref{section:stability-sensitivity} themselves, or of the other parts of this paper.

Recall from Subsection~\ref{subsection:splitting-noise} the two-sided Brownian motion $W$, viewed under the probability $\WW$, and its L\'evy shifts $(\Delta_t)_{t\in \mathbb{R}}$.

\begin{enumerate}[(1),wide, labelwidth=!, labelindent=0pt] 
\item \label{millar-result} Let $\mathsf{e}$ be an exponential random time independent of $W$, defined under an extension $\QQ$ of $\WW$. Then, denoting by $\tau_\mathsf{e}:=\tau_{0,\ee}$ the maximizer of $W$ on $[0,\mathsf{e}]$, Millar's zero-one law holds true:
\begin{equation}\label{eq:milla-basic}
\cap_{t\in (0,\infty)}0_\QQ\lor \sigma(\Delta_{\tau_\mathsf{e}}\vert_{[0,t]})= 0_\QQ,
\end{equation}
i.e. the germ $\sigma$-field of the increments of $W$ after $\tau_\mathsf{e}$ is $\QQ$-trivial -- see \cite[Theorem~3.1(c)]{millar-zero-one-law} (combined with \eqref{triviality-in-intersection})\footnote{Of course $\Delta_{\tau_\mathsf{e}}=\Delta_{\tau_\mathsf{e}}(W)$, since $W$ is just the identity on $\Omega_0$.}. By the (excursion-theoretic considerations surrounding the) Wiener-Hopf factorization \cite{greenwood-pitman}, $\tau_\ee$ is independent of $\Delta_{\tau_\mathsf{e}}\vert_{[0,\infty)}$ [actually, even $((\mathsf{r}\circ\Delta_{\tau_\mathsf{e}})\vert_{[0,\infty)},\tau_\mathsf{e})$, $\mathsf{r}:=(\Omega_0\ni\omega\mapsto (\mathbb{R}\ni t\mapsto \omega(-t)))$ being time-reversal, is independent of and has the same distribution as $(\Delta_{\tau_\mathsf{e}}\vert_{[0,\infty)},\ee-\tau_\ee)$, but we do not need it], so that via \eqref{distributivity-indep-not-complete} we get 
\begin{equation}\label{eq:millar}
\cap_{t\in (0,\infty)} 0_\QQ\lor \sigma(\tau_\mathsf{e},\Delta_{\tau_\mathsf{e}}\vert_{[0,t]})= \sigma(\tau_\mathsf{e})\lor 0_\QQ.
\end{equation}
To comfort the reader regarding the intervention of the independent randomization with $\ee$, we may note that \eqref{eq:milla-basic} would be equally true if a deterministic time were to replace $\ee$, however procuring the ``$\ee$ deterministic'' analog of \eqref{eq:millar} (therefrom) would be more difficult (and would perhaps not be true), since one would lose the independence coming out of the splitting at the maximum of $W$ before an independent exponential random time. We will use up \eqref{eq:millar} in due course, and while the ``$\ee$ deterministic'' version thereof would be sufficient in that context, that of \eqref{eq:milla-basic} would not be (at least not as far as we can see).
\item\label{loc-unif-a.s.} We shall require (only) a (very crude) stochastic integration convergence result. Namely, if $(h_n)_{n\in \mathbb{N}}$ is a bounded sequence of $\leb$-measurable maps with values in $\mathbb{R}$ converging to some $\leb$-measurable $h:\mathbb{R}\to\mathbb{R}$  then \cite[Theorem~IV.2.12]{revuz-yor} $\lim_{n\to\infty}h_n\cdot W=h\cdot W$  in $\WW$-probability, uniformly on every compact time interval. Consequently (by a diagonalization argument, to handle the localization to compact time intervals) there is a subsequence  $(h_{\mathsf{n}_k})_{k\in \mathbb{N}}$  of $(h_n)_{n\in \mathbb{N}}$ such that $\lim_{k\to\infty}h_{\mathsf{n}_k}\cdot W= h\cdot W$ locally uniformly a.s.-$\WW$.
\item \label{near-loc-max}It will be important to understand the behaviour of $W$ in the neighbourhood of a local maximum. To this end  let $s<t$ be real numbers. Denisov \cite[Theorem~1]{denisov} has shown that the process $( W_h)_{h\in [s,t]}$  a.s.-$\WW$ decomposes at the maximizer $\tau_{s,t}$ of $W$ on $[s,t]$ into two independent scaled fragments, $\left( \frac{W( \tau_{s,t})- W(\tau_{s,t}+u(t-\tau_{s,t}))}{\sqrt{t-\tau_{s,t}}}\right)_{u\in [0,1]}$ and $\left( \frac{W( \tau_{s,t})- W(\tau_{s,t}-u(\tau_{s,t}-s))}{\sqrt{\tau_{s,t}-s}}\right)_{u\in [0,1]}$, which have the distribution of Brownian meanders, and which are independent of $\tau_{s,t}$  (we make an arbitrary convention for the value of these fragments on $\{\tau_{s,t}\notin (s,t)\}$, which is $\WW$-negligible). Also, the law of a Brownian meander is equivalent to the law of a Bessel process of dimension three issuing from zero and restricted to the time interval $[0,1]$ \cite[Eq.~(3.4)]{aspects}, a result due to Imhof \cite{imhof}. Consequently, the a.s. behaviour of the  Brownian motion near a local  maximum is   described by the a.s. behaviour of the 3D Bessel process at small times. In particular, the results of  Dvoretzky and Erd\"{o}s \cite{Dvoretzky1951SomePO} describe the lower envelope of the 3D Bessel  process near zero. For a $\uparrow$ function $g>0$ defined on $(0,\delta)$ for some $\delta>0$ we have then that  \cite[Section~4.12, Eq.~(16)]{itomckean}
\begin{align}
&\WW \left( W(\tau_{s,t}+h)- W( \tau_{s,t}) < \sqrt{h} g(h) \mbox{ for arbitrarily small } h\in (0,\infty) \right)= 0 \mbox{ or }  1\label{asymptotics}\\
&\text{according  to whether}\nonumber\\
&\int_{0+} g(h) \frac{\dd h}{h} \mbox{ converges or diverges,}\nonumber
\end{align}
with a similar statement holding true to the left of the maximum. 
\end{enumerate}

Let now also  $E$ be $\leb$-measurable. For real $t$ and $ u$  put 
\begin{equation}\label{E-t-u-notation}
E_{t,u}:=\leb(E\cap [t,u]),
\end{equation} where we understand  $[t,u]:=[u,t]$ in case $t>u$; also, for definiteness, $E_{t,u}:=0$ if one or both of $t$ and $u$ are equal to $\dagger$. 
\begin{enumerate}[(4),wide, labelwidth=!, labelindent=0pt] 
\item\label{maxima-of-censored}  We may connect the local maxima of the censored process $\mathbbm{1}_E\cdot W$ to those of the Brownian motion that results from the latter by time-change, effectively deleting out $\mathbb{R}\backslash E$ of the time axis, as follows. Define the $\uparrow$ continuous map $\rho:\mathbb{R}\to \mathbb{R}$, $$\rho(t):=
\sgn(t)E_{0,t},\quad t\in \mathbb{R},$$ which we may think of as the distribution function of $\mathbbm{1}_E\cdot \leb$ centered at zero, indeed $\mathbbm{1}_E\cdot \leb=\dd\rho$ and $\rho(0)=0$;  then its $\uparrow\uparrow$ right-continuous inverse $ \zeta$, 
$$\zeta(s):=\inf\{t\in \mathbb{R}:\rho(t)>s\}\in \mathbb{R}\cup\{-\infty\},\quad s\in I:=[\rho((-\infty)+),\rho(\infty-))\cap \mathbb{R}.$$ The real range of $\zeta$  is the complement in $\mathbb{R}$ of the maximal right-open intervals of constancy of $\rho$ and \begin{equation}\label{equality-of-measures}
\zeta_\star \leb\vert_{I}=\mathbbm{1}_E\cdot \leb
\end{equation}
 [\eqref{equality-of-measures} is trivial on left-open right-closed bounded intervals, since $(\rho(u),\infty)\subset \{\zeta>u\}\subset [\rho(u),\infty)$ for all $u\in \mathbb{R}$; it extends to equality of measures via monotone class]. Owing to $\rho\circ \zeta=\mathrm{id}_{I}$, by L\'evy's characterization theorem \cite[Theorem~IV.3.6]{revuz-yor} or directly from the properties of Gaussian families, we see that the $\WW$-a.s. continuous  process\footnote{Extend, if necessary, but only when $\leb(E\cap (-\infty,0])<\infty$, the process $\mathbbm{1}_E\cdot W$ a.s.-$\WW$ continuously to the temporal point $-\infty$.} $(\mathbbm{1}_E\cdot W)_\zeta$  has the law of a two-sided Brownian motion restricted to the interval $I$ 
 [to be referred to below as a possibly-initiated-possibly-killed two-sided Brownian motion]. Furthermore, 
 \begin{align}
& \text{$\WW$-a.s. the local maxima of  $(\mathbbm{1}_E\cdot W)_\zeta$ are in a one-to-one and onto}\label{one-to-one}\\
&\text{correspondence with the local maxima of $\mathbbm{1}_E\cdot W$ via the clock $\zeta$}\nonumber
\end{align} [$\because$ on the one hand, $\WW$-a.s. the local maxima of $(\mathbbm{1}_E\cdot W)_\zeta$ do not belong to one of the at most countably many points that map via $\zeta$ into the right endpoints of the maximal intervals of constancy of $\rho$; on the other hand, $\WW$-a.s.  the  process $\mathbbm{1}_E\cdot W$ is constant on the intervals of constancy of $\rho$]. 
We deduce from \eqref{equality-of-measures}-\eqref{one-to-one} that for $\leb$-measurable $M$,
  \begin{align}\label{aux:loc-max-censored}
&\text{$\mathbbm{1}_E\cdot W$ has local maxima  on $M$ with positive $\WW$-probability ($\Leftrightarrow$ $\WW$-a.s.)} \\
&\text{iff }\leb(M\cap E)>0,\nonumber
\end{align} 
in particular the local maxima of $\mathbbm{1}_E\cdot W$ belong to $E$ a.s.-$\WW$.\phantomsection\label{maxima-of-censored'}
\end{enumerate}

Still $E$ is an $\leb$-measurable set. We infer from \ref{near-loc-max}-\ref{maxima-of-censored}  two more observations concerning the behaviour of the censored processes$$\tilde W:=\mathbbm{1}_E\cdot W\text{ and  }\tilde W':= \mathbbm{1}_{\mathbb{R}\backslash E} \cdot W$$ near the maximizer  $\tilde{\tau}_{s,t}$ of $\tilde W$ on $[s,t]$, which are valid for any real $s<t$ for which  $E_{s,t}>0$ (we continue to use up the notation of  \eqref{E-t-u-notation}).

\begin{enumerate}[(5),wide, labelwidth=!, labelindent=0pt] 
 \item Since, as described in \ref{maxima-of-censored}, $\tilde{W}$ is only a time-change away from being a possibly-initiated-possibly-killed two-sided Brownian motion, \eqref{asymptotics} delivers that, for any $\uparrow$ function $g>0$ defined on $(0,\delta)$ for some $\delta>0$, \begin{align}\label{jon:0}
&\WW \Big( \tilde{W}(\tilde{\tau}_{s,t}+h)- \tilde{W}( \tilde{\tau}_{s,t}) < \sqrt{E_{\tilde{\tau}_{s,t}, \tilde{\tau}_{s,t}+h}}\, g( E_{\tilde{\tau}_{s,t}, \tilde{\tau}_{s,t}+h})\\\nonumber
&\qquad \qquad \qquad \qquad \qquad \qquad \mbox{ for arbitrarily small } h\in (0,\infty) \Big)= 0 \mbox{ or }  1\\\nonumber
&\text{according to whether}\\
&\int_{0+} g(h) \frac{\dd h}{h} \mbox{ converges or diverges;}\nonumber
\end{align}
and analogously to the left of the maximum. 
\end{enumerate}
\begin{enumerate}[(6),wide, labelwidth=!, labelindent=0pt] 
\item After the  time-change of \ref{maxima-of-censored} the process $\tilde W'$ too becomes a possibly-initiated-possibly-killed two-sided Brownian motion and so satisfies the law of the iterated logarithm. Since $\tilde\tau_{s,t}$ is independent of $\tilde W'$ ($\because$ $\tilde W$ and $\tilde W'$ are uncorrelated and jointly Gaussian families, therefore independent), it follows that
\begin{equation}\label{jon:2}
\WW\left(\limsup_{h\rightarrow 0}\frac{ |\tilde W'(\tilde\tau_{s,t} +h) -\tilde W'(\tilde\tau_{s,t})|}  {\sqrt{2(|h|-E_{\tilde\tau_{s,t}, \tilde\tau_{s,t}+h}) \log\log (1/(|h|-E_{\tilde\tau_{s,t}, \tilde\tau_{s,t}+h}))}}=1\right)=1
\end{equation}
with the interpretation of the quotient as being $=1$ in case $E_{\tilde\tau_{s,t}, \tilde\tau_{s,t}+h}=\vert h\vert$. 
\end{enumerate}

\section{The extension and max-enumerability}\label{sec:extension}
We start by identifying the domains of $\mathbb{R}$ which are negligible for the noise of splitting in the sense that $\mu$-a.e. spectral set  avoids them. In an intuitive sense they are those subsets of $\mathbb{R}$, the addition or subtraction of which from a given set we can expect to have (we would hope to be able to insist on having) no effect on the information carried by said set (considered as an indexing domain for the noise).

\begin{proposition}\label{negligible-for-splitting}
Let $A\subset \mathbb{R}$ be an $\leb$-measurable set. Then $S_A$ is $\mu$-measurable. Furthermore,  $S_{\mathbb{R}\backslash A}$ is $\mu$-conegligible [i.e. $\mu(S\backslash S_{\mathbb{R}\backslash A})=0$; we will say that ``$A$ is negligible for the noise''] iff $A$ is $\leb$-negligible. Thus, for $\leb$-measurable $E$ and $F$, $S_E=S_F$ a.e.-$\mu$ iff $E=F$ a.e.-$\leb$.
\end{proposition}
\begin{proof}
The first statement will follow if we can show that $A$ being $\leb$-negligible is sufficient for the negligibility of $A$ for the noise. For once this is established,  we can approximate $A$ from the outside to within $\leb$-measure zero by a $G_\delta$ set $R$: $R=\cap_{n\in \mathbb{N}}G_n$  for a $\downarrow$ sequence $(G_n)_{n\in \mathbb{N}}$ of open subsets of $\mathbb{R}$ and  $\leb(R\backslash A)=0$.  Then, on the one hand, for arbitrary  open $G\subset \mathbb{R}$, by compactness of the spectral sets, $S_G=\cup_{P\in \mathcal{P}} S_P$, $\mathcal{P}$ being finite unions of the connected components of $G$ (of which there are countably many); on the other hand, $S_{R}=\cap_{n\in \mathbb{N}}S_{G_n}$. We deduce that $S_R$ is $\mu$-measurable. However, $S_A\subset S_R$ and $S_R\backslash S_A\subset S\backslash S_{\mathbb{R}\backslash (R\backslash A)}$, which is $\mu$-negligible.  Therefore $S_A$ is $\mu$-measurable.

Now suppose $A$ is $\leb$-negligible. Thanks to $\mu(K'=\infty)=0$, we express, $\mu$-a.e., 
\begin{equation*}
S\backslash S_{\mathbb{R}\backslash A}=\{\acc\cap A\ne \emptyset\}\cup \left(\cup_{O\in \mathcal{O}}\{\acc\subset O,A\cap\pr_{\mathbb{R}\backslash O}\ne \emptyset\}\right),\phantomsection\label{it-now-stands}
\end{equation*} $\mathcal{O}$ being the set of finite unions of open intervals of $\mathbb{R}$ with rational endpoints. The countable union on the r.h.s. of the preceding display is  one of $\mu$-negligible sets by the Poisson character of $\mu$: a Poisson point process $\xi$ of finite intensity measure equivalent to $\leb$ does not meet an $\leb$-negligible set with probability one; therefore
\begin{enumerate}[(I)]
\item $\mu(\acc\cap A\ne \emptyset)=\mu(\acc\cap A\ne \emptyset,K'<\infty)=0$ because $\acc_\star(\mu\vert_{\{K'<\infty\}})$ is equivalent the law of $\xi$, while
\item\label{poisson-char-II} for $O\in \mathcal{O}$, $\mu(\acc\subset O,A\cap\pr_{\mathbb{R}\backslash O}\ne \emptyset)\leq \mu(K(\pr_{\mathbb{R}\backslash O})<\infty,A\cap\pr_{\mathbb{R}\backslash O}\ne \emptyset)=0$, because $(\pr_{\mathbb{R}\backslash O})_\star\mu\sim \mu\vert_{S_{\mathbb{R}\backslash O}}$ and $\mu\vert_{\{K<\infty\}}$ is also equivalent to the law of $\xi$.
\end{enumerate}
Thus $A$ is negligible for the noise. 

Conversely, it is also plain by the Poisson character of $\mu$ that if $A$ has positive Lebesgue measure then $\mu(S\backslash S_{\mathbb{R}\backslash A})>0$ (since e.g. $A$ contains a singleton of $\mu$ with positive $\mu$-measure).

The last claim is now immediate.
\end{proof}
Recall that for $A\in \EE$, $H(S_A)=\LLL^2(\PP\vert_{\FF_A})$. An ideal extension of $\FF$ should preserve the nature of the spectral sets in the sense that for a would-be $E$ to which the noise extends one would hope to be able to insist that $S_E$ is $\mu$-measurable and $\LLL^2(\PP\vert_{\FF_E})=H(S_E)$. This happens indeed in the case of the classical Wiener noise $\FF^{W}$ and it strongly motivates
\begin{definition}
 For an $\leb$-measurable $A$ define $$H_A:=H(S_A)=\Psi^{-1}\left(\int_{S_A}^\oplus\HH_s\mu(\dd s)\right)=\{f\in \LLL^2(\PP):\mu_f(S\backslash S_A)=0\}.$$ 
\end{definition}
 Actually all the $H_A$, for $A$ an $\leb$-measurable set (not just for $A\in \EE$), are the $\LLL^2$-space of a complete sub-$\sigma$-field of $\PP$, as the next proposition demonstrates. Later we will find in Proposition~\ref{proposition:completion-in-spectral-ones} that these associated $\sigma$-fields exhaust  $\overline{\BBB}$, so that, insofar as our goals herein are concerned, nothing will have been sacrificed in restricting attention to them.
\begin{proposition}\label{corollary:H_A-spaces}
If $A$ is $\leb$-measurable then $H_A$ is an $\LLL^2$-space, i.e. there is a (then automatically unique) $\FF_A\in \hat\PP$ such that $H_A=\LLL^2(\PP\vert_{\FF_A})$. Besides, 
\begin{equation}\label{eq:in-closure}
\{\FF_A:A\text{ an $\leb$-measurable set}\}=\{\FF_R:R\text{ a $G_\delta$ subset of $\mathbb{R}$}\}\subset \Cl(\BBB).
\end{equation}
\end{proposition}
\begin{proof}
Existence of $\FF_A$. As noted, for an elementary $A$, $\FF_A$ is just the given $\sigma$-field of the noise $\FF$. For $A=R$ a $G_\delta$ set it follows because by the argument of the first paragraph of the proof of Proposition~\ref{negligible-for-splitting}, in the notation thereof, $H_R=\cap_{n\in \mathbb{N}}H_{G_n}$, while for an open $G$, $H_G=\overline{\cup_{P\in \mathcal{P}}H_P}$, and since (i) the $\downarrow$ intersection of $\LLL^2$-spaces is an $\LLL^2$-space (with associated $\sigma$-field being the intersection) and (ii) the closure of the union of an upwards-directed (w.r.t. inclusion) family of $\LLL^2$-spaces is also an $\LLL^2$-space (with associated $\sigma$-field being the join). For an arbitrary $A$ it holds true by approximation to within $\leb$-measure zero of $A$ from above with a $G_\delta$ set $R$, since indeed $H_A=H_R$ by Proposition~\ref{negligible-for-splitting}. 

It is clear that the $\FF_A$ so identified is unique, and from the manner of the construction  that it belongs to $\Cl(\BBB)$. The argument just above also demonstrates the equality in \eqref{eq:in-closure}. 
\end{proof}
\begin{definition}\label{definition:extended-fields}
We retain in what follows the notation $\FF_A$ for $\leb$-measurable $A$ of Proposition~\ref{corollary:H_A-spaces}.
\end{definition}
Due to the inclusion of \eqref{eq:in-closure}, 
\begin{equation} \label{eq:commuting}
\{\PP[\cdot\vert \FF_A]:A\text{ an $\leb$-measurable set}\}\subset\AA,
\end{equation}
in particular 
\begin{equation}\label{eq:commuting'}
\text{the $\sigma$-fields $\FF_A$, for  $\leb$-measurable $A$, are commuting.}
\end{equation}
Let us develop some further basic properties of the $\sigma$-fields of  Definition~\ref{definition:extended-fields}. 
 \begin{proposition}\label{proposition:extensions-mod}
\leavevmode
 \begin{enumerate}[(i)]
 \item\label{proposition:extensions-mod:i} For $\leb$-measurable $A$, $\FF_A\land \FF_{\mathrm{stb}}=\FF^{\mathrm{stb}}_A$. 
\item\label{proposition:extensions-mod:ii}  For $\leb$-measurable $E$ and $F$, $\FF_E=\FF_F$ iff $E=F$ a.e.-$\leb$. 
\end{enumerate}
 \end{proposition}
  \begin{proof}
  For the first claim we note that directly from the definitions and the properties of the classical Wiener noise $\FF^{\mathrm{stb}}$, 
  \begin{align*}
  \LLL^2(\PP\vert_{\FF_A\land \FF_{\mathrm{stb}}})&=\LLL^2(\PP\vert_{\FF_A})\cap \LLL^2(\PP\vert_{\FF_{\mathrm{stb}}})=H_A\cap H(\{K<\infty\})=H(S_A\cap \{K<\infty\})\\
  &=\LLL^2(\PP\vert_{\FF^{\mathrm{stb}}_A}).
  \end{align*}
As for the second claim,  we have only to apply Proposition~\ref{negligible-for-splitting} (and the definitions).
  \end{proof}
    Proposition~\ref{negligible-for-splitting} combined with Proposition~\ref{proposition:extensions-mod}\ref{proposition:extensions-mod:i} and the fact that for an $\leb$-measurable $A$, $\FF^{\mathrm{stb}}_A$ is trivial iff $A$ is $\leb$-negligible, establishes that $\mu(S_A\backslash \{\emptyset\})=0\text{ iff }\mu(S\backslash S_{\mathbb{R}\backslash A})=0$, i.e. $\mu$-a.e. $s\ne \emptyset$ satisfies $s\not\subset A$ iff $\mu$-a.e. $s$ satisfies $s\cap A=\emptyset$. We stress that this is a property for which we see no reason why it should hold true of every (one-dimensional) noise; in fact, we suspect that in general it is probably false, especially for a black noise, like the noise of coalescence \cite[Chapter~7]{picard2004lectures}.

In turn, we have the following ``continuity of complementation at $0_\PP$, $1_\PP$'': for all $\downarrow$ sequences $(A_n)_{n\in \mathbb{N}}$ in $\EE$, $\land_{n\in \mathbb{N}}\FF_{A_n}=0_\PP$ iff $\lor_{n\in \mathbb{N}}\FF_{\mathbb{R}\backslash A_n}=1_\PP$ (see it by noting that $\cup_{n\in \mathbb{N}}S_{\mathbb{R}\backslash A_n}=S_{\cup_{n\in \mathbb{N}}\mathbb{R}\backslash A_n}$ a.e.-$\mu$ $\because$ of the compactness of the spectral sets); or, even more succinctly still: 
\begin{equation}\label{semi-continuity}
\text{for all $\downarrow$ sequences $(x_n)_{n\in \mathbb{N}}$ in $\BBB$, $\land_{n\in \mathbb{N}}x_n=0_\PP$ iff $\lor_{n\in \mathbb{N}}x_n'=1_\PP$.}
\end{equation}  In relation to \eqref{semi-continuity} it is worthile pointing out that while $\lor_{n\in \mathbb{N}}x_n'=1_\PP\Rightarrow \land_{n\in \mathbb{N}}x_n=0_\PP$ in any noise Boolean algebra, the converse implication is not true in general \cite[Remark~4.1]{tsirelson}. 
We emphasize also that \eqref{semi-continuity} does not mean that the  equality $(\land_{n\in \mathbb{N}}x_n)\lor (\lor_{n\in \mathbb{N}}x_n')=1_\PP$ holds true more strongly for all $ \downarrow$ sequences $(x_n)_{n\in \mathbb{N}}$ of $\BBB$; in fact it cannot \cite[Theorem~1.5]{tsirelson} because $\FF$ is nonclassical!  (Though, it holds if $(x_n)_{n\in \mathbb{N}}$ is $\downarrow$ to an element $x$ of $\BBB$, since then $\land_{n\in \mathbb{N}}(x_n\land x')=0_\PP$, hence $1_\PP=\lor_{n\in \mathbb{N}}(x_n\land x')'=\lor_{n\in \mathbb{N}}x_n'\lor x$ and therefore  $x'=x'\land (\lor_{n\in \mathbb{N}}x_n'\lor x)=\lor_{n\in \mathbb{N}}(x'\land (x_n'\lor x))=\lor_{n\in \mathbb{N}}x_n'$ by the sequential continuity of $\land$ on $\Cl(\BBB)$ \cite[Eq.~(4.11)]{tsirelson}.)

\begin{proposition}\label{proposition:meet-and-join}
We have the following assertions.
\begin{enumerate}[(i)]
\item\label{cap} $\land_{n\in \mathbb{N}}\FF_{A_n}=\FF_{\cap_{n\in \mathbb{N}}A_n}$ for any sequence $(A_n)_{n\in \mathbb{N}}$ of $\leb$-measurable sets. 
\item\label{cup} $\lor_{n\in \mathbb{N}}\FF_{A_n}= \FF_{\cup_{n\in \mathbb{N}}A_n}$ for any sequence $(A_n)_{n\in \mathbb{N}}$ of open subsets of $\mathbb{R}$. 
\item\label{indep} $\FF_E$ and $\FF_{\mathbb{R}\backslash E}$ are independent for every $\leb$-measurable $E$. 
\end{enumerate}
\end{proposition}
In particular, because  $\FF_\mathbb{R}=1_\PP$, \ref{cap} entails that $\FF$ respects finite intersections and thus that $\FF_E\subset \FF_F$ for $\leb$-measurable $E\subset F$. Similarly, since $\FF_\emptyset=0_\PP$, \ref{cup} secures $\FF_{E\cup F}=\FF_E\lor \FF_F$ for open subsets $E$ and $F$ of $\mathbb{R}$.
\begin{proof}[Proof of Proposition~\ref{proposition:meet-and-join}]
\ref{cap} is clear, because for all $s\in S$, $s\subset A_n$ for all $n\in \mathbb{N}$ iff $s\subset \cap_{n\in \mathbb{N}}A_n$.

\ref{cup}.  First, let $(A_n)_{n\in \mathbb{N}}$ be a $\uparrow$ sequence of open subsets of $\mathbb{R}$. Then 
\begin{align*}
\LLL^2(\PP\vert_{\lor_{n\in \mathbb{N}}\FF_{A_n}})&=\overline{\cup_{n\in \mathbb{N}}\LLL^2(\PP\vert_{\FF_{A_n}})}=\overline{\cup_{n\in \mathbb{N}}H(S_{A_n})}\\
&=H(\cup_{n\in \mathbb{N}}S_{A_n})=H(S_{\cup_{n\in \mathbb{N}}A_n})=H_{\cup_{n\in \mathbb{N}}A_n},
\end{align*} where the first equality would be true even with just an arbitrary $\uparrow$ sequence $(x_n)_{n\in \mathbb{N}}$ from $\hat\PP$ in lieu of $(\FF_{A_n})_{n\in  \mathbb{N}}$ (and for a general probability $\PP$), the third equality is a property of spectral resolutions \cite[Eq.~(2.15)]{tsirelson}, while in the fourth equality we appeal to the compactness of the spectral sets: for $s\in S$, $s\subset \cup_{n\in \mathbb{N}}A_n$ (i.e. $s\in S_{\cup_{n\in \mathbb{N}}A_n}$) iff $s\subset A_n$ for some $n\in \mathbb{N}$ (i.e. $s\in \cup_{n\in \mathbb{N}}S_{A_n}$). We conclude, by definition, that  $\lor_{n\in \mathbb{N}}\FF_{A_n}= \FF_{\cup_{n\in \mathbb{N}}A_n}$.

 Second, let $E$ and $F$ be open subsets of $\mathbb{R}$. There is a $\uparrow$ sequence $(O_n)_{n\in \mathbb{N}}$ of open elementary sets with union $E$, similarly there is a $\uparrow$ sequence $(U_n)_{n\in \mathbb{N}}$ of open elementary sets with union $F$. By the previous point we evaluate 
 \begin{align*}
 \FF_{E\cup F}&=\FF_{\cup_{n\in \mathbb{N}}(O_n\cup U_n)}=\lor_{n\in \mathbb{N}}\FF_{O_n\cup U_n}=\lor_{n\in \mathbb{N}}(\FF_{O_n}\lor \FF_{U_n})\\
 &=(\lor_{n\in \mathbb{N}}\FF_{O_n})\lor (\lor_{n\in \mathbb{N}}\FF_{U_n})=\FF_E\lor \FF_F. 
\end{align*} 

\ref{indep}. The $\sigma$-fields $\FF_E$ and $\FF_{\mathbb{R}\backslash E}$ are commuting, as was noted already in \eqref{eq:commuting'}. Also, $\FF_{E}\land \FF_{\mathbb{R}\backslash E}=\FF_{E\cap (\mathbb{R}\backslash E)}=\FF_\emptyset=0_\PP$, where we have used  \ref{cap}. Therefore \cite[Proposition~3.5]{tsirelson} $\FF_E$ and $\FF_{\mathbb{R}\backslash E}$ are independent. 
\end{proof}
\begin{corollary}\label{corollary:F-for-open}
For an open $A\subset \mathbb{R}$, $\FF_A$ is generated by $0_\PP$, the increments of $B$ on $A$ and the random signs $\epsilon_{T_n}$, $n\in \mathbb{N}$, for an(y) $\FF^W_A$-measurable enumeration $(T_n)_{n\in\mathbb{N}}$ of the local maxima of $W$ belonging to $A$; also, $\theta_t^{-1}(\FF_A)=\FF_{A+t}$ for all $t\in \mathbb{R}$. 
\end{corollary}
\begin{proof}
Writing $A$ as the $\uparrow$ union of open elementary sets, apply Proposition~\ref{proposition:meet-and-join}\ref{cup}.
\end{proof}
As previously announced, the $\sigma$-fields of Definition~\ref{definition:extended-fields} cover the completion $\overline{\BBB}$:
\begin{proposition}\label{proposition:completion-in-spectral-ones}
 $\overline{\BBB}\subset \{\FF_A:A\text{ an $\leb$-measurable set}\}$. Furthermore, if for an $\leb$-measurable $E$, $\FF_E\in \overline{\BBB}$, then its independent complement in $\overline{\BBB}$ is  ${\FF_E}'=\FF_{\mathbb{R}\backslash E}$. 
 \end{proposition}
  On the other hand, it is certainly not true that  $\Cl(\BBB)\subset \{\FF_A:A\text{ an $\leb$-measurable set}\}$. Indeed, by Theorem~\ref{thm:is-in-closure}, $\FF_{\mathrm{stb}}\in \Cl(\BBB)$ and it is clear from Proposition~\ref{proposition:extensions-mod} and from \eqref{eq:wiener-noise-injective} that there is no $\leb$-measurable $A$ for which $\FF_{\mathrm{stb}}=\FF_A$. 
The same example shows that if an $x\in \Cl(\BBB)$ admits an independent complement $y$ in $\hat\PP$, then we cannot conclude that $x\in \overline{\BBB}$, in other words it does not follow that $y\in \Cl(\BBB)$. For  $\FF_{\mathrm{stb}}\in \Cl(\BBB)$  admits the independent complement $\FF_{\mathrm{stb}}':=\sigma(\epsilon_{\mathsf{S}_k}:k\in \mathbb{N})\lor 0_\PP$,\footnote{Recall the enumeration $\mathsf{S}$ from p.~\pageref{enumeration-page}.} but does not belong to $\overline{\BBB}$ and neither can then $\FF_{\mathrm{stb}}'$ belong to $\Cl(\BBB)$.
 \begin{proof}[Proof of Proposition~\ref{proposition:completion-in-spectral-ones}] 
 Suppose $x\in \overline{\BBB}$ with independent complement $x'$ from $\overline{\BBB}$. By \eqref{eq:closure-identified} there is a sequence $(A_n)_{n\in \mathbb{N}}$ of elementary sets such that $x=\liminf_{n\to\infty}\FF_{A_n}$ and likewise there is a sequence $(A_n')_{n\in \mathbb{N}}$ of elementary sets such that $x'=\liminf_{n\to\infty}\FF_{A_n'}$. 
 
 It is clear from Proposition~\ref{proposition:meet-and-join}\ref{cap} that $x\subset \FF_{\liminf_{n\to\infty}A_n}$ and $x'\subset \FF_{\liminf_{n\to\infty}A_n'}$. 
 Furthermore, $x\land \FF_{\mathrm{stb}}=(\liminf_{n\to\infty}\FF_{A_n})\land \FF_{\mathrm{stb}}\supset  \liminf_{n\to\infty}(\FF_{A_n}\land \FF_{\mathrm{stb}})$.   Hence, by Proposition~\ref{proposition:extensions-mod}\ref{proposition:extensions-mod:i}, $x\land \FF_{\mathrm{stb}}\supset \liminf_{n\to\infty}\FF^{\mathrm{stb}}_{A_n}=\FF^{\mathrm{stb}}_{\liminf_{n\to\infty}A_n}$, since the classical Wiener noise is sequentially monotonically continuous. By the same token we avail ourselves of the inclusion $x' \land \FF_{\mathrm{stb}}\supset \FF^{\mathrm{stb}}_{\liminf_{n\to\infty}A_n'}$.   The $\sigma$-fields $x$ and $x'$ being independent, $\FF^{\mathrm{stb}}_{\liminf_{n\to\infty}A_n\cap \liminf_{n\to\infty}A_n'}=\FF^{\mathrm{stb}}_{\liminf_{n\to\infty}A_n}\cap \FF^{\mathrm{stb}}_{\liminf_{n\to\infty}A_n'}=0_\PP=\FF_\emptyset$; by \eqref{eq:wiener-noise-injective} it must be that 
\begin{equation}\label{eq:lininf-cap}
(\liminf_{n\to\infty}{A_n})\cap (\liminf_{n\to\infty}{A_n'})=\emptyset\text{ a.e.-$\leb$}.
 \end{equation}

 From Proposition~\ref{proposition:meet-and-join}\ref{indep} we now infer that $\FF_{\liminf_{n\to\infty}A_n}$ and $\FF_{\liminf_{n\to\infty}A_n'}$ are independent. Also,  $x\subset \FF_{\liminf_{n\to\infty}A_n}$, $x'\subset \FF_{\liminf_{n\to\infty}A_n'}$ and $x\lor x'=1_\PP$, a fortiori therefore $\FF_{\liminf_{n\to\infty}A_n}\lor \FF_{\liminf_{n\to\infty}A_n'}=1_\PP$. As certainly  $\FF_{\liminf_{n\to\infty}A_n}\lor \FF_{\liminf_{n\to\infty}A_n'}\subset \FF_{(\liminf_{n\to\infty}A_n)\cup (\liminf_{n\to\infty}A_n')}$, we deduce from Proposition~\ref{proposition:extensions-mod}\ref{proposition:extensions-mod:i}  that $$\FF^\mathrm{stb}_{(\liminf_{n\to\infty}A_n)\cup (\liminf_{n\to\infty}A_n')}=\FF_\mathrm{stb}$$ and so by \eqref{eq:wiener-noise-injective} we must also have \begin{equation}\label{eq:lininf-cup}
 (\liminf_{n\to\infty}A_n)\cup (\liminf_{n\to\infty}A_n')=\mathbb{R}\text{ a.e.-$\leb$}.
 \end{equation}

 Thus $\FF_{\liminf_{n\to\infty}A_n}\in \Cl(\BBB )$ and $\FF_{\liminf_{n\to\infty}A_n'}\in \Cl(\BBB )$ are seen to belong to the completion $\overline{\BBB}$ and are each other's independent complements therein. This being so, $x'\subset \FF_{\liminf_{n\to\infty}A_n'}$ implies $ \FF_{\liminf_{n\to\infty}A_n}={\FF_{\liminf_{n\to\infty}A_n'}}'\subset x$. Together with $x\subset \FF_{\liminf_{n\to\infty}A_n}$ we deduce that $x=\FF_{\liminf_{n\to\infty}A_n}$, which establishes the first claim. By the same token $x'=\FF_{\liminf_{n\to\infty}A_n'}$, which combined with \eqref{eq:lininf-cap}-\eqref{eq:lininf-cup} and Proposition~\ref{proposition:extensions-mod}\ref{proposition:extensions-mod:ii} gives also the second claim. 
 \end{proof}

 \begin{definition}\label{definition:EE-completion}
We set $$\overline{\mathcal{E}}:=\{E\in 2^\mathbb{R}:\text{$E$ is $\leb$-measurable and } \FF_E\in \overline{\BBB}\}$$ and then $\overline{\FF}:=(\overline{\mathcal{E}}\ni E\mapsto \FF_E)$.
 \end{definition}
 Not only do the $\sigma$-fields of Definition~\ref{definition:extended-fields} cover $\overline{\BBB}$, but also $\overline{\FF}$ ``faithfully'' prolongates $\FF$ as a bona fide noise: 
 \begin{theorem}\label{theorem:extension}
 We have the following assertions.
 \begin{enumerate}[(i)]
 \item \label{extension:0}  If $E\in \overline{\mathcal{E}}$, then $F\in \overline{\mathcal{E}}$ for any $\leb$-measurable $F$  that is equal to $E$ a.e.-$\leb$. For $\{E,F\}\subset \overline{\mathcal{E}}$,  $\overline{\FF}_E=\overline{\FF}_F$ iff $E=F$ a.e.-$\leb$. 
  \item\label{extension:i}  $\overline{\FF}$ extends $\FF$ as a homomorphism of  the Boolean algebra $\overline{\mathcal{E}}$ onto the Boolean algebra $\overline{\BBB}$, the spectral set associated to the projection $\PP[\cdot\vert \FF_E]$ of $\AA$ being $S_E$ (a.e.-$\mu$) for all $E\in \overline{\EE}$. Furthermore, for all $t\in \mathbb{R}$, $\overline{\EE}$ is closed for the action of the translation by $t$, and $$\theta_t^{-1}(\FF_E)=\FF_{E+t},\quad E\in \overline{\EE};$$ also, $\overline{\FF}$ ``respects the stable part'': $$\FF_{E}\cap \FF_{\mathrm{stb}}=\FF^{\mathrm{stb}}_E,\quad E\in \overline{\mathcal{\EE}}.$$
  \item\label{extension:ii} Suppose $\GG:\EE'\to \Cl(\BBB)$ is an extension of $\FF$,  which respects complements (meaning: $\GG_{\mathbb{R}\backslash A}$ is an independent complement of $\GG_A$ for all $A\in \EE'$)   and the stable part, $\EE'$ being a class of $\leb$-measurable sets closed under complementation in $\mathbb{R}$. Then $\overline{\FF}$ extends $\GG$.
  \end{enumerate}
 \end{theorem}
  \begin{proof}
  \ref{extension:0}. This follows directly from Proposition~\ref{proposition:extensions-mod}\ref{proposition:extensions-mod:ii} and Definition~\ref{definition:EE-completion}.
  
\ref{extension:i}. For $E\in \overline{\EE}$, by \eqref{eq:commuting}, the projection $\PP[\cdot\vert \FF_E]$ belongs to $\AA$ and since it projects onto $H_E=H(S_E)$ its associated spectral set is $S_E$ (a.e.-$\mu$).

By Proposition~\ref{proposition:completion-in-spectral-ones}, $\overline{\mathcal{E}}$ is closed under complementation and $\overline{\FF}$ respects complementation. Suppose now that $E$ and $F$ are $\leb$-measurable sets for which $\FF_E$ and $\FF_F$ belong to $\overline{\BBB}$. By Proposition~\ref{proposition:completion-in-spectral-ones}  there is an $\leb$-measurable $C$ such that $\overline{\BBB}\ni \FF_E\land \FF_F=\FF_C$. But $\FF_E\land \FF_F=\FF_{E\cap F}$ thanks to Proposition~\ref{proposition:meet-and-join}\ref{cap}. We conclude that $C=E\cap F$ a.e.-$\leb$ and hence $E\cap F\in \overline{\mathcal{E}}$. Thus $\overline{\mathcal{E}}$ is closed under intersections and $\overline{\FF}$ respects meets. Clearly we also have $\mathbb{R}\in \overline{\mathcal{E}}$, $\FF_\mathbb{R}=1_\PP$, and the proof of the first statement is complete.

As for the second, the fact that $\overline{\FF}$ respects the stable part was seen already in Proposition~\ref{proposition:extensions-mod}\ref{proposition:extensions-mod:i}.  By temporal homogeneity of $\FF$ it is plain that $\overline{\EE}$ is left invariant by translations. Besides, if $E=\cap_{n\in \mathbb{N}}O_n$ for a $\downarrow$ sequence $(O_n)_{n\in \mathbb{N}}$ of open subsets of $\mathbb{R}$, then by  Corollary~\ref{corollary:F-for-open}, for all $t\in \mathbb{R}$, $$\theta_t^{-1}(\FF_{O_n})=\FF_{O_n+t},\quad n\in \mathbb{N}.$$ Taking intersection we get from Proposition~\ref{proposition:meet-and-join}\ref{cap} that 
\begin{align*}
\theta_t^{-1}(\FF_E)&=\theta_t^{-1}(\FF_{\cap_{n\in \mathbb{N}}O_n})=\theta_t^{-1}(\land_{n\in \mathbb{N}}\FF_{O_n})=\land_{n\in \mathbb{N}}\theta_t^{-1}(\FF_{O_n})\\
&=\land_{n\in \mathbb{N}}\FF_{O_n+t}=\FF_{\cap_{n\in \mathbb{N}}(O_n+t)}=\FF_{E+t},
\end{align*} where the third equality stems from the following simple observation: if $A=\theta_t^{-1}(A_n)$ for some $A_n\in \FF_{O_n}$ for all $n\in \mathbb{N}$, then we can express $A=\theta_t^{-1}(\limsup_{n\to\infty}A_n)$ and certainly $\limsup_{n\to\infty}A_n\in \land_{n\in \mathbb{N}}\FF_{O_n}$, forcing $A\in \theta_t^{-1}(\land_{n\in \mathbb{N}}\FF_{O_n})$ (the converse inclusion is plain). For a general $\leb$-measurable $E$ we see that $\theta_t^{-1}(\FF_E)=\FF_{E+t}$ by approximation of $E$ from the outside by a $G_\delta$ set to within a set of $\leb$-measure zero.

\ref{extension:ii}. For $A\in \EE'$, $\GG_A\in \Cl(\BBB)$ admits the independent complement $\GG_{\mathbb{R}\backslash A}$ in $\Cl(\BBB)$, therefore $\GG_A\in\overline{\BBB}$, hence there is $E\in \overline{\EE}$ such that $\GG_A=\FF_E$. Taking intersection with $\FF_{\mathrm{stb}}$ we deduce that $\FF^{\mathrm{stb}}_A=\FF^{\mathrm{stb}}_E$, which implies $E=A$ a.e.-$\leb$ by \eqref{eq:wiener-noise-injective}. In turn we obtain $A\in \overline{\EE}$ and $\FF_A=\FF_E=\GG_A$, concluding the argument.
  \end{proof}
Let us  characterize $\overline{\mathcal{E}}$. It will emerge that its members are distinguished by enjoying, together with their complements, the property of
\begin{definition}\label{definition:max-enumerable}
An $\leb$-measurable set $E$ is said to be max-enumerable if the local maxima of $W$ that belong to $E$ admit an $\FF^W_E$-measurable enumeration.
\end{definition}
\emph{Caveat lector:} If an $\leb$-measurable set $E$ is such that the local maxima of $W$ that belong to $E$ admit a $\GG$-measurable enumeration for all $\GG\in \mathfrak{G}\subset \hat\WW$, does it mean that they admit an enumeration measurable w.r.t. $\cap \mathfrak{G}$? No, it does not, for the enumerations in question may  vary as $ \GG\in \mathfrak{G}$ varies and therein lies the rub. Indeed if this was not an issue, then by approximation from the outside by a $G_\delta$ set we would conclude at once that any $\leb$-measurable set is max-enumerable. But such is not the case, as we shall see.

 \begin{proposition}\label{lemma:locmaxstable}
Let $E$ be an $\leb$-measurable set. If a selection $\tau$ of a local maximum of $W$ belonging to $E$ is $\FF^W_E$-measurable, then $\epsilon_\tau$ is $\FF_E$-measurable.
\end{proposition}
\begin{proof}
Approximating $E$ from the outside up to an $\leb$-negligible set we may and do assume $E=\cap_{n\in \mathbb{N}}O_n$ for a $\downarrow$ sequence $(O_n)_{n\in \mathbb{N}}$ of open subsets of $\mathbb{R}$. By Corollary~\ref{corollary:F-for-open}, $\epsilon_\tau$ is $\FF_{O_n}$-measurable for all $n\in \mathbb{N}$. It remains to apply Proposition~\ref{proposition:meet-and-join}\ref{cap}. 
\end{proof}
\begin{proposition}\label{characterization-one}
If an $\leb$-measurable set $E$ is such that both $E$ and $\mathbb{R}\backslash E$ are max-enumerable, then $E\in \overline{\EE}$. 
\end{proposition}
\begin{proof}
We apply Proposition~\ref{lemma:locmaxstable} to $E$ and $\mathbb{R}\backslash E$ to find that $\FF_E\lor \FF_{\mathbb{R}\backslash E}=1_\PP$ (noting that $\FF_E^{\mathrm{stb}}\subset \FF_E$ by Proposition~\ref{proposition:extensions-mod}\ref{proposition:extensions-mod:i}, the same for $\mathbb{R}\backslash E$ in lieu of $E$, and certainly $\FF_E^{\mathrm{stb}}\lor \FF_{\mathbb{R}\backslash E}^{\mathrm{stb}}=\FF_{\mathrm{stb}}$). In addition, by Proposition~\ref{proposition:meet-and-join}\ref{indep}, we have that $\FF_E$ and $\FF_{\mathbb{R}\backslash E}$ are independent, while by \eqref{eq:in-closure} both these $\sigma$-fields belong to $\Cl(\BBB)$. Thus $\FF_E$ and $\FF_{\mathbb{R}\backslash E}$ are each others independent complements in $\Cl(\BBB)$. The very definition \eqref{eq:completion-definition} of the completion of a noise Boolean algebra now entails that $\FF_E$  belongs to $\overline{\BBB}$. In turn, again by definition (of $\overline{\EE}$), we have that $E\in \overline{\EE}$.
\end{proof}
 We seek to prove the converse of Proposition~\ref{characterization-one}. In doing so, we must gradually climb out of the abstractness of the definition of the completion $\overline{\BBB}$ towards more tangible properties.
 We set out on this route by extending \eqref{equivalent-to-product}-\eqref{tensor-product} to members of $\overline{\EE}$.

\begin{proposition}\label{proposition:spectral-projection-for-extension}
Suppose $E\in \overline{\mathcal{E}}$. Then $\{s\cap E,s\backslash E\}\subset S$ for $\mu$-a.e. $s$, also $\pr_E:S\to S_E$ and $\pr_{\mathbb{R}\backslash E}:S\to S_{\mathbb{R}\backslash E}$ are measurable; the map $\sqcup^E:=(S_E\times S_{\mathbb{R}\backslash E}\ni (s,s')\mapsto s\cup s'\in S)$ too is measurable and 
\begin{equation}\label{eq:disjoint-union}
\mu\sim {\sqcup^E}_\star\left[ ((\pr_E)_\star \mu)\times ((\pr_{\mathbb{R}\backslash E})_\star \mu)\right]\sim {\sqcup^E}_\star\left[ \mu\vert_{S_E}\times \mu\vert_{S_{\mathbb{R}\backslash E}}\right].
\end{equation}
Besides, for $\mu$-measurable $F$ and $F'$,
 \begin{equation}\label{tensor-product-complete}
H(\pr_E^{-1}(F)\cap \pr_{\mathbb{R}\backslash E}^{-1}(F'))=H(F\cap S_E)\otimes H(F'\cap S_{\mathbb{R}\backslash E})
 \end{equation}
 up to the natural unitary equivalence of $\LLL^2(\PP)$ and $\LLL^2(\PP\vert_{\FF_E})\otimes \LLL^2(\PP\vert_{\FF_{\mathbb{R}\backslash E}})$.
\end{proposition}
\begin{proof}
We restrict $\Psi$ to $\LLL^2(\PP\vert_{\FF_E})$ and get the spectral resolution $\Psi_E:\LLL^2(\PP\vert_{\FF_E})\to \int^{\oplus}_{S_E}\HH_s\mu(\dd s)$ of the commutative von Numann algebra generated by the conditional expectation operators of  $\BBB_E:=\{x\land \FF_E:x\in \BBB\}=\{\FF_{A\cap E}:A\in \EE\}$ acting on $\LLL^2(\PP\vert_{\FF_E})=H_E$. The same with $\mathbb{R}\backslash E$ in lieu of $E$.

The independent complements property implies that $$\LLL^2(\PP)=\LLL^2(\PP\vert_{\FF_E})\otimes \LLL^2(\PP\vert_{\FF_{\mathbb{R}\backslash E}})=H_E\otimes H_{\mathbb{R}\backslash E}$$ up to the natural unitary equivalence, and (by Theorem~\ref{theorem:extension}\ref{extension:i}) also that $$\FF_A=\FF_{A\cap E}\lor \FF_{A\backslash E},\quad A\in \EE.$$ Thus $\AA=\AA_E\otimes \AA_{\mathbb{R}\backslash E}$, as a consequence of which we get in the natural way the spectral resolution $$\Psi_E\otimes \Psi_{\mathbb{R}\backslash E}:\LLL^2(\PP)\to \int^{\oplus}_{S_E\times S_{\mathbb{R}\backslash E}}[\HH_s\otimes \HH_{s'}](\mu\vert_{S_E}\times \mu \vert_{S_{\mathbb{R}\backslash E}})(\dd (s,s'))$$ of the von Neumann algebra $\AA$, the spectral set associated to $\PP[\cdot \vert\FF_A]$ being $S_{A\cap E}\times S_{A\backslash E}$ (a.e.-$\mu\vert_{S_E}\times \mu \vert_{S_{\mathbb{R}\backslash E}}$) for all $A\in \EE$. Pushing forward $\mu\vert_{S_E}\times \mu \vert_{S_{\mathbb{R}\backslash E}}$ (relative to $\Sigma^0$ on the codomain) through the injective map $\sqcup^E$ and completing gives the standard measure $\tilde\mu$ on $S$. The map $\sqcup^E$ is automatically a mod-$0$ isomorphism between $\mu$ and $\tilde\mu$ \cite[p.~22, Section~2.5, Theorem on isomorphisms]{rohlin} (in particular, $\sqcup^E(S_{E}\times S_{\mathbb{R}\backslash E})$ is $\tilde\mu$-almost certain), therefore we obtain the associated  spectral resolution $\tilde \Psi:\LLL^2(\PP)\to \int_S\tilde \HH_{\tilde s}\tilde \mu(\dd \tilde s)$ of $\AA$ relative to which the spectral set corresponding to $\PP[\cdot \vert\FF_A]$ is $S_A$ (within $\tilde\mu$-equivalence, of course) for all $A\in  \EE$.  We now have two spectral resolutions, $\Psi$ and $\tilde \Psi$, for $\AA$ on the same standard measurable space $(S,\Sigma^0)$ with the same spectral sets associated to the generating multiplicative class of projections $\{\PP[\cdot \vert\FF_A]:A\in \EE\}$ for $\AA$. From this, a Dynkin's lemma argument ensures the existence of an isomorphism between the measure algebra of $\mu$ and the measure algebra of $\tilde\mu$ carrying the $\mu$-equivalence class of $R$ to the $\tilde \mu$-equivalence class of $R$ for all $R\in \Sigma^0$: the class $\mathcal{D}$ of sets $R\in \Sigma^0$ for which the projections on $\LLL^2(\PP)$ associated to $R$ via $\Psi$ and $\tilde \Psi$ are the same is indeed a Dynkin class containing the $\pi$-system $\{S_A:A\in\EE\}$ generating $\Sigma^0$; therefore, in fact, $\mathcal{D}=\Sigma^0$. This in in turn renders $\tilde\mu\sim \mu$. All the claims follow.
\end{proof}
In the next proposition recall from p.~\pageref{super-cuonting-map} that $H^{(1)\prime}$ is the first superchaos and $K'=\vert \acc\vert$ the supercounting map.

\begin{proposition}\label{proposition:weird}
For $E\in \overline{\mathcal{E}}$, up to the natural unitary equivalence of   $\LLL^2(\PP)$ and $\LLL^2(\PP\vert_{\FF_E})\otimes \LLL^2(\PP\vert_{\FF_{\mathbb{R}\backslash E}})$, 
\begin{align*}
H^{(1)\prime}\cap H(\{\acc\subset E\})&= ( H^{(1)\prime}\cap \LLL^2(\PP\vert_{\FF_E}))\otimes \LLL^2(\PP\vert_{\FF^{\mathrm{stb}}_{\mathbb{R}\backslash E}})\\
&=H^{(1)\prime}\cap \LLL^2(\PP\vert_{\FF_E\lor \FF^{\mathrm{stb}}_{\mathbb{R}\backslash E}})=H^{(1)\prime}\cap \LLL^2(\PP\vert_{\FF_E\lor \FF_{\mathrm{stb}}}).
\end{align*}
\end{proposition}
\begin{proof}
 Applying \eqref{tensor-product-complete}  we have 
 \begin{align*}
 H(\{K'(\pr_E)=1,K'(\pr_{\mathbb{R}\backslash E})=0\})&= H(\{K'=1\}\cap S_E)\otimes H(\{K'=0\}\cap S_{\mathbb{R}\backslash E}) \\
 &=( H^{(1)\prime}\cap\LLL^2(\PP\vert_{\FF_E}))\otimes \LLL^2(\PP\vert_{\FF^{\mathrm{stb}}_{\mathbb{R}\backslash E}}).\end{align*}
  But, thanks to the first finding of Proposition~\ref{proposition:spectral-projection-for-extension}, $$\{K'(\pr_E)=1,K'(\pr_{\mathbb{R}\backslash E})=0\}=\{K'=1,\acc\subset E\}\text{ a.e.-$\mu$}$$  (which gives the first equality of the proposition) and also 
 \begin{align*} \{K'(\pr_E)=1,K'(\pr_{\mathbb{R}\backslash E})=0\}&=\{K'=1\}\cap \{K'(\pr_{\mathbb{R}\backslash E})=0\}\\
 &=\{K'=1\}\cap \{K(\pr_{\mathbb{R}\backslash E})<\infty\}\text{ a.e.-$\mu$},
 \end{align*}  while, thanks to \eqref{tensor-product-complete} again, $$H( \{K(\pr_{\mathbb{R}\backslash E})<\infty\})=\LLL^2(\PP\vert_{\FF_E})\otimes \LLL^2(\PP\vert_{\FF^{\mathrm{stb}}_{\mathbb{R}\backslash E}})=\LLL^2(\PP\vert_{\FF_E\lor \FF^{\mathrm{stb}}_{\mathbb{R}\backslash E}})$$ (which combine to give the second). The last equality holds because $\FF^{\mathrm{stb}}_E\subset \FF_E$ and $\FF_{\mathrm{stb}}=\FF^{\mathrm{stb}}_E\lor \FF^{\mathrm{stb}}_{\mathbb{R}\backslash E}$.
\end{proof}

To proceed further we need ``spectral measures'' that, in some intuitive sense which we do not attempt to make precise, correspond to observing simultaneously the Brownian motion  $B$ and the accumulation point of the spectral sets of the first superchaos.

\begin{proposition}\label{proposition:rich-measures-basic}
For each $f\in H^{(1)\prime}$ there exists a unique measure $\mu_f'$  on $1_\WW\otimes \mathcal{B}_{\mathbb{R}}$ such that $$\mu_f'(Z\times A)=\PP[\vert\PP[f\vert \FF_A\lor \FF_{\mathrm{stb}}]\vert^2;B\in Z],\quad (Z,A)\in 1_\WW\times \EE.$$  The first marginal of $\mu_f'$ is absolutely continuous w.r.t. $ \WW$. As $f$ runs over any  countable total subset of $H^{(1)\prime}$, the disintegrations $\frac{\dd \mu'_f}{\dd \WW_1}$ of the $\mu_f'$ against $\WW$ on the first marginal --- viewed as maps from $\Omega_0$ into the space of measures on $\mathcal{B}_\mathbb{R}$: $$\mu_f'(\dd(\omega,t))=\WW(\dd \omega)\left(\frac{\dd \mu'_f}{\dd \WW_1}(\omega)\right)(\dd t),\quad (\omega,t)\in \Omega_0\times \mathbb{R}$$ --- are $\WW$-a.s. purely atomic and (between them) their atoms precisely exhaust the local maxima of $W$. Besides, $$\frac{\dd \mu'_{f+g}}{\dd \WW_1}\ll \frac{\dd \mu'_f}{\dd \WW_1}+\frac{\dd \mu'_g}{\dd \WW_1}\text{ a.s.-$\WW$},\quad \{f,g\}\subset H^{(1)\prime}.$$
\end{proposition}
 \begin{proof}
 Recall the enumeration $\mathsf{S}$ and the notation $\epsilon_T$ for the random sign attached to a selection $T$ of a local maximum of $W$ from p.~\pageref{enumeration-page}.  
 
 Since $f$ is from the first superchaos we may express it as the orthogonal sum $\oplus_{i\in \mathbb{N}}f_i(B)\epsilon_{\mathsf{S}_i}$ for some unique $f_i\in \LLL^2(\WW)$, $i\in \mathbb{N}$, satisfying $\sum_{i\in \mathbb{N}}\WW[\vert f_i\vert^2]<\infty$. Then, for any elementary $A\subset \mathbb{R}$ and any $\WW$-measurable $Z$, taking into account 
 \begin{enumerate}[(I)]
 \item that  $\epsilon_{\mathsf{S}_i}\mathbbm{1}_{\{\mathsf{S}_i(B)\in A\}}$ is $\FF_A\lor \FF_{\mathrm{stb}}$-measurable, while $\PP[\epsilon_{\mathsf{S}_i}\mathbbm{1}_{\{\mathsf{S}_i(B)\in \mathbb{R}\backslash A\}} \vert \FF_A\lor \FF_{\mathrm{stb}}]=0$ a.s.-$\PP$ for all $i\in \mathbb{N}$, and
 \item that the $\epsilon_{\mathsf{S}_i}\mathbbm{1}_{\{\mathsf{S}_i(B)\in A\}}$, $i\in \mathbb{N}$, are independent mean zero given $\FF_{\mathrm{stb}}$, 
 \end{enumerate}
 we express 
 \begin{align*}
& \PP[\vert\PP[f\vert \FF_A\lor \FF_{\mathrm{stb}}]\vert^2;B\in Z]=\PP\left[\left\vert\sum_{i\in \mathbb{N}}f_i(B)\epsilon_{\mathsf{S}_i}\mathbbm{1}_{\{\mathsf{S}_i(B)\in A\}}\right\vert^2;B\in Z\right]\\
  &=\PP\left[\sum_{(i,j)\in \mathbb{N}^2} \overline{f_j}(B)f_i(B)\mathbbm{1}_{\{\mathsf{S}_j(B)\in A,\mathsf{S}_i(B)\in A\}}\PP[\epsilon_{\mathsf{S}_i}\epsilon_{\mathsf{S}_j}\vert\FF_{\mathrm{stb}}];B\in Z\right]\\
 &=\PP\left[\sum_{i\in \mathbb{N}}\vert f_i(B)\vert^2\mathbbm{1}_{Z\times A}(B,\mathsf{S}_i(B))\right]=\sum_{i\in \mathbb{N}}\WW[\vert f_i\vert^2;(W,\mathsf{S}_i)\in Z\times A],
 \end{align*} which proves existence of $\mu_f'$, indeed we may take \begin{equation}\label{eq:rich}
 \mu_f'=\sum_{i\in \mathbb{N}}{(W,\mathsf{S}_i)}_\star [\vert f_i\vert^2\cdot \WW].
 \end{equation}
 
  Uniqueness of $\mu'_f$ is by Dynkin's lemma: any two candidates for $\mu_f'$ agree on the $\pi$-system $\{Z\times A:(Z,A)\in 1_\WW\times \EE\}$ generating $1_\WW\otimes \mathcal{B}_{\mathbb{R}}$ and have the same finite mass $\PP[\vert f\vert^2 ]$, therefore they agree everywhere on $1_\WW\otimes \mathcal{B}_{\mathbb{R}}$. 
  
  From \eqref{eq:rich} we see that the first marginal of $\mu'_f$ is absolutely continuous w.r.t. $\WW$ and that, moreover, 
  \begin{equation}\label{eq:rich-rich}
  \frac{\dd \mu'_f}{\dd \WW_1}=\sum_{i\in \mathbb{N}}\vert f_i\vert^2\delta_{\mathsf{S}_i}\text{ a.s.-$\WW$},
  \end{equation} which is evidently purely atomic $\WW$-a.s., the atoms being $\WW$-a.s. contained in $\{\mathsf{S}_i:i\in \mathbb{N}\}$, i.e. in the local maxima of $W$.
  
Let $\{f,g\}\subset H^{(1)\prime}$.   By the elementary inequality $\vert a+b\vert^2\leq 2(\vert a\vert ^2+\vert b\vert^2)$, which is valid for $\{a,b\}\subset \mathbb{C}$ and because $(f+g)_i=f_i+g_i$ for $i\in \mathbb{N}$, we see from \eqref{eq:rich-rich} that  $  \frac{\dd \mu'_{f+g}}{\dd \WW_1}\leq 2(  \frac{\dd \mu'_f}{\dd \WW_1}+  \frac{\dd \mu'_g}{\dd \WW_1})$ a.s.-$\WW$, a fortiori $  \frac{\dd \mu'_{f+g}}{\dd \WW_1} \ll  \frac{\dd \mu'_f}{\dd \WW_1}+  \frac{\dd \mu'_g}{\dd \WW_1}$ a.s.-$\WW$.

Suppose now that $F\subset  H^{(1)\prime}$ is countable and total. Then  $\WW(\cup_{f\in F}\{f_i\ne 0\})=1$ for all $i\in \mathbb{N}$, since the converse would imply that for some $i\in \mathbb{N}$, the vector $\mathbbm{1}_{\Omega_0\backslash \cup_{f\in F}\{f_i\ne 0\}}(B)\epsilon_{\mathsf{S}_i}$ belongs to $H^{(1)\prime}\backslash \{0\}$ and is orthogonal to $F$, which is a contradiction. If we combine this observation with \eqref{eq:rich-rich} we see that $\WW$-a.s. the atoms of the disintegrations   $\frac{\dd \mu'_f}{\dd \WW_1}$ as $f$ runs over $F$ exhaust the local maxima of $W$.
 
Thus all the assertions of the proposition are established.
 \end{proof}
\begin{definition}\label{definition:measures-rich}
For $f\in H^{(1)\prime}$ we retain in what follows the two pieces of notation $\mu_f'$, $\frac{\dd \mu'_f}{\dd \WW_1}$  of Proposition~\ref{proposition:rich-measures-basic}.
\end{definition}
 \begin{remark}\label{remark:representation-of-mu-f'}
 The proof of Proposition~\ref{proposition:rich-measures-basic} has shown that for all $f\in H^{(1)\prime}$ one has the representation \eqref{eq:rich}: $$\mu_f'=\sum_{i\in \mathbb{N}}{(W,\mathsf{S}_i)}_\star [\vert f_i\vert^2\cdot \WW],$$ where the enumeration  $\mathsf{S}$ is from p.~\pageref{enumeration-page} and $f=\oplus_{n\in \mathbb{N}}f_i(B)\epsilon_{\mathsf{S}_i}$ is an orthogonal sum decomposition ($f_i\in \LLL^2(\WW)$ for $i\in \mathbb{N}$).
 \end{remark}

Let us expose the salient properties of the measures of Definition~\ref{definition:measures-rich}.
\begin{proposition}\label{proposition:rich-measures}
Suppose $E$ is $\leb$-measurable and $f\in H^{(1)\prime}\cap \LLL^2(\PP\vert_{\FF_E})$. Then $\frac{\dd \mu'_f}{\dd \WW_1}$  is  $\WW$-a.s. carried by the local maxima of $W$ belonging to $E$,  and is measurable w.r.t. $\FF^W_E$.
\end{proposition}
\begin{proof}
For an open $E$, this is obtained readily via  a computation  akin to the one leading to \eqref{eq:rich}, since in this case an arbitrary element of $H^{(1)\prime}\cap \LLL^2(\PP\vert_{\FF_E})$ admits an orthogonal decomposition of the form  $\oplus_{i\in \mathbb{N}}f_i(B)\epsilon_{T_i}$ for some $f_i\in \LLL^2(\WW\vert_{\FF^W_E})$, $i\in \mathbb{N}$, with $\sum_{i\in \mathbb{N}}\WW[\vert f_i\vert^2]<\infty$, $T$ an $\FF^W_E$-measurable enumeration of the local maxima of $W$ belonging to $E$ satisfying $T_i\ne T_j$ a.s.-$\WW$ on $\{T_i\ne\dagger,T_j\ne\dagger\}$ for $i\ne j$ from $\mathbb{N}$. The latter decomposition property one can, in turn, glean from Corollary~\ref{corollary:F-for-open},  which via functional monotone class guarantees that the union of the families $M_k:=\{g(B)\prod_{j\in K}\epsilon_{T_j}:g\in \LLL^2(\WW\vert_{\FF^W_E}) \text{ bounded},\, K\subset \mathbb{N},\, \vert K\vert=k\}$, $k\in \mathbb{N}_0$, is total in $\LLL^2(\PP\vert_{\FF^W_E})$, but only $M_1$ is not orthogonal to $H^{(1)\prime}\supset M_1$. 

For a general $E$ we approximate it to within  $\leb$-measure zero from the outside by a $G_\delta$ set and note that the local maxima of $W$ have an empty intersection with an $\leb$-measure zero set, as well as the  $\downarrow$ sequential  continuity (mod-$\leb$) of $\FF^W$.
\end{proof}
A key factorization property of the spectral measures $\mu_f'$, $f\in H^{(1)\prime}$, is contained in
\begin{proposition}\label{proposition:tensor}
For $E\in \overline{\EE}$,  $f\in H^{(1)\prime}\cap \LLL^2(\PP\vert_{\FF_E})$ and $g\in  \LLL^2(\PP\vert_{ \FF^W_{\mathbb{R}\backslash E}})$, we have that 
$$\frac{\dd\mu'_{ g(B)\cdot f}}{\dd \WW_1}=\vert g\vert^2\frac{\dd\mu'_f}{\dd \WW_1}\quad \text{a.s.-$\WW$}.$$
\end{proposition}
\begin{proof}
Remark that, owing to $E\in \overline{\EE}$, $f\in H^{(1)\prime}\cap \LLL^2(\PP\vert_{\FF_E})$ and $g\in  \LLL^2(\PP\vert_{ \FF^W_{\mathbb{R}\backslash E}})$,  Proposition~\ref{proposition:weird} assures us that  $ g(B)\cdot f\in H^{(1)\prime}$.
For $A\in \EE$, $Z_E\in \FF^W_E$ and $Z_{\mathbb{R}\backslash E}\in \FF^W_{\mathbb{R}\backslash E}$, using 
\begin{enumerate}[(I)]
\item\label{descent.I} the fact that, like $f$, so too  $\PP[f\vert \FF_{A}\lor\FF_{\mathrm{stb}}]$ is $\FF_E$-measurable (by \eqref{tensor-product}, $\PP[\cdot\vert \FF_A\lor \FF_{\mathrm{stb}}]=\PP[\cdot \vert\FF_A\lor \FF^{\mathrm{stb}}_{\mathbb{R}\backslash A}]\in \AA$ (even $ \FF_A\lor \FF_{\mathrm{stb}}\in \Cl(\BBB) $ for all $A\in \overline{\EE}$ (see Corollary~\ref{corollary:stb-in-closure}), but we do not need it);  by  \eqref{eq:commuting},  $\PP[\cdot\vert \FF_E]\in \AA$; therefore $\PP[\cdot\vert \FF_A\lor \FF_{\mathrm{stb}}]$ commutes with $\PP[\cdot\vert \FF_E]$),
\item\label{descent.II} the observation that $\FF_E\supset \FF_{E}^\mathrm{stb}$ is independent of $\FF_{\mathbb{R}\backslash E}\supset \FF_{\mathbb{R}\backslash E}^\mathrm{stb}$, and
\item\label{descent.III} the $\FF^W_E$-measurability of $\frac{\dd\mu'_f}{\dd \WW_1}$ due to Proposition~\ref{proposition:rich-measures}, 
\end{enumerate} 
we can then evaluate:\label{looooooong}
\begin{align*}
&\WW\left[\frac{\dd\mu'_{g(B)\cdot f}}{\dd \WW_1}(A);Z_E\cap Z_{\mathbb{R}\backslash E}\right]=\mu_{g(B)\cdot f}'((Z_E\cap Z_{\mathbb{R}\backslash E})\times A)\\
&=\PP\left[\vert\PP[f\vert \FF_A\lor \FF_{\mathrm{stb}}]\vert^2\mathbbm{1}_{ Z_E}(B)\vert g(B)\vert^2 \mathbbm{1}_{Z_{\mathbb{R}\backslash E}}(B) \right]\\
&=\PP\left[\vert\PP[f\vert \FF_A\lor \FF_{\mathrm{stb}}]\vert^2;B\in Z_E\right]\PP\left[\vert g(B)\vert^2 ;
B\in Z_{\mathbb{R}\backslash E} \right]\\
&\qquad\qquad\qquad\qquad\qquad\qquad\qquad \text{($\because$ \ref{descent.I}-\ref{descent.II}, $Z_E\in \FF^W_E$, $Z_{\mathbb{R}\backslash E}\in \FF^W_{\mathbb{R}\backslash E}$, $g\in \LLL^2(\PP\vert_{\FF^W_{\mathbb{R}\backslash E}})$)}\\
&=\mu_f'(Z_E\times A)\WW[\vert g\vert^2;Z_{\mathbb{R}\backslash E}]=\WW\left[\frac{\dd\mu'_{f}}{\dd \WW_1}(A);Z_E\right]\WW\left[\vert g\vert^2;Z_{\mathbb{R}\backslash E}\right]\\
&=\WW\left[\vert g\vert^2\frac{\dd\mu'_{f}}{\dd \WW_1}(A);Z_E\cap Z_{\mathbb{R}\backslash E}\right] \text{($\because$ \ref{descent.III}, $Z_E\in \FF^W_E$, $Z_{\mathbb{R}\backslash E}\in \FF^W_{\mathbb{R}\backslash E}$, $g\in  \LLL^2(\PP\vert_{\FF^W_{\mathbb{R}\backslash E}})$)},
\end{align*}
 which establishes the claim by a twofold application of Dynkin's lemma (since  the countable subclass of $\EE$ consisting of finite unions of intervals with e.g. rational endpoints is closed under intersections and generates $\mathcal{B}_\mathbb{R}$).
\end{proof}

We use up Proposition~\ref{proposition:tensor} in the proof of
\begin{proposition}\label{proposition:rich-measures-bis}
Let $E\in \overline{\EE}$. As $f$ runs over any countable total subset of $H^{(1)\prime}\cap \LLL^2(\PP\vert_{\FF_E})$, $\WW$-a.s. the atoms of the disintegrations $\frac{\dd \mu'_f}{\dd \WW_1}$  precisely exhaust the local maxima of $W$ belonging to $E$.
\end{proposition}
\begin{proof}
By Proposition~\ref{proposition:weird} we have the orthogonal sum decomposition $$H^{(1)\prime}=\underbrace{[( H^{(1)\prime}\cap \LLL^2(\PP\vert_{\FF_E}))\otimes \LLL^2(\PP\vert_{\FF^{\mathrm{stb}}_{\mathbb{R}\backslash E}})]}_{H^{(1)\prime}\cap \LLL^2(\PP\vert_ {\FF_E\lor \FF_{\mathrm{stb}}})}\oplus\underbrace{[\LLL^2(\PP\vert_{\FF^{\mathrm{stb}}_{E}})\otimes ( H^{(1)\prime}\cap \LLL^2(\PP_{\FF_{\mathbb{R}\backslash E}}))]}_{H^{(1)\prime}\cap \LLL^2(\PP\vert_ {\FF_{\mathbb{R}\backslash E}\lor \FF_{\mathrm{stb}}})}.$$ 
Let $M$, $G$, $O$, $K$  be countable total sets in $H^{(1)\prime}\cap \LLL^2(\PP\vert_{\FF_E})$, $\LLL^2(\PP\vert_{\FF^{\mathrm{stb}}_{\mathbb{R}\backslash E}})$, $\LLL^2(\PP\vert_{\FF^{\mathrm{stb}}_{E}})$, $  H^{(1)\prime}\cap \LLL^2(\PP\vert_{\FF_{\mathbb{R}\backslash E}})$ respectively.  The $\WW$-a.s. absolute continuity relations $$\frac{\mu'_{mg+ok}}{\dd\WW_1}\ll \frac{\dd \mu'_{mg}}{\dd \WW_1}+\frac{\dd \mu'_{ok}}{\dd \WW_1},\quad (m,g,o,k)\in M\times G\times O\times K,$$ were noted in Proposition~\ref{proposition:rich-measures-basic} as was the fact that the atoms of the disintegrations $$\frac{\dd \mu'_{mg+ok}}{\dd \WW_1}, \quad (m,g,o,k)\in M\times G\times O\times K,$$ exhaust the local maxima of $W$ a.s.-$\WW$. By Propositions~\ref{proposition:rich-measures}-\ref{proposition:tensor} the  $\frac{\dd \mu'_{ok}}{\dd \WW_1}$, $(o,k)\in O\times K$,  are $\WW$-a.s. carried -- by $\mathbb{R}\backslash E$; the $\frac{\dd \mu'_{mg}}{\dd \WW_1}$, $(m,g)\in M\times G$ -- by the atoms of the $\frac{\dd \mu'_{m}}{\dd \WW_1}$, $m\in M$, said  atoms also belonging to $E$ a.s.-$\WW$. But that means that, $\WW$-a.s., already the atoms of the disintegrations $\frac{\dd \mu'_{m}}{\dd \WW_1}$, $m\in M$, must precisely exhaust the local maxima of $W$ on $E$.
\end{proof}

We are now in a position to conclude the validity of
\begin{proposition}\label{characterization-two}
If $E\in \overline{\EE}$, then $E$ and $\mathbb{R}\backslash E$ are max-enumerable.\qed
\end{proposition}
\begin{proof}
Combine Proposition~\ref{proposition:rich-measures-bis} with the measurability noted in Proposition~\ref{proposition:rich-measures}.
\end{proof} 
Moreover, we have
\begin{theorem}\label{theorem:characterization}
For an $\leb$-measurable $E$ the following are equivalent. 
\begin{enumerate}[(a)]
\item\label{theorem:characterization:i} $E\in \overline{\EE}$. 
\item\label{theorem:characterization:ii}  $E$ and $\mathbb{R}\backslash E$ are both max-enumerable. 
\item\label{theorem:characterization:iii} $\{s\cap E,s\backslash E\}\subset S$ for $\mu$-a.e. $s$.
\end{enumerate}
When so, then $\FF_E$ is generated by  $\FF_E^{\mathrm{stb}}$ and by the random signs $\epsilon_{T_n}$, $n\in \mathbb{N}$, for an(y) $\FF^W_E$-measurable enumeration $T$ of the local maxima of $W$ belonging to $E$. 

Furthermore, there is sequential monotonic continuity of $\overline{\FF}$:
for every sequence $(A_n)_{n\in \mathbb{N}}$  in $\overline{\EE}$ that is $\downarrow$ (resp. $\uparrow$) to an $A\in \overline{\EE}$, one has that $\land_{n\in \mathbb{N}}\FF_{A_n}= \FF_A$  (resp. $\lor_{n\in \mathbb{N}}\FF_{A_n}= \FF_A$).
\end{theorem}
\begin{proof}
The equivalence \ref{theorem:characterization:i}$\Leftrightarrow$\ref{theorem:characterization:ii} is the subject matter of  Propositions~\ref{characterization-one} and~\ref{characterization-two}. 
By Proposition~\ref{proposition:spectral-projection-for-extension},  \ref{theorem:characterization:iii} is necessary for  \ref{theorem:characterization:i}. Now assume  \ref{theorem:characterization:iii}. The condition, together with $K'<\infty$ a.e.-$\mu$, entails that $\mu$-a.e. $$S=\cup_{U\in \mathcal{U}}\{\pr_U\in S_E,\pr_{\mathbb{R}\backslash U}\in S_{\mathbb{R}\backslash E}\},$$ where $\mathcal{U}$ is the collection of finite unions of intervals with rational endpoints. 
As a consequence, by \eqref{tensor-product}, the collection $$\cup_{U\in \mathcal{U}}H(S_E\cap S_U)\otimes H(S_{\mathbb{R}\backslash E}\cap S_{\mathbb{R}\backslash U})$$ (the term corresponding to $U\in \mathcal{U}$ being viewed up to the natural unitary equivalence of $\LLL^2(\PP)$ and $\LLL^2(\PP\vert_{\FF_U})\otimes \LLL^2(\PP\vert_{\FF_{\mathbb{R}\backslash U}})$) is total in $\LLL^2(\PP)$, a fortiori $$\{fg:(f,g)\in \LLL^2(\PP\vert_{\FF_E})\times \LLL^2(\PP\vert_{\FF_{\mathbb{R}\backslash E}})\}$$ is total in $\LLL^2(\PP)$ rendering $\FF_E\lor \FF_{\mathbb{R}\backslash E}=1_\PP$ and therefore, combined with Proposition~\ref{proposition:meet-and-join}\ref{indep} and \eqref{eq:in-closure}, $E\in \overline{\EE}$, which is \ref{theorem:characterization:i}.

Now suppose $E$ and $\mathbb{R}\backslash E$ are both max-enumerable. Of course, $\FF_E$ and $\FF_{\mathbb{R}\backslash E}$ are independent. Put $\GG_E:= \FF_E^{\mathrm{stb}}\lor \sigma(\epsilon_{T_n}:n\in \mathbb{N})$, $T$ an enumeration of the local maxima of $W$ on $E$, and likewise introduce $\GG_{\mathbb{R}\backslash E}$.  By Proposition~\ref{lemma:locmaxstable} certainly $\GG_E\subset \FF_E$ and $\GG_{\mathbb{R}\backslash E}\subset \FF_{\mathbb{R}\backslash E}$. Intersecting the evident equality $\GG_E\lor \GG_{\mathbb{R}\backslash E}=1_\PP$ with $\FF_E$ we get by general distributivity over independent $\sigma$-fields \eqref{distributivity-indep-not-complete}  that $\FF_E=\GG_E$.

Concerning finally the sequential monotonic continuity of $\overline{\FF}$, by Proposition~\ref{proposition:meet-and-join}\ref{cap}, $\overline{\FF}$ respects $\downarrow$ intersections.  Let us handle the respective $\uparrow$ case. It is clear that $\lor_{n\in \mathbb{N}}\FF_{A_n}\supset \lor_{n\in \mathbb{N}}\FF_{A_n}^{\mathrm{stb}}=\FF^{\mathrm{stb}}_A$. Let, for each $n\in \mathbb{N}$,  $T^n$ be an $\FF_{A_n}^{W}$-measurable enumeration of the local maxima of $W$ on $A_n$. Then $\lor_{n\in \mathbb{N}}\FF_{A_n}\supset \sigma(\epsilon_{T^n_k}:(k,n)\in \mathbb{N}^2)$ as well. If now $T$ is any $\FF^W_A$-measurable selection of a local maximum of $W$ on $A$, then $\{T\ne \dagger\}=\cup_{(n,k)\in \mathbb{N}^2} \{T=T^n_k\ne \dagger\}\in \FF^W_A$ and hence $\epsilon_T$ is $\lor_{n\in \mathbb{N}}\FF_{A_n}$-measurable. We conclude that $\lor_{n\in \mathbb{N}}\FF_{A_n}=\FF_A$. 
\end{proof}
Notice that the continuity of $\overline\FF$ noted in Theorem~\ref{theorem:characterization} and its nonclassicality imply automatically \cite[Theorem~1.5]{tsirelson} that $\BB_\mathbb{R}\not\subset \overline{\EE}$. 

In closing this section let us remark  that it is perhaps not at all surpring that the extension $\overline{\FF}$ should be precisely to max-enumerable sets with max-enumerable complement, since these constitute the very class of domains on  which, and on the complements of which,  the ``locality'' property \cite[Definition~3.2(b)]{stationary} of the local maxima considered as a random countable set over the Wiener noise persists. With the benefit of hindsight it seems to be indeed exactly what one should have expected from a ``faithful'' full extension of the splitting noise. In turn, it lends credence to treating $\overline{\BBB}$ as the ``correct'' largest extension of $\BBB$.

\section{Stability of the local maxima}\label{section:stability-sensitivity}

If a regular open subset $E$  of $\mathbb{R}$ is such that $\leb(\partial E)=0$, then \cite[Proposition~3b9]{tsirelson-arxiv-2} (coupled with the fact that $\leb$-negligible sets are negligible for $\FF$, as we have seen) entails $E\in \overline{\EE}$. More generally,  by Proposition~\ref{proposition:meet-and-join}\ref{cup}-\ref{indep}, if an $\leb$-measurable $E$ is such that both $E$ and $\mathbb{R}\backslash E$ are equal  mod-$\leb$ each  to their own open subset of $\mathbb{R}$, then $E\in \overline{\mathcal{E}}$. This already demonstrates $\mathcal{E}\ne \overline{\mathcal{E}}$ albeit in a rather trivial way.
  
We wish to understand more about the kind of elements that $\overline{\EE}$ contains/does not contain and (thus) more about the nature of max-enumerable sets. The chief vehicle  which we adopt to explore this is that of coupling.

 Throughout this section let then
 \begin{quote}\normalsize
  $\mathbb{Q}$ be a complete extension of the probability $\WW$ 
  \end{quote}
  supporting another  
 \begin{quote}
 \normalsize $\Omega_0$-valued Brownian motion $W'$ independent of $W$
 \end{quote}
  (an independent copy of $W$). 

 \begin{definition}\label{def:pertrubed}
 For $\leb$-measurable $E$ put $$W_E:=\mathbbm{1}_E\cdot W+\mathbbm{1}_{\mathbb{R}\backslash E}\cdot W',$$ which we may and do insist is $\Omega_0$-valued;  more generally, for a random element $g$  of $\WW$, $$g_E:=g(W_E).$$ 
 \end{definition}
 
We may think of the  $W_E$ of   Definition~\ref{def:pertrubed} as $W$ independently resampled off $E$. A different, but equivalent, description is that $(W,W_E)$ is a bivariate $\Omega_0\times\Omega_0$-valued Brownian motion such that $\dd \langle W,W_E\rangle= \mathbbm{1}_E\cdot \leb$. In particular, 
\begin{equation}\label{eq:symmetric-coupling}
(W,W_E)_\star \QQ=(W_E,W)_\star \QQ,
\end{equation}
i.e. we have in $(W,W_E)$ under $\QQ$ a symmetric self-coupling of $W$. 

The ensuing Proposition~\ref{proposition:max-enumerable-necessary} demonstrates that the coupling $(W,W_E)$ is intimately related to $\FF_E$ on the one hand, and to   how stable/sensitive a selection $T$ of local maximum of $W$ is to the perturbation encoded by $(W,W_E)$, viz. to how $T$ compares with $T_E$,  on the other. Two extremes of this stability/sensitivity and their connection to $\overline{\EE}$  will be explored in Subsections~\ref{subsection:unstable}-\ref{subsection:stable}.

 \begin{proposition}\label{proposition:max-enumerable-necessary}
Let $E$ be an $\leb$-measurable set, $P$ a finite partition of $\mathbb{R}$ into elementary sets (modulo a finite set) and for each $p\in P$: $g^p$  an element of  $\LLL^2(\WW\vert_{\FF_p^W})$; $T^p$ an $\FF^W_p$-measurable selection of a local maximum of $W$ on $p$. 
Then 
\begin{equation}\label{probability-formula}
\PP \left[\left\vert\PP \left[\prod_{p\in P } g^p(B)\epsilon_{T^p}\vert \FF_E \right]\right\vert^2\right]=\prod_{p\in P}\QQ[\overline{g^p}g^p_E; T^p=T^p_E\in E].
\end{equation}
 \end{proposition}
 It may be helpful for the reader to restrict to the case $P=\{\mathbb{R}\}$ when initially contemplating Proposition~\ref{proposition:max-enumerable-necessary}. The extension to an arbitrary finite partition into elementary sets is largely a technicality (but a relevant, cf. Lemma~\ref{lemma-key-totality}, and non-trivial technicality).
 
 For the purposes of establishing Proposition~\ref{proposition:max-enumerable-necessary} an ancillary technical result will come in handy, that is to do with a certain kind of continuity of measurable selections of local maxima. In it, we flesh out more or less exactly what  will ultimately be used up: we have not attempted to optimize on the $(\star)$-continuity featuring therein.

\begin{lemma}\label{lemma:star-ct}
There is a $\WW$-a.s. event $\Omega_0'$ for which the following holds. For all $A\in \EE$, for any  $\FF^W_A$-measurable selection $T$ of a local maximum of $W$ on $A$, for each $\delta\in (0,\infty)$, there exists an $\FF_A^W$-measurable selection $T^\delta$ of a local maximum of $W$ on $A$  having the properties: $T^\delta=T$ on $\{T^\delta\ne\dagger \}$;  $\WW(T^\delta\ne T\ne \dagger)<\delta$; and 
\begin{quote}
\normalsize [to be referred to as ($\star$)-continuity of $T^\delta$ on $\Omega_0'$] for every sequence $(\omega_n)_{n\in \mathbb{N}}$ in $\Omega_0$ converging to an $\omega_0\in\Omega_0'$ in the locally uniform topology, if $T^\delta(\omega_n)$ is equal to some constant from $\mathbb{R}$ for infinitely many $n\in \mathbb{N}$, then $T^\delta(\omega_0)$ is equal to this constant. 
\end{quote}
\end{lemma}
The idea of the proof is quite simple: speaking somewhat loosely, the $(\star)$-continuity holds  on a $\WW$-a.s. event simultaneously for all maximizers of $W$ on intervals with dyadic endpoints, so one has only to see $T$ as  a suitable combination thereof, except maybe on an event of small probability. The formalities are somewhat unpleasant, but quite manageable. The case of a general $A\in \EE$ rather than $A=\mathbb{R}$ adds only superficial nuissance (but we go ahead with treating the general $A$ anyway).
\begin{proof}[Proof of Lemma~\ref{lemma:star-ct}]
For an open  interval $I$ with dyadic endpoints (henceforth, a dyadic interval) denote  by $S^I$ the maximizer of $W$ on $I$, setting it equal to $\dagger$ on the $\WW$-exceptional set on which it fails to exist uniquely. Fix a $\WW$-a.s. event $\Omega_0'$ on which $S^I$ exhibits $(\star)$-continuity for all dyadic intervals $I$ (e.g. the event on which $W$ attains its maximum uniquely on each dyadic interval will do). Also, let $\mathcal{I}$ be the open dyadic intervals of $\mathbb{R}$ of length $2^{-N}$ that are contained in $A$, $N\in \mathbb{N}$ having been chosen (and fixed) so large as to ensure that 
\begin{equation}\label{fuck1}
\WW(\{T\ne \dagger\}\backslash \cup_{I\in \mathcal{I}}\{T=S^I\ne\dagger\})< \delta/2.
\end{equation}

Next, for technical reasons that will become clear shortly, we specify a canonical encoding of the increments of $W$ on $A$. Since $A$ is elementary, it is the finite disjoint union of its connected components, which are intervals of the real line that we may and do assume are non-degenerate, closed and sharing no endpoints (it is without loss of generality $\because$ changing $A$ by a finite set leaves $\FF^W_A$ invariant). Let $\JJ$ be the collection of these connected components. For $J\in\JJ$ and $\omega\in\Omega_0$: 
\begin{enumerate}[(a)]
\item\label{cases:a} if $J$ is bounded from below with left endpoint $l_J$ put $\Phi_J(\omega):=\omega\vert_J-\omega(l_J)$; 
\item\label{cases:b} else, if $J$ is bounded from above with right endpoint $r_J$ put $\Phi_J(\omega):=\omega\vert_J-\omega(r_J)$;
\item\label{cases:c} otherwise ($J=\mathbb{R}$ and) set $\Phi_J(\omega):=\omega$.
\end{enumerate}
Thus we have introduced the continuous map $\Phi_A:=(\Phi_J)_{J\in \JJ}:\Omega_0\to \times_{J\in\JJ}C_0(J,\mathbb{R})$, where for $J\in \JJ$, $C_0(J,\mathbb{R})$ is the set of all continuous maps from $J$ to $\mathbb{R}$ vanishing at $l_J$, $r_J$ or $0$ according as to whether \ref{cases:a}, \ref{cases:b} or \ref{cases:c} above occurs, endowed with the locally uniform topology. 

Proceeding onwards, since $\mathcal{I}$ is countable, there is a map $\tilde\delta:\mathcal{I}\to (0,1]$ such that $\sum_{I\in \mathcal{I}}\tilde\delta_I\leq \delta/2$; pick some,  any. Now, for $I\in \mathcal{I}$, we have that $I\subset A$, hence $S^I$ is $\FF^W_A$-measurable, whilst $T$ is $\FF^W_A$-measurable by assumption. Therefore $\{S^I=T\ne\dagger\}\in \FF^W_A$. It means that there is a measurable subset $\tilde C^I$ of $\times_{J\in\JJ}C_0(J,\mathbb{R})$ such that $\{S^I=T\ne\dagger\}\supset \Phi_A^{-1}(\tilde C^I)$ and at the same time $\WW$-a.s.\footnote{We cannot ask for equality with certainty because $\FF_A^W$ contains also $0_\WW$, not just the $\sigma$-field of the increments of $W$ on $A$.} $\{S^I=T\ne\dagger\}= \Phi_A^{-1}(\tilde C^I)$. But, like any probability measure on a metric space \cite[Theorem~1.1]{bill}, $\WW_A:=(\Phi_A)_\star \WW$ is  inner regular w.r.t. the closed sets. So, there is a closed $C^{I\prime}\subset \times_{J\in\JJ}C_0(J,\mathbb{R})$ for which $C^{I\prime}\subset \tilde C_I$ and $\WW_A(\tilde C_I\backslash C^{I\prime})\leq \tilde\delta_I$. Pulling it back via $\Phi_A$ we get the closed $C^I:=\Phi_A^{-1}(C^{I\prime})\subset \Omega_0$ satisfying 
\begin{equation}\label{fuck0}
C^I\subset \{S^I=T\ne\dagger\}
\end{equation}
and $\WW(\{S^I=T\ne\dagger\}\backslash C^I)\leq \tilde \delta_I$. By countable subadditivity it follows that  
\begin{equation}\label{fuck2}
\WW(\cup_{I\in \mathcal{I}}\{T=S^I\ne\dagger\}\backslash C^I)\leq \delta/2.
\end{equation}
Notice also (from \eqref{fuck0}) that the $C^I$, $I\in \mathcal{I}$, are pairwise disjoint ($\because$ the elements of $\mathcal{I}$ are pairwise disjoint, $S^I\in I$ on $\{S^I\ne \dagger\}$ for all $I\in \mathcal{I}$, and surely $T$ cannot at the same sample point take values in disjoint intervals). 

In order to conclude the proof, it remains to set $T^\delta:=S^I$ on $C^I$ for  $I\in \mathcal{I}$ and $T^\delta:=\dagger$ on $\Omega_0\backslash \cup_{I\in \mathcal{I}}C^I$.  Indeed, by \eqref{fuck0} evidently $T=T^\delta$ on $\cup_{I\in \mathcal{I}}C^I=\{T^\delta\ne \dagger\}$. Also,
$$\{T^\delta\ne T\ne \dagger\}\subset \left(\{T\ne \dagger\}\backslash \cup_{I\in \mathcal{I}}\{T=S^I\ne \dagger\}\right)\cup \left(\cup_{I\in \mathcal{I}}(\{T=S^I\ne\dagger\}\backslash C^I)\right);$$ therefore from \eqref{fuck1} \& \eqref{fuck2} we get $\WW(T^\delta\ne T\ne \dagger)<\delta$.  Finally, the ($\star$)-continuity of $T^\delta$ on $\Omega_0'$ is a consequence of the observation that  $T^\delta\in I$ on the closed set $C^I$ for all $I\in \mathcal{I}$, the members of $\mathcal{I}$ being pairwise disjoint: so, if $T^\delta(\omega_n)$ is equal to some constant $c$ from $\mathbb{R}$ for infinitely many $n\in \mathbb{N}$, then there is (a unique)  $I\in \mathcal{I}$ such that ($c\in I$ and hence) for these $n\in \mathbb{N}$, $\omega_n\in C^I$ (in particular, $S^I(\omega_n)=T^\delta(\omega_n)=c$), rendering $\omega_0=\lim_{n\to\infty}\omega_n\in C^I$ and thus $T^\delta(\omega_0)=S^I(\omega_0)=c$ (by the $(\star)$-continuity of $S^I$ on $\Omega_0'$).
\end{proof}
 \begin{proof}[Proof of Proposition~\ref{proposition:max-enumerable-necessary}]
 The case when $\leb(E)=0$ is trivial; assume $\leb(E)>0$.
 
 For now suppose $E$ is open. By Propositions~\ref{proposition:meet-and-join}\ref{cup}-\ref{indep} and~\ref{lemma:locmaxstable} we may evaluate the l.h.s. of \eqref{probability-formula} as
 \begin{align*}
& \PP \left[\left\vert\PP \left[\prod_{p\in P } g^p(B)\epsilon_{T^p}\vert \FF_E \right]\right\vert^2\right]= \PP \left[\left\vert\PP \left[\prod_{p\in P } g^p(B)\epsilon_{T^p}\vert \lor_{p\in P}\FF_{E\cap p} \right]\right\vert^2\right]\\
 &=\PP \left[\prod_{p\in P }\left\vert\PP \left[ g^p(B)\epsilon_{T^p}\vert \FF_{E\cap p} \right]\right\vert^2\right]=\prod_{p\in P }\PP \left[\left\vert\PP \left[ g^p(B)\epsilon_{T^p}\vert \FF_{E\cap p} \right]\right\vert^2\right].
 \end{align*} As for the r.h.s. of \eqref{probability-formula}, for any elementary  $A$ and for any $\FF^W_A$-measurable random element $r$, $\QQ$-a.s. $r_E=r(W_E)=r(W_{E\cap A})=r_{E\cap A}$; therefore 
 $$\prod_{p\in P}\QQ[\overline{g^p}g^p_E; T^p=T^p_E\in E]=\prod_{p\in P}\QQ[\overline{g^p}g^p_{E\cap p}; T^p=T^p_{E\cap p}\in E\cap p].$$
 We conclude that in proving \eqref{probability-formula} for an open $E$ we may (and shall) just as well assume that $P=\{\mathbb{R}\}$, the trivial partition. Accordingly, we drop the superscript ${}^\mathbb{R}$ to ease the notation.  
 
Let now $(T_n)_{n\in \mathbb{N}}$ be an a.s.-$\WW$ real-valued $\FF_E^W$-measurable enumeration of the local maxima of $W$ on $E$ such that $T_i\ne T_j$ a.s.-$\WW$ on $\{T_i\ne \dagger,T_j\ne \dagger\}$ for all $i\ne j$ from $\mathbb{N}$, and let $(T'_n)_{n\in \mathbb{N}}$ be an enumeration of the local maxima of $W$ on $\mathbb{R}\backslash E$ such that $T_i'\ne T_j'$ a.s.-$\WW$ on $\{T_i'\ne \dagger,T_j'\ne \dagger\}$ for all $i\ne j$ from $\mathbb{N}$ (they exist).

By Corollary~\ref{corollary:F-for-open}, conditionally on $\FF_{\mathrm{stb}}$, the random signs $(\epsilon_{T_n'})_{n\in \mathbb{N}}$ are mean zero and independent of $\FF_E$. Hence, as the $T_n'$, $n\in \mathbb{N}$, are $\WW$-a.s. pairwise distinct if real and since they exhaust the local maxima of $W$ off $E$ (justifying the first of the equalities to follow)\phantomsection\label{long-display-one}
\begin{align*}
\PP[g(B)\epsilon_T\mathbbm{1}_{\{T(B)\notin E\}}\vert \FF_E]&=\PP\left[\sum_{n\in \mathbb{N}}g(B)\epsilon_{T_n'} \mathbbm{1}_{\{T(B)=T'_n(B)\ne\dagger\}}\vert\FF_E\right]\\
&=\PP\left[\sum_{n\in \mathbb{N}}\PP[\epsilon_{T_n'}\vert \FF_E\lor  \FF_{\mathrm{stb}}] g(B)\mathbbm{1}_{\{T(B)=T'_n(B)\ne\dagger\}}\vert\FF_E\right] \\
&=\PP\left[\sum_{n\in \mathbb{N}}\PP[\epsilon_{T_n'}\vert \FF_{\mathrm{stb}}] g(B)\mathbbm{1}_{\{T(B)=T'_n(B)\ne\dagger\}}\vert\FF_E\right]=0\text{ a.s.-$\PP$}.
\end{align*}
The random signs $(\epsilon_{T_n})_{n\in \mathbb{N}}$ too are mean zero, they are mutually independent, independent of $\FF_{\mathrm{stb}}$, and they are $\FF_E$-measurable (the latter by Corollary~\ref{corollary:F-for-open}). Also, $(g,T)$ and $(g_E,T_E)$ are independent and equally distributed given $\FF_E^W$. Thus, because the $T_n$, $n\in \mathbb{N}$, are $\WW$-a.s. pairwise distinct and as they exhaust the local maxima of $W$ on $E$ (justifying the first and penultimate of the equalities to follow)\phantomsection\label{long-display-two}
\begin{align*}
&\PP [\vert\PP[g(B)\epsilon_T\mathbbm{1}_{\{T(B)\in E\}}\vert \FF_E]\vert^2]=\PP\left[\left\vert\PP\left[\sum_{n\in \mathbb{N}}g(B)\epsilon_{T_n} \mathbbm{1}_{\{T(B)=T_n(B)\}}\vert\FF_E\right]\right\vert^2\right]\nonumber\\ 
&=\PP\left[\left\vert\sum_{n\in \mathbb{N}}\epsilon_{T_n} \PP[g(B)\mathbbm{1}_{\{T(B)=T_n(B)\}}\vert \FF_E]\right\vert^2\right]\\
&=\PP\left[\left\vert\sum_{n\in \mathbb{N}}\epsilon_{T_n} \PP[\PP[g(B)\mathbbm{1}_{\{T(B)=T_n(B)\}}\vert\FF_{\mathrm{stb}}]\vert \FF_E]\right\vert^2\right]\\
&=\PP\left[\left\vert\sum_{n\in \mathbb{N}}\epsilon_{T_n} \PP[g(B)\mathbbm{1}_{\{T(B)=T_n(B)\}}\vert  \FF_E^{\mathrm{stb}}]\right\vert^2\right]\, (\because \,\FF_E\text{ and  }\FF_{\mathrm{stb}}\text{ are commuting, in fact} \\ 
&\qquad \qquad \qquad\{\PP[\cdot\vert\FF_E],\PP[\cdot\vert\FF_{\mathrm{stb}}]\}\subset \AA\text{ (see \eqref{eq:commuting} \& \eqref{eq:stable-in-vNa}), while }\FF_E\land \FF_{\mathrm{stb}}=\FF_E^{\mathrm{stb}} \\
&\qquad \qquad \qquad  \text{ by Proposition~\ref{proposition:extensions-mod}\ref{proposition:extensions-mod:i}})\\\nonumber 
&=\PP\left[\sum_{n\in \mathbb{N}}\left\vert\PP[g(B)\mathbbm{1}_{\{T(B)=T_n(B)\}}\vert \FF_E^{\mathrm{stb}}]\right\vert ^2\right] =\WW\left[\sum_{n\in \mathbb{N}}\left\vert\WW[g\mathbbm{1}_{\{T=T_n \}}\vert \FF_E^W]\right\vert^2\right]\\
&=\QQ\left[\sum_{n\in \mathbb{N}}\QQ[\overline{g}\mathbbm{1}_{\{T=T_n\}}g_E\mathbbm{1}_{\{T_E=T_n\}}\vert \FF_E^W]\right]=\QQ[\QQ[\overline{g}g_E;T=T_E\in E\vert\FF^W_E]]\\
&=\QQ[\overline{g}g_E;T=T_E\in E],
\end{align*}
which completes the proof in case $E$ is open. 

Now consider a general $E$, the partition $P$ being again arbitrary. By approximation from the outside to within a set of $\leb$-measure zero it suffices to consider an $E$ that is a $G_\delta$ set so that $E=\cap_{n\in \mathbb{N}}E_n$ for a $\downarrow$ sequence $(E_n)_{n\in \mathbb{N}}$ of open subsets of $\mathbb{R}$.  By \ref{loc-unif-a.s.} from p.\pageref{loc-unif-a.s.}, passing to a subsequence if necessary, we may and do assume that $W_{E_n}\to W_E$ locally uniformly as $n\to\infty$ on an event $\Omega'$ of $\QQ$-probability one.

By what we have already established, for each $n\in \mathbb{N}$, the stipulated equality \eqref{probability-formula} holds with $E_n$ in lieu of $E$. Passing to the limit $n\to\infty$, the l.h.s. converges  by decreasing martingale convergence. On the r.h.s., for the purposes of passing to the limit in each factor separately, we may just as well assume $P=\{\mathbb{R}\}$, drop the superscript $^\mathbb{R}$ and seek to establish 
\begin{equation} \label{eq:goal}
\lim_{n\to\infty}\QQ[\overline{g}g_{E_n};T=T_{E_n}\in E_n]=\QQ[\overline{g}g_{E};T=T_{E}\in E],
\end{equation}
whereby the proof will be completed.

Now, in order to check \eqref{eq:goal}, assume $g$ is continuous, bounded and nonnegative in the first instance. 

Holding such $g$, but also $E$ and the $E_n$, $n\in \mathbb{N}$, fixed, suppose for the time being [meaning: until the end of this paragraph] that \eqref{eq:goal} has been established whenever $T$ has the $(\star)$-continuity property  on the $\WW$-a.s. set $\Omega_0'$ of Lemma~\ref{lemma:star-ct}. With $T$ arbitrary, let $\delta\in (0,\infty)$ also be arbitrary. By Lemma~\ref{lemma:star-ct} we get a selection $T^\delta$ of a local maximum of $W$ having the properties: $T^\delta=T$ on $\{T^\delta\ne\dagger \}$;  $\WW(T^\delta\ne T\ne \dagger)<\delta$; the $(\star)$-continuity on $\Omega_0'$. Now, 
\begin{align*}
&\vert \QQ[\overline{g}g_{E};T=T_{E}\in E]-\QQ[\overline{g}g_{E};T^\delta={T^\delta}_{E}\in E]\vert\\
&\leq \Vert g\Vert_\infty^2\QQ(\{E\ni T\ne T^\delta\}\cup \{E \ni T_E\ne {T^\delta}_E\}\cup \{T\ne T^\delta\in E\}\cup \{ T_E\ne {T^\delta}_E\in E\})\\
&\leq \Vert g\Vert_\infty^2\left(\QQ(\dagger\ne T\ne T^\delta)+\QQ(\dagger \ne T_E\ne {T^\delta}_E)+\QQ(T\ne T^\delta\ne \dagger)+\QQ(T_E\ne {T^\delta}_E\ne \dagger)\right)\\
&=2\Vert g\Vert_\infty^2\left(\WW(\dagger\ne T\ne T^\delta)+\WW(T\ne T^\delta\ne \dagger)\right)\quad (\because\, {W_E}_\star \QQ=\WW=W_\star\QQ)\\
&=2\Vert g\Vert_\infty^2\WW(\dagger\ne T\ne T^\delta)\quad (\because\ T^\delta=T\text{ on }\{T^\delta\ne \dagger\})\\
&\leq 2\Vert g\Vert_\infty^2\delta\quad (\because\, \WW(T^\delta\ne T\ne \dagger)<\delta).
\end{align*}
By the same token, $$\vert \QQ[\overline{g}g_{{E_n}};T=T_{{E_n}}\in {E_n}]-\QQ[\overline{g}g_{{E_n}};T^\delta={T^\delta}_{{E_n}}\in {E_n}]\vert\leq  2\Vert g\Vert_\infty^2\delta,\quad n\in \mathbb{N}.$$ Since, by the pro tempore assumption, \eqref{eq:goal} holds with $T^\delta$ in lieu of $T$, there is $N\in \mathbb{N}$ such that $$\vert \QQ[\overline{g}g_{{E_n}};T^\delta={T^\delta}_{{E_n}}\in {E_n}]-\QQ[\overline{g}g_{E};T^\delta={T^\delta}_{E}\in E]\vert\leq \Vert g\Vert_\infty^2\delta,\quad n\in \mathbb{N}_{\geq N};$$ whence, altogether,
$$\vert \QQ[\overline{g}g_{{E_n}};T=T_{{E_n}}\in {E_n}]-\QQ[\overline{g}g_{E};T=T_{E}\in E]\vert\leq 5\Vert g\Vert_\infty^2\delta,\quad n\in \mathbb{N}_{\geq N}.$$ As $\delta$ was arbitrary, we deduce that \eqref{eq:goal} holds also for the arbitrary $T$. 

As a consequence of the finding of the preceding paragraph, when proving \eqref{eq:goal} for $g$ continuous, bounded and nonnegative, to which we now return, we may and shall assume without loss of generality that $T$ has the $(\star)$-continuity property on the  $\WW$-a.s. set $\Omega_0'$ of Lemma~\ref{lemma:star-ct}. Then, on the one hand, by (reverse) Fatou's lemma, and by this very property, noting that  $ W_E^{-1}(\Omega_0')\cap \Omega'$ has $\QQ$-probability one, $$\limsup_{n\to\infty}\QQ[gg_{E_n};T=T_{E_n}\in E_n]\leq \QQ[\limsup_{n\to\infty}gg_{E_n}\mathbbm{1}_{\{T=T_{E_n}\in E_n\}}]\leq \QQ[gg_{E};T=T_E\in E].$$ On the other hand, by Lemma~\ref{lemma:elementary} for the second inequality, recalling that $(g,T)$ and $(g_E,T_E)$  are independent and equally distributed given $\FF_E^W$, the same with $E_n$ in lieu of $E$ ($n\in \mathbb{N}$),
\begin{align*}
&\liminf_{n\to\infty}\QQ[gg_{E_n};T=T_{E_n}\in E_n]\geq \liminf_{n\to\infty}\QQ[gg_{E_n};T=T_{E_n}\in E]\\
&=\liminf_{n\to\infty}\QQ[g\mathbbm{1}_{\{T\in E\}}g_{E_n}\mathbbm{1}_{\{T_{E_n}\in E\}};T=T_{E_n}]\\
&\geq \QQ[g\mathbbm{1}_{\{T\in E\}}g_{E}\mathbbm{1}_{\{T_{E}\in E\}};T=T_{E}]= \QQ[gg_{E};T=T_E\in E].
\end{align*}
This concludes the verification of \eqref{eq:goal} in the case when $g:\Omega_0\to [0,\infty)$ is continuous and bounded. 

Take next arbitrary $g$ and $h$, continuous, bounded maps from $\Omega_0$ to $[0,\infty)$. Applying \eqref{eq:goal} consecutively to $g$, $h$ and $g+h$ in lieu of $g$, it follows from linearity and from \eqref{eq:symmetric-coupling} that 
\begin{equation*} 
\lim_{n\to\infty}\QQ[\overline{h}g_{E_n};T=T_{E_n}\in E_n]=\QQ[\overline{h}g_{E};T=T_{E}\in E].
\end{equation*}
By linearity again it is now elementary to extend the preceding equality separately in $g$ and in $h$ to any continuous complex-valued bounded maps on $\Omega_0$. In particular, setting $h=g$, we deduce \eqref{eq:goal} for the case when $g$ is bounded and continuous. The class of such $g$ is dense in $\LLL^2(\WW)$ by functional monotone class;  a straightforward approximation (using e.g. Cauchy-Schwartz to effect the relevant estimates) allows finally to conclude that   \eqref{eq:goal} holds with no restriction on $g$ (beyond it belonging to $\LLL^2(\WW)$).  The proof is now complete. 
 \end{proof}

\subsection{Max-totally unstable sets}\label{subsection:unstable}
Here we consider sets whose local maxima are as sensitive to resampling as can be, in the precise sense of

 \begin{theorem}\label{theorem:max-unstable}
  For an $\leb$-measurable $E$ the following are equiveridical. 
  \begin{enumerate}[(a)]
    \item\label{unstable:iii} The local maxima of $W$ belonging to $E$ are disjoint with the local maxima of $W_E$ belonging to $E$ a.s.-$\QQ$, but the local maxima of $W$ belonging to $E$ are not a.s.-$\QQ$ the empty set.
        \item\label{unstable:v}  The local maxima of $W$ belonging to $E$ are disjoint with the local maxima of $\mathbbm{1}_E\cdot W$ a.s.-$\WW$,\footnote{As was explained in \ref{maxima-of-censored} on p.~\pageref{maxima-of-censored'}  the local maxima of $\mathbbm{1}_E\cdot W$ belong $\WW$-a.s. to $E$, so the qualification ``belonging to $E$'' can be dropped.} but the local maxima of $W$ belonging to $E$ are not a.s.-$\WW$ the empty set.
\item\label{unstable:iv}  $\QQ(\tau_E=\tau\in E)=0$ for every selection $\tau$ of a local maximum of $W$, but $\WW(\tau\in E)>0$ for some such selection.
  \item\label{unstable:i} $\QQ(\tau_E=\tau\in E)=0$ for all maximizers  $\tau$ of $W$, but $\WW(\tau\in E)>0$ for some such maximizer.
     \item\label{unstable:new} There is no $\FF^W_E$-measurable selection $T$ of a local maximum of $W$ on $E$ for which $\WW(T\ne \dagger)>0$, but $W$ has local maxima belonging to $E$ with positive $\WW$-probability.
    \item\label{unstable:ii} $0_\PP\ne \FF_E\subset \FF_{\mathrm{stb}}$ (i.e. $S_E\subset \{K<\infty\}$ a.e.-$\mu$ but $S_E= \{\emptyset\}$ a.e.-$\mu$ fails). 
 \end{enumerate}
  When so, then $E\notin\overline{\EE}$.
 \end{theorem}
 \begin{definition}\label{def-unstb}
 An $\leb$-measurable $E$ is said to be max-totally unstable if it meets one and then all of the equivalent conditions \ref{unstable:iii}-\ref{unstable:i}  of Theorem~\ref{theorem:max-unstable}.
 \end{definition}
 Of course the complement of a max-totally unstable set must be dense, in particular a closed max-totally unstable set must be nowhere dense. 
 
  In connecting \ref{unstable:ii} to \ref{unstable:iii}-\ref{unstable:i} (the equivalence of the latter shall follow by relatively elementary arguments; separately we will argue  \ref{unstable:new}$\Leftrightarrow$\ref{unstable:ii}), Proposition~\ref{proposition:max-enumerable-necessary} will be instrumental. The significance of having stated it for an arbitrary finite partition $P$ of $\mathbb{R}$ into elementary sets, rather than just $P=\{\mathbb{R}\}$, stems from the observation of
 
 \begin{lemma}\label{lemma-key-totality}
 The collection of random variables of the form $\prod_{p\in P } g^p(B)\epsilon_{T^p}$ --- where for each $p\in P$, $g^p$ is  an element of  $\LLL^2(\WW\vert_{\FF_p^W})$ and $T^p$ is an $\FF^W_p$-measurable selection of a local maximum of $W$ on $p$ --- as $P$ ranges over all finite partitions  of $\mathbb{R}$ into elementary sets, is total in the sensitive subspace $H_{\mathrm{sens}}$.
 \end{lemma}
 \begin{proof}
Denote by $(2^\mathbb{N})_{\mathrm{fin}}$  the collection of the finite subsets of $\mathbb{N}$.	Let $G$ be the closure of the linear span of the indicated random variables. By approximation within $\LLL^2(\WW)$, with $P$ fixed, we see immediately that $G$ contains $g(B)\prod_{p\in P}\epsilon_{T^p}$ for all $g\in \LLL^2(\WW)$, any choice of $T^p$ an $\FF^W_p$-measurable selection of a local maximum of $W$ on $p$ for $p\in P$, all non-empty finite collections  $P$ of pairwise disjoint elementary sets.  Fix a sequence $(P_n)_{n\in \mathbb{N}}$ of ever finer finite partitions of $\mathbb{R}$ into intervals, whose union separates the points of $\mathbb{R}$. For $p\in \cup_{n\in \mathbb{N}}P_n$ denote by $\tau_p$ the maximizer of $W$ on $p$, setting it equal to $\dagger$ on the event on which it fails to exist uniquely. Recalling  the enumeration $\mathsf{S}$ from p.~\pageref{enumeration-page},  it is clear that for any $I\in (2^\mathbb{N})_{\mathrm{fin}}\backslash \{\emptyset\}$, $$\prod_{i\in I}\epsilon_{\mathsf{S}_i}=\lim_{n\to\infty}\sum_{F\in (2^{P_n})_{\mathrm{fin}}\backslash \{\emptyset\}}\mathbbm{1}_{\{\{\mathsf{S}_i(B):i\in I\}=\{\tau_p:p\in F\}\}}\prod_{p\in F}\epsilon_{\tau^p}\text{ a.s.-$\PP$ boundedly}.$$ 	
Also,
$\left\{g(B)\prod_{i\in I}\epsilon_{\mathsf{S}_i}:(g,I)\in \LLL^2(\WW)\times ((2^\mathbb{N})_{\mathrm{fin}}\backslash \{\emptyset\})\right\}$  is total in $H_{\mathrm{sens}}$. 
(Indeed, the family $\{g\otimes \prod_{i\in I}\xi_i:(g,I)\in \LLL^2(\WW)\times (2^\mathbb{N})_{\mathrm{fin}}\}$, where $(\xi_i)_{i\in \mathbb{N}}$ is the sequence of the coordinate projections on $\{-1,1\}^\mathbb{N}$, is total in $\LLL^2(\WW)\otimes \LLL^2((\frac{1}{2}\delta_{-1}+\frac{1}{2}\delta_1)^{\times \mathbb{N}})=\LLL^2(\WW\times (\frac{1}{2}\delta_{-1}+\frac{1}{2}\delta_1)^{\times \mathbb{N}})$ (equality up to the natural unitary equivalence), and therefore --- directly from the construction of $\PP$ via $\Theta$ in Subsection~\ref{subsection:splitting-noise} --- we get that $\left\{g(B)\prod_{i\in I}\epsilon_{\mathsf{S}_i}:(g,I)\in \LLL^2(\WW)\times (2^\mathbb{N})_{\mathrm{fin}}\right\}$ is total in $\LLL^2(\PP)$.) It is now elementary to conclude the argument.
 \end{proof}
 
 \begin{proof}[Proof of Theorem~\ref{theorem:max-unstable}]
The case $\leb(E)=0$ is elementary (\ref{unstable:iii}-\ref{unstable:ii} all fail). Assume $\leb(E)>0$. The equivalences \ref{unstable:iii}$\Leftrightarrow$\ref{unstable:iv}$\Leftrightarrow$\ref{unstable:i} are essentially trivial on noting that: maximizers of $W$ on intervals with rational endpoints, say, exhaust the local maxima of $W$; for any maximizer $\tau$ of $W$ on a non-degenerate compact interval $I$ and any non-degenerate compact subinterval $J$ of $I$, $\tau$ is the maximizer of $W$ on $J$ a.s.-$\WW$ on $ \{\tau\in J\}$.

 \ref{unstable:v}$\Rightarrow$\ref{unstable:iii}. Assuming \ref{unstable:iii} fails (then \ref{unstable:i} does and hence) there is a maximizer $\tau$ of  $W$ such that $\QQ(\tau=\tau_E\in E)>0$. Consequently, since $\tau$ and $\tau_E$ are independent and equally distributed given $\FF^W_E$, the conditional distribution of $\tau$ given $\FF^W_E$ must contain an atom belonging to $E$ with  positive $\WW$-probability. Thus there exists an $\FF^W_E$-measurable selection $Z$ of a local maximum of $W$ belonging to $E$ and having $\WW(Z\in E)>0$. This $Z$ must be a local maximum of $\mathbbm{1}_E\cdot W$ a.s-$\WW$ on $\{Z\in E\}$ (which is a contradiction with  \ref{unstable:v}) for if it was not, then there would exist an $\FF^W_E$-measurable sequence $(Z_n)_{n\in \mathbb{N}}$ of random variables converging to $Z$ as $n\to\infty$, but different from $Z$, and such that $(\mathbbm{1}_E\cdot W)_{Z_n}\geq (\mathbbm{1}_E\cdot W)_{Z}$ on an event $A\in \FF^W_E$ of  $\WW$-positive probability contained in $\{Z\in E\}$. But $\WW$-a.s. $W$ results from $\mathbbm{1}_E\cdot W$ by the addition of $\mathbbm{1}_{\mathbb{R}\backslash E}\cdot W$, which is independent of $\FF^W_E$ and whose increments are symmetric. It follows that $\WW(W_{Z_n}\geq W_Z\vert \FF^W_E)\geq \frac{1}{2}$ a.s.-$\WW$ on $A$, a fortiori $\WW(W_{Z_n}\geq W_Z\vert A)\geq \frac{1}{2}$, for all $n\in \mathbb{N}$. By (reverse) Fatou's lemma we deduce that, conditionally on $A$, $W_{Z_n}\geq W_Z$ for infinitely many $n\in\mathbb{N}$ with probability at least a half. But then certainly with $\WW$-positive probability on $\{Z\in E\}$, $Z$ is not a local maximum of $W$, which by itself is in contradiction with the choice of $Z$.
 
  \ref{unstable:iii}$\Rightarrow$\ref{unstable:v}. Again let us argue by contradiction. If \ref{unstable:v} does not hold true, then for some $\FF^W_E$-measurable selection $\tau$ of a local maximum of $\mathbbm{1}_E\cdot W$, $\WW(\tau\text{ a local maximum of $W$})>0$. Hence, with $\QQ$-positive probability we would also have 
$$\QQ(\tau\text{ a local maximum of $W$ and of $W_E$}\vert \FF_E^W)=\QQ(\tau\text{ a local maximum of $W$}\vert \FF_E^W)^2>0,$$ the equality coming from the conditional independence and equality in distribution of $W$ and $W_E$ given $\FF^W_E$. This contradicts   \ref{unstable:iii}.

\ref{unstable:ii}$\Rightarrow$\ref{unstable:i}. Assume \ref{unstable:ii}, i.e. that $0_\PP\ne \FF_E\subset \FF_{\mathrm{stb}}$. For any maximizer $\tau$ of $W$, $\epsilon_T\in H^{(1)\prime}\perp \LLL^2(\PP\vert_{\FF_{\mathrm{stb}}})$, hence Proposition~\ref{proposition:max-enumerable-necessary} yields $\QQ(\tau=\tau_E\in E)=0$. Since  $\leb(E)>0$  we must have $\WW(\tau\in E)>0$ for some such $\tau$. This is \ref{unstable:i}.  

\ref{unstable:i}$\Rightarrow$\ref{unstable:ii}. If \ref{unstable:i} holds, then  Proposition~\ref{proposition:max-enumerable-necessary} and Lemma~\ref{lemma-key-totality} show that $\LLL^2(\PP\vert_{\FF_E})$ is orthogonal to the sensitive subspace, which forces $\LLL^2(\PP\vert_{\FF_E})\subset {H_{\mathrm{sens}}}^\perp=H_{\mathrm{stb}}=\LLL^2(\PP\vert_{\FF_{\mathrm{stb}}})$, therefore $\FF_E\subset \FF_{\mathrm{stb}}$; also, $\leb(E)>0$ entails $\FF_E\ne 0_\PP$. Altogether, we get \ref{unstable:ii}. 

 \ref{unstable:ii}$\Rightarrow$\ref{unstable:new}.  Barring \ref{unstable:new} we get an $\FF^W_E$-measurable selection $T$ of a local maximum of $W$ on $E$ with $\WW(T\ne \dagger)>0$.  By Proposition~\ref{lemma:locmaxstable}, $\epsilon_T$ is $\FF_E$-measurable. Since $\WW(T\ne\dagger)>0$, $\PP(\epsilon_T\ne 0)>0$. From  \ref{unstable:ii},  $\FF_E\subset \FF_{\mathrm{stb}}$, and so  $\epsilon_T=\PP[\epsilon_T\vert \FF_{\mathrm{stb}}]=0$ a.s.-$\PP$. Put together, we have a contradiction. 
   
\ref{unstable:new}$\Rightarrow$\ref{unstable:ii}.  Assume  \ref{unstable:ii} fails so that $\FF_E\not\subset \FF_{\mathrm{stb}}$. Then for some $f\in \LLL^2(\PP\vert_{\FF_E})$, $f-\PP[f\vert \FF_{\mathrm{stb}}]\in (H_{\mathrm{sens}}\cap \LLL^2(\PP\vert_{\FF_E}))\backslash \{0\}$, where we have taken into account that by \eqref{eq:stable-in-vNa} and \eqref{eq:commuting}, $\PP[\cdot\vert\FF_E]$ and $\PP[\cdot\vert \FF_{\mathrm{stb}}]$ both belong to $\AA$ and are therefore commuting. In turn, the $\mu$-a.e. equality $\{K=\infty\}\cap S_E=\cup_P \cap_{p\in P}\pr_p^{-1}(\{ K'=1\}\cap S_E)$, where $\cup_P$ is over all partitions $P$ of $\mathbb{R}$ into intervals with rational endpoints, combined with \eqref{tensor-product} assures us that  $H^{(1)\prime}\cap \LLL^2(\PP\vert_{\FF_E})\ne  \{0\}$. This contradicts Proposition~\ref{proposition:rich-measures} and \ref{unstable:new}.
 
 Suppose finally that $0_\PP\ne \FF_E\subset \FF_{\mathrm{stb}}$ and per absurdum $E\in \overline{\EE}$. Since $\FF_E\subset \FF_{\mathrm{stb}}$, certainly $H^{(1)\prime}\cap \LLL^2(\PP\vert_{\FF_E})=\{0\}$. Owing to $E\in \overline{\EE}$, by Proposition~\ref{proposition:weird} we obtain $\mu(K'=1,\acc\subset E)=0$.  Also,  $0_\PP\ne \FF_E$ entails $\leb(E)>0$.  Taken together this cannot stand by the Poisson character of $\mu$: recalling the discussion of p.~\pageref{Poisson-character}, we know that $\acc_\star(\mu\vert_{\{K'<\infty\}})$ is equivalent to the law of a Poisson point process on $\mathbb{R}$ with finite intensity measure equivalent to $\leb$; therefore, since $\leb(E)>0$, $\mu(\vert \acc\cap E\vert=1,\vert \acc\backslash E\vert=0)>0$, which is in direct contradiction with $\mu(K'=1,\acc\subset E)=0$.
 \end{proof}
 
 \begin{corollary}\label{corollary:stable-not-in-completion}
$\BBB_{\mathrm{stb}}\cap \overline{\BBB}=\{0_\PP\}$.
 \end{corollary}
 \begin{proof}
Apply the last asssertion of Theorem~\ref{theorem:max-unstable} with condition \ref{unstable:ii} thereof.
 \end{proof}
Theorem~\ref{thm:is-in-closure} from Appendix~\ref{appendix:ntba} will show that  $\FF_{\mathrm{stb}}\in \Cl(\BBB)$.
 It is remarkable that the existence of $\leb$-measurable sets that are max-totally unstable  follows already from this highly abstract result:
 \begin{example}\label{example:unstable}
Since $\FF_{\mathrm{stb}}\in \Cl(\BBB)$, by \eqref{eq:closure-identified} there is a sequence  $(A_n)_{n\in \mathbb{N}}$ of closed elementary sets such that $\liminf_{n\to\infty}\FF_{A_n}=\FF_{\mathrm{stb}}$.  Remark~\ref{rmk:intersect} and sequential monotone continuity of $\FF^\mathrm{stb}$ yield $\liminf_{n\to\infty}A_n=\mathbb{R}$ a.e.-$\leb$ (or see it directly by combining Proposition~\ref{proposition:meet-and-join}\ref{cap}, Proposition~\ref{proposition:extensions-mod}\ref{proposition:extensions-mod:i}, $\liminf_{n\to\infty}\FF_{A_n}\subset \FF_{\liminf_{n\to\infty}A_n}$). From Theorem~\ref{theorem:max-unstable} we  conclude that for all large enough $n\in \mathbb{N}$, $E_n:=\cap_{m\in \mathbb{N}_{\geq n}}A_m$, which by the preceding is $\uparrow\mathbb{R}$ a.e.-$\leb$ as $n\to\infty$, is  a closed max-totally unstable set (``large enough'' ensuring $ \leb(E_n)>0$).
\end{example}
For the remainder of this subsection let $\tau:=\tau_{0,1}$ be the maximizer of $W$ on $[0,1]$. We have just seen in the previous example
\begin{enumerate}[(A)]
\item\label{dichotomy:B} the  existence of a sequence $(E_n)_{n\in \mathbb{N}}$ of closed [automatically nowhere dense\footnote{For a closed not-nowhere dense $E\subset [0,1]$ trivially $\QQ(\tau=\tau_E)>0$.}] subsets of $[0,1]$ that is $\uparrow [0,1]$ a.e.-$\leb$ and such that $\QQ(\tau=\tau_{E_n})=0$ for all $n\in \mathbb{N}$. \qed
\end{enumerate}
The reader may well wonder whether  $\QQ(\tau=\tau_E)=0$ does not perhaps anyway hold true always whenever $E$ is a closed nowhere dense set. It is not so, indeed we claim 
\begin{enumerate}[(B)]
\item\label{dichotomy:A} the  existence of a sequence $(E_n)_{n\in \mathbb{N}}$ of closed  nowhere dense subsets of $[0,1]$ that is $\downarrow \emptyset$ a.e.-$\leb$ and such that $\QQ(\tau=\tau_{E_n})>0$ for all $n\in \mathbb{N}$.
\end{enumerate}
\begin{proof}[Proof of \ref{dichotomy:A}]\label{proof:A}
 We apply  Theorem~\ref{theorem:dense-fuck-stable} to $\{\FF_A:A\in \EE\cap 2^{[0,1]}\}$ as the noise  Boolean algebra (under $\PP\vert_{\FF_{[0,1]}}$), to $\epsilon_\tau$ as the random variable $\xi$, and to the sequence of finite noise Boolean subalgebras that are attached to the successive dyadic partitions of $[0,1]$ as $(b_n)_{n\in \mathbb{N}}$ (so, for each $n\in \mathbb{N}$, the atoms of $b_n$ are the $\FF_I$ for dyadic intervals $I$ of $[0,1]$ of lenght $2^{-n}$). Condition \eqref{dense-fuck-stable-weird} is met evidently, \eqref{dense-fuck-stable-dense} due to $K'<\infty$ a.e.-$\mu$. We get existence of a $\downarrow$ sequence $(A_n)_{n\in \mathbb{N}}$ of finite unions of closed subintervals of $[0,1]$, whose intersection $A$ is nowhere dense in $[0,1]$ (vis-\`a-vis \ref{dense-stbl-a}  of Theorem~\ref{theorem:dense-fuck-stable}) and such that (vis-\`a-vis \ref{dense-stbl-b}  of Theorem~\ref{theorem:dense-fuck-stable}) $$0<\PP[\PP[\epsilon_\tau\vert \land_{n\in \mathbb{N}}\FF_{A_n}]^2]=\PP[\PP[\epsilon_\tau\vert\FF_{A}]^2]=\QQ(\tau=\tau_A),$$ furthermore  (vis-\`a-vis \ref{dense-stbl-c}  of Theorem~\ref{theorem:dense-fuck-stable}) $$0<\PP[\PP[\epsilon_\tau\vert \FF_I\land (\land_{n\in \mathbb{N}}\FF_{A_n})]^2]=\PP[\PP[\epsilon_\tau\vert\FF_{A\cap I}]^2]=\QQ(\tau=\tau_{A\cap I}),$$ for every dyadic interval $I$ of $[0,1]$ for which $\leb(E\cap I)>0$, on using Proposition~\ref{proposition:max-enumerable-necessary} for the final two equalities of the preceding displays. It is now easy to conclude by taking a $\downarrow\downarrow$ sequence $(I_n)_{n\in \mathbb{N}}$ of closed dyadic intervals, $ \leb(I_n\cap A)>0$  and $I_n$ having length  $2^{-n}$ for all $n\in \mathbb{N}$, which surely exists, then setting $E_n:=I_n\cap A$ for $n\in \mathbb{N}$.
\end{proof}


\subsection{Max-totally stable sets}\label{subsection:stable}
Let us turn to the other end of the spectrum of sensitivity or rather stability of the local maxima.

\begin{theorem}\label{thm:stability}
For an $\leb$-measurable $E$, the following are equivalent.
\begin{enumerate}[(a)]
\item\label{stable.i} The local maxima of $W$ belonging to $E$ are precisely the local maxima of $W_E$ belonging to $E$ a.s.-$\QQ$.
        \item\label{stable.v}  The local maxima of $W$ belonging to $E$ are contained in the local maxima of $\mathbbm{1}_E\cdot W$  a.s.-$\WW$.
\item \label{stable.iii} For every selection $\tau$ of a local maximum of $W$: $\WW(\tau\in E)>0\Rightarrow \QQ(\tau=\tau_E\in E)>0$.
\item \label{stable.iv} For all maximizers $\tau$ of $W$,  all  $G\in\mathcal{B}_\mathbb{R}$: $\WW(\tau\in E\cap G)>0\Rightarrow \QQ(\tau=\tau_E\in E\cap G)>0$. 
\item \label{stable.ii} $E$ is max-enumerable.
\end{enumerate}
\end{theorem}
 \begin{definition}\label{definition:totally-stable}
 An $\leb$-measurable $E$ is said to be max-totally stable if it meets one and then all of the equivalent conditions  \ref{stable.i}-\ref{stable.iv} of Theorem~\ref{theorem:max-unstable}.
 \end{definition}
So, max-enumerability is equivalent to max-total stability and thus, by  Theorem~\ref{theorem:characterization}, $E\in \overline{\EE}$ iff $E$ and $\mathbb{R}\backslash E$ are both max-totally stable. Of course  an open $E$ is max-totally stable. 

Recall from Subsection~\ref{subsection:splitting-noise}, p.~\pageref{levy-shift},  the L\'evy shifts $(\Delta_t)_{t\in \mathbb{R}}$ acting on $\Omega_0$. For a continuous $\omega:\mathbb{R}\to \mathbb{R}$, interpret $\Delta_t(\omega):=\omega(t+\cdot)-\omega(t)$ for $t\in \mathbb{R}$ and $\Delta_\dagger(\omega)\equiv \mathsf{w}_0$, even if $\omega\notin \Omega_0$. We attempt the proof of Theorem~\ref{thm:stability} after preparing
\begin{lemma}\label{lemma:disaster}
Let $E$ be $\leb$-measurable and let, under $\QQ$, $\mathsf{e}$ be an exponentially distributed $(0,\infty)$-valued random variable independent of $(W,W')$. Denote by $\tau_\ee$ the maximizer of $W$ on $[0,\ee]$. Then 
\begin{align}\label{eq:trivial-deep}
\cap_{t\in (0,\infty)}0_\QQ\lor \sigma(\tau_\ee,(\Delta_{\tau_\ee}(W_E))\vert_{[0,t]})&=0_\QQ\lor \sigma(\tau_\ee)\text{ and}\\\label{eq:trivial-superficial}
\cap_{t\in (0,\infty)}0_\QQ\lor \sigma(\tau_\ee,(\Delta_{\tau_\ee}(\mathbbm{1}_E\cdot W))\vert_{[0,t]})&=0_\QQ\lor \sigma(\tau_\ee).
\end{align}
Also (no need for $\ee$ here), for any  $\leb$-measurable $A$,
\begin{equation}\label{metamorphosis}
(\mathbbm{1}_A\cdot W)_E=\mathbbm{1}_A\cdot W_E=\mathbbm{1}_{A\cap E}\cdot W+\mathbbm{1}_{A\backslash E}\cdot W'\text{ a.s.-$\QQ$}.
\end{equation}
\end{lemma}
\begin{proof}
Since the $0_\QQ\lor \sigma(W,\ee)$-measurable $\tau_\ee$ is independent of $W'$, the law of which is invariant under the L\'evy shifts,  we see that $\Delta_{\tau_\ee}W'$ is $\QQ$-independent of $(W,\ee)$ and ${\Delta_{\tau_\ee}W'}_\star\QQ=\WW$. By Blumenthal's zero-one law (for Brownian motion) it follows that $\cap_{t\in (0,\infty)}0_\QQ\lor \sigma((\Delta_{\tau_\ee}W')\vert_{[0,t]}])=0_\QQ$.
From the latter  and from \eqref{eq:millar} we deduce via \eqref{distributivity-indep-not-complete} that  $$\cap_{t\in (0,\infty)}0_\QQ\lor \sigma(\tau_\ee, \Delta_{\tau_\ee}\vert_{[0,t]},(\Delta_{\tau_\ee}W')\vert_{[0,t]}])=0_\QQ\lor \sigma(\tau_\ee).$$

At this point  \eqref{eq:trivial-deep}-\eqref{eq:trivial-superficial} seem almost immediate on an intuitive level: it appears to be indeed clear that the increments of  $\mathbbm{1}_E\cdot W$ and of $\mathbbm{1}_{\mathbb{R}\backslash E}\cdot W'$, therefore of $W_E$, on $[\tau_\ee,\tau_\ee+t]$ are recoverable from $\tau_\ee$ and from  the increments of $W$ and $W'$ on $[\tau_\ee,\tau_\ee+t]$, this for each $t\in (0,\infty)$. The latter  is indeed plain when $E\in \EE$, but for a general $\leb$-measurable $E$ it is in fact not completely transparent, the Wiener integral being defined by a ``global'' procedure (limits in $\LLL^2(\WW)$) rather than ``pathwise''. Nevertheless, we can harness this observation for elementary $E$, together with an approximation to complete the argument. 

As is well-known, on a finite measure space the elements of the underlying $\sigma$-algebra are approximated arbitrarily well in measure by those of a generating algebra; applying this to a finite measure equivalent to $\leb$ and passing to a subsequence if necessary,  we get existence of  a sequence $(E_n)_{n\in \mathbb{N}}$ in $\EE$ such that $\lim_{n\to\infty}\mathbbm{1}_{E_n}= \mathbbm{1}_E$ a.e.-$\leb$. By \ref{loc-unif-a.s.} of p.~\pageref{loc-unif-a.s.} (applied to $(W,(\mathbbm{1}_{E_n})_{n\in \mathbb{N}})$ and to $(W',(\mathbbm{1}_{\mathbb{R}\backslash E_n})_{n\in \mathbb{N}})$), after passing to a further subsequence if necessary, it follows that $\lim_{n\to\infty}W_{E_n}=W_E$ (uniformly on compact time intervals) a.s.-$\QQ$. Accordingly, 
\begin{align*}
\cap_{t\in (0,\infty)}0_\QQ\lor \sigma(\tau_\ee,(\Delta_{\tau_\ee}(W_E))\vert_{[0,t]})&\subset \cap_{t\in (0,\infty)}0_\QQ\lor \left(\lor_{n\in \mathbb{N}}\sigma(\tau_\ee,(\Delta_{\tau_\ee}(W_{E_n}))\vert_{[0,t]})\right)\\
&\subset \cap_{t\in (0,\infty)}0_\QQ\lor \sigma(\tau_\ee,\Delta_{\tau_\ee}\vert_{[0,t]},\Delta_{\tau_\ee}(W')\vert_{[0,t]}).\end{align*}
Combining the preceding two displays we get at once \eqref{eq:trivial-deep}, while \eqref{eq:trivial-superficial} follows by the slightest adjustment of the above argument.

As for \eqref{metamorphosis}, the second  equality is just associativity of (stochastic) integration, while the first is got from: if $\mathbbm{1}_A\cdot W=F(W)$ a.s.-$\WW$ for a measurable $F$ (such $F$ is of course unique within $\WW$-a.s. equality), then $\mathbbm{1}_A\cdot H=F(H)$ a.s.-$\mathbb{F}$  for any two-sided Brownian motion $H$ under a probability $\mathbb{F}$ (indeed, to get \eqref{metamorphosis}, just apply this with $H=W_E$ under $\QQ$). In turn, the latter implication clearly holds for $A\in \EE$; for a general $A$ approximate as above.
\end{proof}

\begin{proof}[Proof of Theorem~\ref{thm:stability}]
\ref{stable.v}$\Rightarrow$\ref{stable.ii}. By approximation from above  to within a set of $\leb$-measure zero we may and do assume $E$ is a $G_\delta$.  Consider an $\FF^W_E$-measurable enumeration $S$ of the local maxima of $\mathbbm{1}_E\cdot W$ belonging to $E$. By the $G_\delta$ property of $E$ and the $\downarrow$ sequential monotone continuity of~$\FF^W$, $$\{S_k\text{ is a local maximum of } W\}\in\FF^W_E,\quad k\in \mathbb{N}.$$  Thus we obtain at once an $\FF^W_E$-measurable enumeration of the local maxima of $W$ on $E$.

\ref{stable.ii}$\Rightarrow$\ref{stable.v}. We have seen in the proof of Theorem~\ref{theorem:max-unstable}, \ref{unstable:v}$\Rightarrow$\ref{unstable:iii}, that an $\FF^W_E$-measurable selection $Z$  of a local maximum of $W$ belonging to $E$ is automatically a local maximum of $\mathbbm{1}_E\cdot W$ a.s.-$\WW$ on $\{Z\in E\}$. 

 \ref{stable.i}$\Rightarrow$\ref{stable.ii}. Set $\tilde W:=\mathbbm{1}_E\cdot W$ and $\tilde W':= \mathbbm{1}_{\mathbb{R}\backslash E} \cdot W$. The local maxima of $W$ belonging to $E$ admit an enumeration $S$. Since  $\FF^W_E\lor \FF^W_{\mathbb{R}\backslash E}=1_\WW$, $\FF^W_E=\sigma(\tilde W)\lor 0_\WW$ and $\FF^W_{\mathbb{R}\backslash E}=\sigma(\tilde W')\lor 0_\WW$ we see that
 $S=F(\tilde{W},\tilde{W}')$ a.s.-$\WW$ for some measurable $F$. Then $\QQ$-a.s. the range of $S_E=F(\tilde{W}_E,\tilde{W}'_E)=F(\tilde{W}, \mathbbm{1}_{\mathbb{R}\backslash E} \cdot W')$  (second a.s. equality due to \eqref{metamorphosis}) 
 without $\dagger$ precisely covers the local maxima of $W_E$ belonging to $E$ ($\because$  ${W_E}_\star\QQ=W_\star\QQ$), hence the local maxima of $W$ belonging to $E$ ($\because$ \ref{stable.i}).  Now, $(\tilde{W},W)$ is $\QQ$-independent of $ \mathbbm{1}_{\mathbb{R}\backslash E} \cdot W'$.  It means that for $((\tilde{W},W), \mathbbm{1}_{\mathbb{R}\backslash E} \cdot W')_\star \QQ=((\tilde{W},W)_\star\WW)\times (( \mathbbm{1}_{\mathbb{R}\backslash E} \cdot W')_\star\QQ)$-a.e. $((\tilde w,w),w')$, the range of $F(\tilde w,w')$ without $\dagger$ precisely exhausts the local maxima of $w$ belonging to $E$. By Tonelli, for $( \mathbbm{1}_{\mathbb{R}\backslash E} \cdot W')_\star\QQ$-a.e. $w'$ --- hence certainly for some $w'$ --- for $(\tilde{W},W)_\star\WW$-a.e. $(\tilde w,w)$, the range of $F(\tilde w,w')$ without $\dagger$ precisely covers the local maxima of $w$ belonging to $E$. Thus, for the stated $w'$, $\WW$-a.s. the range of $F(\tilde{W},w')$ without $\dagger$ precisely exhausts the local maxima of $W$ belonging to $E$. On noting that $\FF^W_E\supset 0_\WW$ so that the exceptional set on which the property fails does not matter, we deduce that $E$ is max-enumerable.

 \ref{stable.ii}$\Rightarrow$\ref{stable.iii}.  Let $\tau$ be a selection of a local maximum of $W$ such that $\WW(\tau\in E)>0$. The random elements $\tau$ and $\tau_E$ are independent given $\FF^W_E$ and consequently for an $\FF_E^W$-measurable enumeration $(\tau_n)_{n\in \mathbb{N}}$ of the local maxima of $W$ on $E$   satisfying $\tau_i\ne \tau_j$ a.s.-$\WW$ on $\{\tau_i\ne\dagger,\tau_j\ne\dagger\}$ for $i\ne j$ from $\mathbb{N}$, we get   
 \begin{align*}
 0&<\WW\left[\sum_{n\in \mathbb{N}}\WW(\tau=\tau_n\in E\vert \FF_E^W)^2\right]=
 \sum_{n\in \mathbb{N}}\QQ\left[\QQ(\tau=\tau_n\in E,\tau_E=\tau_n\in E \vert \FF_{E}^W)\right]\\
 &=\QQ(\tau=\tau_E\in E).
 \end{align*}
 
   \ref{stable.iii}$\Rightarrow$\ref{stable.iv}. We have only to apply \ref{stable.iii} to the selection which is equal to $\tau$ on $\{\tau\in G\}$ and equal to $\dagger$ otherwise. 
 
  \ref{stable.iv}$\Rightarrow$\ref{stable.i}. We shall appeal here to \eqref{eq:trivial-deep} of Lemma~\ref{lemma:disaster}. 
  Assume then, without loss of generality, that $\QQ$ supports also an independent exponential random variable $\mathsf{e}$ of mean one, $(0,\infty)$-valued with certainty. 
  
  For $a\in \mathbb{R}$ (fixed for a while) let us consider the maximizer $\tau_a:=\tau_{a,a+\ee}$ of $W$ on the random interval $[a,a+\mathsf{e}]$.   Below, a right (resp. left) local maximum of a continuous $\omega:\mathbb{R}\to \mathbb{R}$  means a time $t\in \mathbb{R}$ for which $\omega$ stays strictly below $\omega(t)$ on $(t,t+\delta)$ (resp. on $(t-\delta,t)$) for some $\delta>0$. According to \eqref{eq:trivial-deep}, 
\begin{align*}
  \Gamma_a&:=\{\text{$\tau_a$ is a right local maximum of $W_E$ belonging to $E$}\}\\
  &=\{\text{$0$ is a right local maximum of $\Delta_{\tau_a}(W_E)$},\tau_a\in E\}\in  0_\QQ\lor \sigma(\tau_a);
  \end{align*}
thus $\Gamma_a=\tau_a^{-1}(F_a)$ a.s.-$\QQ$ for some Borel set $F_a\subset E\cap [a,\infty)$.

  We want  to argue now that  each $F_{a}$ can be chosen to be the intersection of a single Borel $F\subset E$ with $[a,\infty)$ across all $a\in \mathbb{R}$.    To this end, let $a\leq b$ from $\mathbb{R}$ be arbitrary.  We have that $\tau_a=\tau_b$ and hence ($\Gamma_a=\Gamma_b$, therefore) $\{\tau_a\in F_a\}=\{\tau_a\in F_b\}$  a.s.-$\QQ$ on $$\{\tau_a\geq b\}\cap \underbrace{\{\text{$W$ stays strictly below $W_{\tau_a}$ on $(a+\mathsf{e},b+\mathsf{e}]$}\}}_{\text{independent of $\tau_a$ and a.s.-$\QQ$ of positive probability \emph{given} $\mathsf{e}$}};$$ 
the noted fact in the underbrace (together with the independence of $\ee$ and $W$) entails that actually $\{\tau_a\in F_a\}=\{\tau_a\in F_b\}$  a.s.-$\QQ$ on $\{\tau_a\geq b\}$. Since ${\tau_a}_\star \QQ\vert_{\{\tau_a\geq b\}}\sim \leb\vert_{[b,\infty)}$, we deduce that $F_a\cap[b,\infty)=F_b$ a.e.-$\leb$. Therefore, setting $F:=E\cap \left(\text{$(\leb\vert_{\BB_\mathbb{R}})$-ess-$\cup$}\{F_d: d\in \mathbb{R}\}\right)$, we have indeed that
\begin{equation}\label{doublefuck}
\Gamma_d=\tau_d^{-1}(F)\text{ a.s.-$\QQ$ for all $d\in\mathbb{R}$}.
\end{equation}

 From \eqref{doublefuck}, the independence of $\ee$ and $(W,W_E)$ and Tonelli's theorem it now follows that $\QQ$-a.s.  the local maxima of $W$ that belong to $F$ are right local maxima of $W_E$, while $\QQ$-a.s. the local maxima of $W$ belonging to $\mathbb{R}\backslash F$ are not right local maxima of $W_E$ belonging to $E$. Of course, by the symmetry of time-reversal, there is also a Borel set $G\subset E$ such that the preceding statement holds with ``left'' replacing ``right'' and $G$ replacing $F$. Setting $H:=F\cap G\subset E$ we deduce that:

 $(*_1)$ $\QQ$-a.s. the local maxima of $W$ on $H$ are local maxima of $W_E$; but also
 
$(*_2)$ $\QQ$-a.s. the local maxima of $W$ on $ \mathbb{R}\backslash H$ are not local maxima of $W_E$ belonging to $E$. 

Notice that in order to arrive at $(*_1)$-$(*_2)$ we have not had to use \ref{stable.iv}.

Suppose next, per absurdum, that $\leb(E\backslash H)>0$. Then, for some compact non-degenerate interval $I\subset \mathbb{R}$, $\leb((E\backslash H)\cap I)>0$ and hence, with $\tau$ the maximizer of $W$ on $I$, $\QQ(\tau\in E\backslash H)>0$. In turn, from \ref{stable.iv} we get $\QQ(\tau=\tau_E\in E\backslash H)>0$. But $\QQ$-a.s. on $\{\tau=\tau_E\in E\backslash H\}$, on the one hand $\tau$ is a local maximum of $W_E$ belonging to $E$ ($\because$ $\tau=\tau_E\in E$), on the other hand $\tau$ is a local maximum of $W$ that does not belong to $H$. By $(*_2)$ this is a contradiction and we are forced to conclude that $H=E$ a.e.-$\leb$. 

Finally, no local maxima of $W$ belong to an $\leb$-negligible set a.s.-$\WW$ anyway. We infer from this,  from $E=H$ a.e-$\leb$ and from ($*_1$)  that the local maxima of $W$ on $E$ are local maxima of $W_E$ a.s.-$\QQ$. By symmetry \eqref{eq:symmetric-coupling} the converse is also true, and so we have \ref{stable.i}. 
\end{proof}
We conclude this subsection with the following result that will prove important later.

\begin{proposition}\label{proposition:loc-max-of-censored}
Let $E$ be $\leb$-measurable. Suppose that $\WW$-a.s. the local maxima of $\mathbbm{1}_E\cdot W$  are local maxima of $W$. Then $\WW$-a.s. the local maxima of $W$ belonging to $E$ are local maxima of  $\mathbbm{1}_E\cdot W$. (Cf. the condition of Theorem~\ref{thm:stability}\ref{stable.v}.)
\end{proposition}
\begin{proof}
 Much the same as in, and in the notation of the proof of Theorem~\ref{thm:stability}, \ref{stable.iv}$\Rightarrow$\ref{stable.i} --- we have literally only to replace $W_E$ with $\mathbbm{1}_E\cdot W$ and appeal to \eqref{eq:trivial-superficial} instead of \eqref{eq:trivial-deep} in the segment of the argument leading up to $(*_1)$-$(*_2)$ ---  we avail ourselves of the existence of a Borel $H\subset E$ such that   $\WW$-a.s. the local maxima of $W$ on $H$ are local maxima of $\mathbbm{1}_E\cdot W$, but also  $\WW$-a.s. the local maxima of $W$ on $ \mathbb{R}\backslash H$ are not local maxima of $\mathbbm{1}_E\cdot W$ belonging to $E$. Recall from \ref{maxima-of-censored} on p.~\pageref{maxima-of-censored'}  that the local maxima of $\mathbbm{1}_E\cdot W$ belong $\WW$-a.s. to $E$ anyway.

Assuming per absurdum that $\leb(E\backslash H)>0$, then by \eqref{aux:loc-max-censored} there is a selection $\tau$ of a local maximum of $\mathbbm{1}_E\cdot W$ such that $\WW(\tau\in E\backslash H)>0$. But $\WW$-a.s. on the event $\{\tau\in E\backslash H\}$,  by the very assumption of this proposition $\tau$ is a local maximum of $W$, hence, by the finding of the preceding paragraph, it is not a local maximum of $\mathbbm{1}_E\cdot W$, which is a contradiction, since it has nevertheless been chosen as being one of the local maxima of $\mathbbm{1}_E\cdot W$. Therefore $\leb(E\backslash H)=0$, and we conclude easily.
\end{proof}

\subsection{Density considerations}\label{subsection:density-considerations}
Recall the notation \eqref{E-t-u-notation} for an $\leb$-measurable $E$. According to the Lebesgue density theorem \cite[Theorem~7.10]{rudin}, 
\begin{equation}\label{eq:convergence-of-density}
\text{for $\leb$-a.e. $t$},  \quad  h^{-1}E_{t,t+h}  \rightarrow  \mathbbm{1}_E(t)\quad \mbox{ as } h\to 0.
\end{equation}
We are interested in how the speed at which this convergence occurs on $E$ controls the max-total (in)stability of $E$.

The first part of Proposition~\ref{proposition:density-of-sets} to follow says that if the convergence of \eqref{eq:convergence-of-density} on $E$ is fast enough then $E$ is max-totally stable: specifically,  $\WW$-a.s. the  local maxima of the censored Brownian motion $\mathbbm{1}_E\cdot W$ will be seen to be local maxima of $W$, which implies max-total stability by Theorem~\ref{thm:stability} and Proposition~\ref{proposition:loc-max-of-censored}.  Then,  complementing the previous result, we have that  $E$ is max-totally unstable if  the convergence of \eqref{eq:convergence-of-density} is slow enough.  Note that there is a gap between the conditions imposed on the behaviour of the densities which a more sophisticated approach might close.

\begin{proposition}\label{proposition:density-of-sets}
Let  a $\uparrow$ function $g>0$ be defined on $(0,\delta)$ for some $\delta>0$ and let $E$ be $\leb$-measurable.
\begin{enumerate}[(i)]
\item \label{towards-stable}
If $\int_{0+} g(h) \frac{\dd h}{h} < \infty$ and  
\[
\limsup_{h\to 0}\frac{|h| - E_{t,t+h}}{ \frac{ |h| \,g(|h|)^2}{ {\log\log (1/(\sqrt{|h|}g(|h|)))}}} <\infty\text{ for $\leb$-a.e. $t\in E$},
\]
then $E$ is max-totally stable.
\item \label{towards-unstable}
If $\int_{0+} g(h) \frac{\dd h}{h} =\infty$, $\leb(E)>0$ and  
\[
\left(\liminf_{h\downarrow 0}\frac{h - E_{t,t+h}}{hg(h)^2}\right)\lor \left(\liminf_{h\uparrow 0}\frac{|h| - E_{t,t+h}}{|h|g(|h|)^2}\right)>0\text{ for $\leb$-a.e. $t\in E$},
\]
then $E$ is max-totally unstable.
\end{enumerate}
\end{proposition}
\begin{remark}\label{remark:alphas}
In particular, if for some $\alpha\in (2,\infty)$,
\[
\limsup_{h\to 0}\frac{|h| - E_{t,t+h}}{\frac{ |h|}{(\log(1/|h|))^\alpha}} <\infty \text{ for $\leb$-a.e. $t\in E$},
\]
 then $E$ is max-totally stable; if, on the other hand, $\leb(E)>0$ and 
\[
\left(\liminf_{h\downarrow 0}\frac{h - E_{t,t+h}}{\frac{ h}{(\log(1/h))^2}}\right)\lor \left(\liminf_{h\uparrow 0}\frac{|h| - E_{t,t+h}}{\frac{ |h|}{(\log(1/|h|))^2}}\right) >0\text{ for $\leb$-a.e. $t\in E$,}
\]
 then $E$ is  max-totally unstable. 
\end{remark}
We require the elementary
 \begin{lemma}\label{lemma:jon}
Let $C\in (0,\infty)$. Put $\phi(t):= \sqrt{C t \log\log (1/t)}$, which is well-defined and $\uparrow\uparrow$ for all sufficiently small $t\in (0,\infty)$.  Then, for any $C^\prime\in (C,\infty)$,  for all sufficiently small $u\in (0,\infty)$, 
     \[
     \phi\left(\frac{1}{C^\prime} \frac{u^2}{\log\log (1/u)}\right)\leq u.
     \]
 \end{lemma}
 \begin{proof}
 Writing $t_u:= \frac{1}{C^\prime} \frac{u^2}{\log\log (1/u)} $, then $\log \log(1/t_u)\sim\log\log(1/u)$ as $u\downarrow 0$, and in particular
 $ \frac{\log \log(1/t_u)}{\log\log(1/u)} \leq \frac{C^\prime}{C}$ for sufficiently small $u\in (0,\infty)$. Now, substituting for $t_u$ in the definition of the function $\phi$ gives, for such $u$, $ \phi(t_u) = u \sqrt{ \frac{C}{C^\prime} \frac{\log \log(1/t_u)}{\log\log(1/u)}} \leq u$. 
 \end{proof}
 \begin{proof}[Proof of Proposition~\ref{proposition:density-of-sets}] Denote $\tilde W:=\mathbbm{1}_E\cdot W$ and  $\tilde W':= 1_{\mathbb{R}\backslash E} \cdot W$. For arbitrary real $s<t$ for which $E_{s,t}>0$, the findings of \ref{maxima-of-censored} from p.~\pageref{maxima-of-censored}, in particular \eqref{aux:loc-max-censored}, show that $(\tilde\tau_{s,t})_\star\WW\ll \leb(\cdot\cap E)$, where $\tilde{\tau}_{s,t}$ is the maximizer of $\tilde W$ on $[s,t]$. 
 
Let us prove \ref{towards-stable}. Pick arbitrary real $s<t$ for which $E_{s,t}>0$. The second assumption of \ref{towards-stable} and $(\tilde\tau_{s,t})_\star\WW\ll \leb(\cdot\cap E)$ entail 
\begin{equation}\label{jon:3}
\WW\left(\limsup_{h\to 0}\frac{|h| - E_{\tilde{\tau}_{s,t},\tilde{\tau}_{s,t}+h}}{\frac{ |h| \,g(|h|)^2}{ {\log\log (1/(\sqrt{|h|}g(|h|)))}}}<\infty\right)=1.
\end{equation}
  By  Lemma~\ref{lemma:jon} (applied with $u=\sqrt{\vert h\vert}g(\vert h\vert)$, $C=2$), \eqref{jon:2} and \eqref{jon:3},
 \[
\WW\left( \limsup_{h\to 0}\frac{|\tilde W'(\tilde{\tau}_{s,t} +h) -\tilde W'(\tilde{\tau}_{s,t})|}{\sqrt{|h|} g(|h|)}<\infty\right)=1.
 \]
 But the Lebesgue density theorem  \eqref{eq:convergence-of-density} for $E$ and $(\tilde\tau_{s,t})_\star\WW\ll \leb(\cdot\cap E)$ certainly show that $E_{\tilde{\tau}_{s,t} ,\tilde{\tau}_{s,t} +h}\geq |h|/2 $ for sufficiently small $h\in \mathbb{R}\backslash \{0\}$ a.s.-$\WW$. Therefore, the first assumption of \ref{towards-stable} and \eqref{jon:0} with  $\alpha g(2\cdot )$ in lieu of $g$, $\alpha\in (0,\infty)$ arbitrary, render
 \[
 \WW\left( \liminf_{h\to 0}\frac{\tilde{W}(\tilde{\tau}_{s,t}+h)- \tilde{W}( \tilde{\tau}_{s,t})}{\sqrt{|h|} g(|h|)}=\infty\right)=1.
 \]
Finally, writing $W$ as the sum of $\tilde{W}$ and $\tilde W'$ (a.s.-$\WW$) we see from the preceding two displays that $W$ has a local maximum at the time $\tilde\tau_{s,t}$ (a.s.-$\WW$).  By letting  $s<t$ vary through the rationals, say, for which $E_{s,t}>0$, we  see that $\WW$-a.s. all local maxima of the censored Brownian motion $\tilde{W}$ are also local maxima of $W$. As announced, by Theorem~\ref{thm:stability} and Proposition~\ref{proposition:loc-max-of-censored}, this is sufficient to secure the max-total stability of $E$.

 We turn to  the proof of \ref{towards-unstable}. Pick again  arbitrary real $s<t$ for which $E_{s,t}>0$. This time the first assumption of \ref{towards-unstable} and \eqref{jon:0}
 give the  existence of a  $\sigma(\tilde W)\lor 0_\WW$-measurable $(0,\infty)$-valued sequence $(h_i)_{i\in \mathbb{N}}$   that is $ \to 0$, strictly positive (resp. strictly negative) on $2\mathbb{N}$ (resp. $2\mathbb{N}-1$) and  such that $\WW$-a.s.
 \[
 \tilde{W}(\tilde{\tau}_{s,t}+h_i)- \tilde{W}( \tilde{\tau}_{s,t})< \sqrt{ \vert h_i \vert}g(\vert h_i\vert),\quad i\in \mathbb{N}.
 \]
From the independence of $\tilde W$ and $\tilde W'$, from the third hypothesis of \ref{towards-unstable} and from $(\tilde\tau_{s,t})_\star\WW\ll \leb(\cdot\cap E)$, we therefore have, $\WW$-a.s.,
 \begin{align*}
&\limsup_{i\to\infty} {\WW} \left( W(\tilde{\tau}_{s,t}+h_i)<  W( \tilde{\tau}_{s,t}) \vert \tilde W\right)\\
 &= \limsup_{i\to\infty}{\WW} \left( \tilde{W}(\tilde{\tau}_{s,t}+h_i)- \tilde{W}( \tilde{\tau}_{s,t})< -\left( \tilde{W}'(\tilde{\tau}_{s,t}+h_i)-\tilde{W}'( \tilde{\tau}_{s,t})\right)\vert \tilde W\right) \\ 
 &\geq \limsup_{i\to\infty}{\WW} \left( \sqrt{\vert h_i \vert}g(\vert h_i\vert)< -\left( \tilde{W}'(\tilde{\tau}_{s,t}+h_i)-\tilde{W}'( \tilde{\tau}_{s,t})\right)\vert \tilde W\right) \\
 &= \limsup_{i\to\infty}\left(1- \Phi\left( \frac{\sqrt{\vert h_i\vert}g(\vert h_i\vert)}{\sqrt{ \vert h_i\vert-E_{\tilde{\tau}_{s,t},\tilde{\tau}_{s,t}+h_i} }}   \right) \right)>0,
 \end{align*}
 where we have denoted, just for a moment, by $\Phi$ the c.d.f. of the $\mathrm{N}(0,1)$ law and interpret, for definiteness, $\Phi(\infty)=1$.  Now, by (reverse conditional) Fatou it follows that, $\WW$-a.s.,
 \[
\WW \bigl( W(\tilde{\tau}_{s,t}+ h_i)<  W( \tilde{\tau}_{s,t}) \mbox{ for infinitely many } i\in \mathbb{N} \vert \tilde W\bigr)>0.
 \]
 But this probability must be either $0$ or $1$ a.s.-$\WW$ by Kolmogorov's zero-one law for the increments of $\tilde W'$ and by the  independence of $\tilde W$ and $\tilde W'$, indeed conditionally on $\tilde W$ the event in question concerns only increments of the additive process $\tilde W'$ on arbitrarily small neighborhoods of $\tilde \tau_{s,t}$.   Therefore in fact  \[
\WW\left(\WW \left( W(\tilde{\tau}_{s,t}+h_i)<  W( \tilde{\tau}_{s,t}) \mbox{ for infinitely many } i\in \mathbb{N} \vert \tilde W\right)=1\right)=1,
 \]
 so $$\WW \left( W(\tilde{\tau}_{s,t}+h_i)<  W( \tilde{\tau}_{s,t}) \mbox{ for infinitely many } i\in \mathbb{N} \right)=1.$$
 Once again by letting $s<t$ vary over the rationals, say, for which $E_{s,t}>0$ we deduce that $\WW$-a.s. no  local maximum of $\tilde{W}$ is a local maximum of $W$, and this implies that $E$ is max-totally unstable by Theorem~\ref{theorem:max-unstable} (since \ref{towards-unstable} has as its second assumption also that $\leb(E)>0$). 
 \end{proof}
Proposition~\ref{proposition:density-of-sets} allows us to put in evidence non-trivial closed max-totally stable sets (which then, according to Theorem~\ref{thm:stability}, belong to $\overline{\EE}$, since their complements, being open, are also max-totally stable).
\begin{example}\label{example:stable}
Consider, under some probability $\mathbb{F}$, the closure \cite[Section~1.4]{subordinators}  $$E:=\overline{\{X_t:t\in [0,\infty)\}}=\{X_t:t\in [0,\infty)\}\cup \{X_{t-}:t\in (0,\infty)\}$$ of the range of a subordinator $X=(X_t)_{t\in [0,\infty)}$ having L\'{e}vy measure $\Pi$ and drift $d\in (0,\infty)$, $X(0)=0$ (of course $E$ differs from the range of $X$ by a countable set only a.s.-$\FFF$). For $t\in E$, the asymptotics of $E_{t,t+h}$ as $h\to 0$ can then be deduced from known results about the growth of $X$, or rather, of the associated driftless process $X^0$, $$X^0(t):=X(t)-dt,\quad t\in [0,\infty),$$ under suitable assumptions. We have indeed from \cite[Proposition~1.8]{subordinators} that,  $\mathbb{F}$-a.s.,
 \[
 \{ E_{0,h}  \leq t \} = \{ X(t/d) \geq h\},\quad \{h,t\}\subset [0,\infty).
 \]
 Consequently, $\FFF$-a.s., for all $t\leq h$ from $[0,\infty)$,
\begin{equation}\label{eq:ancillary}
 \{ h-E_{0,h} \geq t \} =\{ X((h-t)/d) \geq h\}= \{  X^0((h-t)/d) \geq t\}.
\end{equation}
\begin{enumerate}[(A)]
\item \label{subo:A} Suppose first that the tail $\overline{\Pi}$ of the L\'{e}vy measure satisfies, for some $\alpha\in (2,\infty)$, 
 \[
 \limsup_{x\downarrow 0}{\overline{\Pi} (x)}{x(\log (1/x))^{1+\alpha}}<\infty.
 \]
Note in particular that all stable subordinators satisfy this condition.   We will show that $E$ is max-totally stable a.s.-$\FFF$.  To this end, pick arbitrary $\alpha^\prime\in (2,\alpha)$ and define
  \[
  \theta (x) := \frac{x}{(\log(1/x))^{\alpha^\prime}},\quad x\in (0,1).
  \]
Then, because $\alpha'<\alpha$,
\[
\int_{0+} \overline{\Pi}(\theta(x)) \dd x<\infty;
\]
 consequently, by \cite[Theorem~III.9]{bertoin},
  \[
 \FFF\left( \lim_{h \downarrow 0} X^0(h)/ \theta(h) =0\right)=1.
  \]
  This certainly implies (there is no need to strive for optimality here) that, $\FFF$-a.s., for all sufficiently small $h\in (0,\infty)$,
$ X^0((h-\theta(h))/d)  < \theta(h/d)$, and thus, via \eqref{eq:ancillary}, $
  h- E_{0,h} <\theta (h/d)$. 
So, we have 
\begin{equation}\label{verbatim-onwards}
\FFF\left(\limsup_{h\downarrow 0}\frac{  h- E_{0,h}}{\theta(h/d)}\leq 1\right)=1.
\end{equation} By the Markov property of $X$, $$\FFF\left(\limsup_{h\downarrow 0}\frac{  h- E_{X_t,X_t+h}}{\theta(h/d)}\leq 1\right)=1,\quad t\in [0,\infty).$$ By Tonelli, $\FFF$-a.s., for $\leb$-a.e. $t \in [0,\infty)$, 
  \[
\limsup_{h\downarrow 0}\frac{  h- E_{X_t,X_t+h}}{ \theta (h/d)}\leq 1.
  \]
 Since $\FFF(X_\star\leb\vert_{[0,\infty)}\sim \mathbbm{1}_E\cdot \leb\vert_{[0,\infty)})=1$, in fact  even $\FFF(X_\star\leb\vert_{[0,\infty)}=d^{-1}\mathbbm{1}_E\cdot \leb\vert_{[0,\infty)})=1$ (which is easy to check from \cite[Proposition~1.8]{subordinators}, recalling that $X$ has strictly positive drift $d$)  we deduce that $\FFF$-a.s., for $\leb$-a.e. $t\in E$,
\begin{equation}\label{eq:verbatim-fin}
\limsup_{h\downarrow 0}\frac{  h- E_{t,t+h}}{ \theta (h/d)}\leq 1.
\end{equation}
Finally, by time-reversal of $X$ at deterministic times and using duality \cite[Lemma~II.2]{bertoin} we get the analogous control $\limsup_{h\uparrow 0}\frac{  \vert h\vert- E_{t,t+h}}{ \theta (\vert h\vert/d)}\leq 1$ for $\leb$-a.e. $t\in E$ a.s.-$\FFF$. Taking into account Remark~\ref{remark:alphas}  allows to infer that $\FFF$-a.s. $E$ is max-totally stable (since $\alpha'>2$).  

Notice that the preceding is only really interesting if $\Pi((0,\infty))=\infty$, giving then non-trivial max-totally stable $E$ belonging to $\overline{\EE}$ a.s.-$\FFF$ (but true regardless). For, if the L\'evy measure $\Pi$ is infinite, then $\FFF$-a.s. $E$ is  not equal to an open subset of $\mathbb{R}$ modulo an $\leb$-negligible set  (if it was, then, having positive $\leb$-measure due to $d>0$, it would have to contain $O\backslash N$ for some open non-empty interval $O\subset \mathbb{R}$ and $\leb$-negligible $N$, therefore, being closed, it would have to contain $O$, but this is a.s.-$\FFF$ impossible since $X$ has infinite activity). Conversely, if $\Pi$ is finite, then clearly $E$ is  a union of open intervals modulo a countable set a.s.-$\FFF$.

\item\label{subo:B} We look now for a conditon on the subordinator $X$ ensuring that $E$ is a.s.-$\FFF$ max-totally unstable. Suppose then that the tail $\overline{\Pi}$ of the L\'{e}vy measure satisfies
 \[
 \liminf_{x\downarrow 0}\overline{\Pi} (x) x(\log (1/x))^{3}>0,
 \]
 a condition that can certainly be met (is non-vacuous), and which
  implies a lower bound near $\infty$ on the  Laplace exponent $\Phi$ of the driftless  subordinator $X^0$:
\begin{align*}
\liminf_{\lambda\to\infty}(\log \lambda)^2 \frac{\Phi(\lambda)}{\lambda}&= \liminf_{\lambda\to\infty}(\log \lambda)^2\int_0^\infty e^{-\lambda x} \overline{\Pi}(x) \dd x \\
&\geq\liminf_{\lambda\to\infty}(\log \lambda)^2e^{-1}\int_0^{\lambda^{-1}}  \overline{\Pi}(x) \dd x >0.
\end{align*}
It means that, for some $\delta\in (0,\infty)$, for all large enough $\lambda\in (0,\infty)$, $\Phi(\lambda)\geq \delta \frac{\lambda}{(\log \lambda)^2}=:t_\lambda$; therefore,  since $\Phi\uparrow\uparrow\infty$,  for all large enough $\lambda\in (0,\infty)$, $\Phi^{-1}(t_\lambda)\leq \lambda=\delta^{-1} t_\lambda (\log \lambda)^2 \leq \delta^{-1} t_\lambda (\log t_\lambda)^2$, i.e.
 \[
\limsup_{t\to\infty} \frac{\Phi^{-1}(t)}{t(\log t)^2}<\infty.
 \]
 We want to consider  the normalizing function that appears in \cite[Eq.~(8) with $\gamma=2$]{pruitt}, namely,
 \[
 \theta(h):= \frac{\log \log (1/h)}{ \Phi^{-1}(2h^{-1} \log \log (1/h))},
 \]
which is well-defined for all sufficiently small $h\in (0,\infty)$.
 The upper bound on $\Phi^{-1}$ near $\infty$ implies a lower bound on $\theta$ near $0$:
 \begin{equation}\label{lower-bound}
\liminf_{h\downarrow 0} \theta(h) \frac{(\log(1/h))^2}{h}>0.
 \end{equation}
  Now, according to \cite[Lemma~4]{pruitt},
 \[
\FFF\left( \liminf_{h  \downarrow 0} \frac{ X^0(h)}{ \theta(h) } \geq 1\right)=1.
 \]
On the other hand, it is also clear that $\frac{\theta(h)}{h}\downarrow 0$ as $h\downarrow 0$. From \eqref{eq:ancillary} and the preceding display we are able therefore to conclude that
 $$\FFF\left(\liminf_{h\downarrow 0}\frac{  h- E_{0,h}}{\theta(h/(2d))}\geq 1\right)=1.$$  
Proceeding now essentially verbatim as in \ref{subo:A} between \eqref{verbatim-onwards} \& \eqref{eq:verbatim-fin} we deduce that   $\FFF$-a.s., for $\leb$-a.e. $t\in E$,
$$\liminf_{h\downarrow 0}\frac{  h- E_{t,t+h}}{ \theta (h/(2d))}\geq 1.$$
 and consequently $\FFF$-a.s. $E$ is max-totally unstable by Remark~\ref{remark:alphas} and the noted bound \eqref{lower-bound} on $\theta$ (for sure $\FFF(\leb(E)>0)=1$).
\end{enumerate}
\end{example}

\begin{appendix}
\section{The stable $\sigma$-field belongs to the closure}\label{appendix:ntba}
In this self-contained and, apart from Subsections~\ref{miscellaneous}-\ref{subsection:lattice} that remain in effect, notationally independent appendix we work in the setting of \cite{tsirelson}. In particular: 
\begin{enumerate}[(a)] 
\item $B$ is a noise(-type) Boolean algebra \cite[Definition~1.1]{tsirelson} under an essentially separable probability $\PP$ (meaning that $1_\PP=\PP^{-1}([0,1])$ is generated by a countable family of events together with $0_\PP=\PP^{-1}(\{0,1\})$, equivalently that $\LLL^2(\PP)$ is separable);
\item  $\Cl(B)=\left\{\liminf_{n\to\infty}x_n:x\text{ a sequence in $B$}\right\}$ is the sequential monotone --- or, which is the same, topological \cite[Eq.~(4.5)]{tsirelson} --- closure of $B$  \cite[Theorem~1.6]{tsirelson} (for the relevant topology on the lattice  $\hat\PP$ of complete sub-$\sigma$-fields of $\PP$, see \cite[Subsection~3.1]{tsirelson} -- it is just the strong operator topology of the associated conditional  expectation operators [acting on $\LLL^2(\PP)$]);
\item  $\overline{B}$ is the (noise-type) completion of $B$  \cite[Theorem~1.7]{tsirelson}, that is to say, the family of those members of $\Cl(B)$ that are independently complemented in $\Cl(B)$; 
\item $\AA$ is the von Neumann algebra generated by the conditional expectations $\PP[\cdot\vert x]$, $x\in B$ (or, which amounts to the same $\AA$, $x\in \Cl(B)$), acting on $\LLL^2(\PP)$ \cite[Subsection~7.2]{tsirelson}; 
\item  $\mu$ is the associated ([finite, complete] standard) spectral measure  with spectral sets $S_x\subset S$, $x\in \Cl(B)$ \cite[Eq.~(7.7)]{tsirelson} (it means that $ \AA$ is isomorphic to $\LLL^\infty(\mu)$ via a $^*$-isomorphism $\alpha:\AA\to \LLL^\infty(\mu)$, $\PP[\cdot\vert x]$ corresponding to (multiplication with) $\mathbbm{1}_{S_x}$, i.e. $\alpha(\PP[\cdot\vert x])=\mathbbm{1}_{S_x}$ for $x\in \Cl(B)$), subspaces $H(F)=\alpha^{-1}(\mathbbm{1}_F)\LLL^2(\PP)\subset\LLL^2(\PP)$ for $F\in \mu^{-1}([0,\infty))$ \cite[just before Eq.~(7.10)]{tsirelson}, and counting map $$K:=\text{$\mu$-ess sup }K_b,$$  $b$ running over all the finite noise Boolean subalgebras of $B$ \cite[Eq.~(7.21)]{tsirelson} (we shall recall what $K_b$ stands for just below); 
\item finally, $H^{(1)}=H(\{K=1\})$ is the  first chaos \cite[Definition~1.2, Proposition~7.8]{tsirelson}.
\end{enumerate}

For a finite noise Boolean subalgebra $b$ of $B$: write $\at(b)$ for the atoms of $b$; for $\mu$-a.e. $s$, 
\begin{quote}
\normalsize denote by $\underline{b}(s)$  the smallest $x\in b$ such that $s\in S_x$, 
\end{quote}
so that $$K_b(s)= \vert \at(b)\cap 2^{\underline{b}(s)}\vert,$$ i.e. the number of atoms of $b$ contained in $\underline{b}(s)$. We also introduce the stable $\sigma$-field of $B$,  $$\sigma_{\mathrm{stb}}:=\sigma(H^{(1)})\lor 0_\PP,$$  corresponding  \cite[Proposition~7.9]{tsirelson} \cite[Proposition~7.1(i)]{vidmar-noise} to the stable subspace $$H_{\mathrm{stb}}:=H(\{K<\infty\})=\LLL^2(\PP\vert_{\sigma_{\mathrm{stb}}}).$$
Like for the noise of splitting \eqref{eq:stable-in-vNa}, it means automatically that $\PP[\cdot\vert \sigma_{\mathrm{stb}}]\in \AA$, indeed $\alpha(\PP[\cdot \vert \sigma_{\mathrm{stb}}])=\mathbbm{1}_{\{K<\infty\}}$. We assert however that, moreover,
\begin{theorem}\label{thm:is-in-closure}
 $\sigma_{\mathrm{stb}}\in \mathrm{Cl}(B)$.
\end{theorem}
\begin{remark}\label{rmk:intersect}
So, for some sequence $x$ in $B$, $\sigma_{\mathrm{stb}}=\lim_{n\to\infty}x_n$ and if so, then ($\because$ $\land$ is sequentially continuous on $\mathrm{Cl}(B)$ \cite[Eq.~(4.11)]{tsirelson}) $\sigma_{\mathrm{stb}}=\lim_{n\to\infty}(x_n\land \sigma_{\mathrm{stb}})$.
\end{remark}
In the proof we shall employ the probabilistic method: the desired object will be constructed with positive probability (in fact a.s.) under a probability $\QQ$. The basic idea is as follows.  We approximate $B$ with a sequence $(b_n)_{n\in \mathbb{N}_0}$ of its finite noise Boolean subalgebras, $b_0=\{0_\PP,1_\PP\}$. The atoms of the approximating subalgebras can be considered as belonging to a rooted tree of atoms, $T=\cup_{n\in \mathbb{N}_0}\{n\}\times \at(b_n)$: root $(0,1_\PP)$ ($0_\PP= 1_\PP$ is a trivial case that we exclude for the purposes of this discussion); $(n+1,a)$, $a\in \at(b_{n+1})$, is connected to $(n,a')$, $a'\in \at(b_n)$, if and only if $a\subset a'$, this for all $n\in \mathbb{N}_0$, no other connections. At each approximation level we  ``prune away'' some of the atoms (together with all their ``descendants'') of this tree with a certain probability depending on the level, doing so  $\QQ$-independently across the atoms and the levels. Provided the pruning probabilities are judiciously chosen, what survives in the limit of this abstract pruning belongs to $\sigma_{\mathrm{stb}}$  --- the sensitive infromation has disappeared ---  moreover, the remnant can be made to $\uparrow \sigma_{\mathrm{stb}}$ as we delay the pruning further and further away  along the approximating level. Such pruning techniques have been previously employed by Tsirelson to establish  that a noise (Boolean algebra) cannot be completed unless it is classical \cite[Section 6.3]{picard2004lectures} \cite[Section~7]{tsirelson}. Here the process is a little more delicate, but still possible. Let us make it precise! 

\begin{proof}[Proof of Theorem~\ref{thm:is-in-closure}]
We may and do assume $\mu(K=1)>0$ ($\because$ otherwise the stable part is $0_\PP$, which of course belongs to the closure).  In particular, $0_\PP\ne 1_\PP$.

Let $b=(b_n)_{n\in \mathbb{N}}$ be a $\uparrow$ sequence of finite noise Boolean subalgebras of $B$ such that $K_{b_n}\uparrow K$ as $n\to\infty$ a.e.-$\mu$ (it exists, any will do).

Choose sequences $(p_n)_{n\in \mathbb{N}}$ in $(0,1)$ and $(c_n)_{n\in \mathbb{N}}$ in $\mathbb{N}$ [respectively decaying to $0$ and increasing to $\infty$ 
fast enough, in a sense to be made precise at once] such that 
\begin{equation}\label{conditions:c-p}
\sum_{n=1}^\infty p_n<\infty,
\end{equation}
moreover, 
\begin{equation}\label{eq:delta}
\sum_{m=1}^\infty\underbrace{\sqrt{\sum_{n=m}^\infty p_n}}_{=:\delta_m}<\infty,
\end{equation}
while
\begin{equation}\label{conditions:c-p-2}
\lim_{n\to\infty}(1-p_n)^{c_n}=0.
\end{equation}
In terms of the story above, for $n\in \mathbb{N}$, $p_n$ will be the probability of ``pruning out'' an atom of $b_n$ at ``approximation level'' $n$.  However, the pruning itself will only be done along a subsequence increasing to infinity fast enough relative to the $c_n$, $n\in \mathbb{N}$, as now follows.

 Since $\mu$ is finite, by continuity of $\mu$ from above,  $$\lim_{m\to\infty}\mu(K_{b_{m}}<c_n,K=\infty)=0,\quad n\in \mathbb{N},$$ hence 
there is a sequence $(\mathsf{n}_n)_{n\in \mathbb{N}}$ in $\mathbb{N}$ that is $\uparrow\uparrow\infty$ and such that 
\begin{equation*}
\sum_{n\in \mathbb{N}}\mu(K_{b_{\mathsf{n}(n)}}<c_n,K=\infty)<\infty; 
\end{equation*}
by Borel-Cantelli we deduce that 
\begin{equation*}
\text{[either $K<\infty$ or else ($K_{b_{\mathsf{n}(n)}}\geq c_n$ for all $n\in \mathbb{N}$ large enough)] a.e.-$\mu$.}
\end{equation*}
Replacing $b$ with $b_\mathsf{n}$ if necessary, we may and do assume  that
\begin{equation}\label{eq:funny-condition}
\text{[either $K<\infty$ or else ($K_{b_{n}}\geq c_n$ for all $n\in \mathbb{N}$ large enough)] a.e.-$\mu$.}
\end{equation}

Under a single probability $\QQ$ let now, for each $n\in \mathbb{N}$, $X_n$ be a random element (under $\QQ$) that takes its values in $b_n$ (with the discrete measurable structure) and that results from including in it each atom of $b_n$ independently of the others with probability $p_n$, so that $$\QQ(X_n=x)={\vert \at(b_n)\vert\choose \vert\at(b_n)\cap 2^x\vert}p_n^{\vert\at(b_n)\cap 2^x\vert}(1-p_n)^{\vert\at(b_n)\vert-\vert\at(b_n)\cap 2^x\vert},\quad x\in b_n$$ ($X_n$ is what will be ``pruned out'' at level $n$). We also insist that under $\QQ$ the $X_n$, $n\in \mathbb{N}$, are independent.

Let  $m\in \mathbb{N}$  be fixed for a while. Define the ``level-$m$ delayed running joins''
\begin{equation}\label{eq:theYs}
Y_n^m:=\underbrace{X_m\lor \cdots \lor X_n}_{=0_\PP\text{ for }n<m},\quad n\in \mathbb{N},
\end{equation}
and then ``what survives in the limit of delayed pruning''
\begin{equation*}
Z^m:=\land_{n\in \mathbb{N}} (Y_n^m)',
\end{equation*} 
also \begin{equation*} S^m:=\cap_{n\in \mathbb{N}}S_{(Y_n^m)'}.
\end{equation*} $S^m$ is a random (under $\QQ$) measurable set of $\mu$. We have deliberately avoided calling it $S_{Z^m}$, because  there is liberty up to $\mu$-a.e. equality in choosing each of the $S_{Z_m(\omega)}$ as $\omega$ runs over the sample space of $\QQ$, these choices being not necessarily countably many. But anyway, for all $\omega$ from the sample space of $\QQ$,  $S^m(\omega)=S_{Z^m(\omega)}$ a.e.-$\mu$. In a sense we have exploited the fact that there are only countably many $S_x$, $x\in \cup_{n\in \mathbb{N}}b_n$, to get a ``version'' of $S_{Z^m}$ that is jointly ``sufficiently well-behaved'' in the $(\QQ\times\mu)$-space, namely $(\omega,s)\mapsto \mathbbm{1}_{S^m(\omega)}(s)$ is measurable for $\QQ\times \mu$.

We now notice that thanks to \eqref{conditions:c-p-2}-\eqref{eq:funny-condition}, 
\begin{equation}\label{eq:closure-in-stable-one}
\QQ\text{-a.s.}\quad \mu(S^m\cap \{K=\infty\})=0.
\end{equation}
To explain how one arrives at \eqref{eq:closure-in-stable-one} in more detail we identify, for $\mu$-a.e. $s$, 
\begin{align*}
\{s\in S^m\}&=\cap_{n\in \mathbb{N}}\{s\in S_{(Y_n^m)'}\}= \cap_{n\in \mathbb{N}_{\geq m}}\{s\in S_{X_n'}\}
=\cap_{n\in \mathbb{N}_{\geq m}}\{\underline{b_n}(s)\subset X_n'\}\\
&=\cap_{n\in \mathbb{N}_{\geq m}}\cap_{a\in \at(b_n)\cap 2^{\underline{b_n}(s)}}\{a\not\subset X_n\},
\end{align*}
 conclude from \eqref{conditions:c-p-2}-\eqref{eq:funny-condition}  that  for $\mu$-a.e. $s\in \{K=\infty\}$, $$\QQ(s\in S^m)\leq \liminf_{n\to\infty}(1-p_{n})^{K_{b_{n}}(s)}\leq \lim_{n\to\infty}(1-p_{n})^{c_n}=0,$$ and apply Tonelli's theorem.

Besides, since clearly $\{K=1\}\cap S_x=\cup_{a\in \at(b)\cap 2^x}\{K=1\}\cap S_a$ (a disjoint union) a.e.-$\mu$ for any $x\in b$ for any finite noise Boolean subalgebra $b$ of $B$, then
\begin{align}\label{app:long}
&\QQ[\mu(S^m\cap \{K=1\})]=\QQ[\lim_{n\to\infty}\mu(S_{(Y_n^{m})'}\cap \{K=1\})]\\\nonumber
&=\lim_{n\to\infty}\QQ[\mu(S_{(Y_n^m)'}\cap \{K=1\})]=\lim_{n\to\infty}\QQ[\mu(\cap_{k=m}^nS_{{X_k}'}\cap \{K=1\})]\\\nonumber
&\geq \lim_{n\to\infty}\QQ\left[\mu(K=1)-\sum_{k=m}^n\mu(S_{X_k}\cap \{K=1\})\right]\\\nonumber
&=\mu(K=1)-\sum_{n=m}^\infty\QQ[\mu(S_{X_n}\cap \{K=1\})]\\\nonumber
&=\mu(K=1)-\sum_{n=m}^\infty\sum_{a\in \at(b_n)}\QQ[\mu(S_{a}\cap \{K=1\});a\subset X_n]\\\nonumber
&=\mu(K=1)-\sum_{n=m}^\infty \sum_{a\in \at(b_n)}\mu(S_a\cap \{K=1\})p_n\\\nonumber
&=\mu(K=1)\left(1-\sum_{n=m}^\infty p_n\right)=\mu(K=1)(1-\delta_m^2).
\end{align}
Consequently 
\begin{equation}\label{second-borel-cantelli}
\QQ\Big(\mu(S^m\cap \{K=1\})\leq (1-\delta_m)\mu(K=1)\Big)\leq \delta_m,
\end{equation}
 for otherwise we would get  
\begin{align*}
&\QQ\big[\mu(S^m\cap \{K=1\})\big]\\
&=\QQ\Big[\mu(S^m\cap \{K=1\});\mu(S^m\cap \{K=1\})\leq (1-\delta_m)\mu(K=1)\Big]\\
&\qquad+\QQ\Big[\mu(S^m\cap \{K=1\});\mu(S^m\cap \{K=1\})> (1-\delta_m)\mu(K=1)\Big]\\
&\leq \Big[ (1-\delta_m)\QQ\big(\mu(S^m\cap \{K=1\})\leq (1-\delta_m)\mu(K=1)\big)\\
&\qquad +1-\QQ\big(\mu(S^m\cap \{K=1\})\leq (1-\delta_m)\mu(K=1)\big)\Big]\mu(K=1)\\
&=\Big[1-\delta_m\QQ\big(\mu(S^m\cap \{K=1\})\leq (1-\delta_m)\mu(K=1)\big)\Big]\mu(K=1)< (1-\delta_m^2)\mu(K=1),
\end{align*} 
which is in direct contradiction with \eqref{app:long}. 
(No longer holding $m$ fixed) From \eqref{second-borel-cantelli} together with \eqref{eq:delta}  we deduce via  Borel-Cantelli (again) that  
\begin{equation}\label{eq:closure-in-stable-two}
\QQ\text{-a.s.:} \quad \mu(S^{m}\cap \{K=1\})> (1-\delta_m)\mu(K=1)\text{ for all except finitely many }m\in \mathbb{N},
\end{equation}
while \eqref{eq:closure-in-stable-one} entails 
\begin{equation}\label{eq:closure-in-stable-three}
\QQ\text{-a.s.:}\quad \mu(S^m\cap \{K=\infty\})=0\text{ for all $m\in \mathbb{N}$}.
\end{equation}

Combining \eqref{eq:closure-in-stable-two}-\eqref{eq:closure-in-stable-three} certainly there is at least one $\omega$ from the sample space of $\QQ$ such that

 ($\omega_a$) $\mu(S^m(\omega)\cap \{K=1\})> (1-\delta_m)\mu(K=1)$ for all except finitely many $m\in \mathbb{N}$, 
 
 \noindent but also 
 
 ($\omega_b$)  $\mu(S^m(\omega)\cap \{K=\infty\})=0$ for all $m\in \mathbb{N}$. 
 
 \noindent But then 
$\lor_{m\in \mathbb{N}}Z^{m}(\omega)$ belongs to the closure $\Cl(B)$ (indeed it is the $\uparrow$ join of $\downarrow$ intersections of elements of $B$), is contained in the stable $\sigma$-field ($\because$ by ($\omega_b$) each $Z^m(\omega)$, $m\in \mathbb{N}$, is) and its spectral set, it being equal to  $\cup_{m\in \mathbb{N}}S^m(\omega)$ a.e.-$\mu$, contains the whole of $\{K=1\}$ a.e.-$\mu$ ($\because$ of ($\omega_a$) and the continuity of $\mu$ from below), which means that $\lor_{m\in \mathbb{N}}Z^{m}(\omega)$ is actually equal to the stable $\sigma$-field (since the first chaos generates the stable $\sigma$-field).
\end{proof}

Theorem~\ref{thm:is-in-closure} pays an immediate dividend:
\begin{corollary}\label{corollary:stb-in-closure}
For $x\in \overline{B}$, $\{x\lor \sigma_{\mathrm{stb}},x\land  \sigma_{\mathrm{stb}}\}\subset \Cl(B)$.
\end{corollary}
\begin{proof}
Since $B$ and $\overline{B}$ have the same closure and the same stable $\sigma$-field \cite[Proposition~1.10]{tsirelson} we may just as well assume that $B=\overline{B}$. From Theorem~\ref{thm:is-in-closure} we know that $ \sigma_{\mathrm{stb}}=\lim_{n\to\infty}z_n$ for some sequence $z$ in $B$; Remark~\ref{rmk:intersect}  delivers $\sigma_{\mathrm{stb}}\land x=\lim_{n\to\infty}(z_n\land x)$. Thus $x\land  \sigma_{\mathrm{stb}}\in \Cl(B)$. Then, since $\sigma_{\mathrm{stb}}=( \sigma_{\mathrm{stb}}\land x)\lor ( \sigma_{\mathrm{stb}}\land x')$ (we can see it from $z_n=(z_n\land x)\lor (z_n\land x')$ on passing to the limit, $\lor$ being certainly sequentially continuous ``over independent complements'' \cite[Theorem~3.8]{tsirelson}), we have  $x\lor \sigma_{\mathrm{stb}}=x\lor (\sigma_{\mathrm{stb}}\land x')$. But $\Cl(B)$ is closed for $\lor$ ``over independent complements from $B$''\footnote{We mean that for $x\in B$ and $\{u,v\}\subset\Cl(B)$ if $u\subset x$ and $v\subset x'$, then $u\lor v\in \Cl(B)$. Proof: There are sequences $(u_n)_{n\in \mathbb{N}}$ and $(v_n)_{n\in \mathbb{N}}$ in $B$ converging to $u$ and $v$, respectively. By sequential continuity of $\land$ on $\Cl(B)$ and the sequential continuity of $\lor$ over independent $\sigma$-fields, $(u_n\land x)\lor (v_n\land x')\to u\lor v$ as $n\to\infty$. Q.E.D.} and we deduce that $x\lor \sigma_{\mathrm{stb}}\in \Cl(B)$. 
\end{proof}
Let us close this appendix with a kind of complement to Theorem~\ref{thm:is-in-closure} shedding further light on its non-triviality (cf. Remark~\ref{remark:contrast}). For its formulation and eventual proof, recall from spectral theory that for each $\xi\in \LLL^2(\PP)$ there exists a unique measure $\mu_\xi$ on the spectral space $S$ that is absolutely continuous w.r.t. $\mu$ and such that $$\mu_\xi(S_x)=\PP[\vert\PP[\xi\vert x]\vert^2],\quad x\in B.$$
\begin{theorem}\label{theorem:dense-fuck-stable}
Let $b=(b_n)_{n\in \mathbb{N}}$ be a $\uparrow$ sequence of finite noise Boolean subalgebras of $B$   and set  $B_0:=\cup_{n\in \mathbb{N}}b_n$. Assume that 
\begin{equation}\label{dense-fuck-stable-weird}
\liminf_{n\to\infty}\frac{\vert \at(b_n)\cap 2^x\vert}{\vert\at(b_n)\vert}>0\text{ for all $x\in B_0\backslash \{0_\PP\}$}
\end{equation} [there is a  lower bound on how finely we are dissecting each non-trivial part of $1_\PP$ from $B_0$ relative to the whole using the atoms of $b_n$ as $n\to\infty$]. 
Let $\xi\in \LLL^2(\PP)$ be of zero mean  and one for which 
\begin{equation}\label{dense-fuck-stable-main}
\mu_\xi\left(\lim_{n\to\infty}\frac{K_{b_n}}{\vert \at(b_n)\vert}=0\right)>0
\end{equation}
[somehow the degree of sensitivity, as measured by the rate of growth of $(K_{b_n})_{n\in \mathbb{N}}$, must be smaller than that of $(\vert \at(b_n)\vert)_{n\in \mathbb{N}}$ with positive $\mu_\xi$-measure]. Then there exists a $\uparrow$ sequence $(y_n)_{n\in \mathbb{N}}$ in $B_0$ satisfying:
\begin{enumerate}[(i)]
\item\label{dense-stbl-a} for all $x\in B_0\backslash \{0_\PP\}$, for some (and therefore all large enough) $n\in \mathbb{N}$, $y_n\land x\ne 0_\PP$ [in a very figurative sense we may say that $\lor_{n\in \mathbb{N}}y_n$ is ``dense'' relative to the ``open sets'' of $B_0$];
\item\label{dense-stbl-b} $\PP[\xi\vert  \land_{n\in \mathbb{N}}y_n']\ne 0$ [some of $\xi$ ``survives'' on $\land_{n\in \mathbb{N}}y_n'$]. 
\end{enumerate}
Suppose now instead of \eqref{dense-fuck-stable-main} the stronger condition
\begin{equation}\label{dense-fuck-stable-dense}
\text{$\mu_{\PP[\xi\vert x]}\left(\lim_{n\to\infty}\frac{K_{b_n}}{\vert \at(b_n)\vert}=0\right)>0$ for all $x\in B_0\backslash \{0_\PP\}$},
\end{equation}
holds true.  Then in lieu of \ref{dense-stbl-b} we may ask for the fiercer
\begin{enumerate}[resume*]
\item\label{dense-stbl-c} 
$\PP[\xi\vert x\land ( \land_{n\in \mathbb{N}}y_n')]\ne 0$ for all $x\in B_0$ for which $x\land  ( \land_{n\in \mathbb{N}}y_n')\ne 0_\PP$ and also for $x=1_\PP$ [some of $\xi$ ``survives densely'' on  $\land_{n\in \mathbb{N}}y_n'$].
\end{enumerate}
\end{theorem}
\begin{remark}
Condition \eqref{dense-fuck-stable-dense} holds true for a not-mean zero, non-zero $\xi$ automatically and indeed we may then take $y_n=1_\PP$ for all $n\in \mathbb{N}$; such a case would be trivial and is for this reason excluded. On the other hand, for a mean zero $\xi$, \eqref{dense-fuck-stable-main} can only ever hold true if $\lim_{n\to\infty}\vert \at(b_n)\vert=\infty$; when $\lim_{n\to\infty}\vert \at(b_n)\vert=\infty$, then \eqref{dense-fuck-stable-main} is automatic for a stable  (mean zero)  non-zero $\xi$ (stable means that $\xi\in H_{\mathrm{stb}}$), since in such case $\lim_{n\to\infty}\frac{K_{b_n}}{\vert \at(b_n)\vert}=0$ a.e.-$\mu$ on $\{K<\infty\}$ (and therefore a.e.-$\mu_\xi$). 
\end{remark}
\begin{remark}\label{remark:adiuvat}
If $\lim_{n\to\infty}\frac{K_{b_n}}{\vert \at(b_n)\vert}=0$ a.e.-$\mu_\xi$, then: \eqref{dense-fuck-stable-main} holds true also with $\PP[\xi\vert y]$ in lieu of $\xi$ for any $y\in B_0\backslash \{0\}$ for which $\PP[\xi\vert y]\ne 0$; \eqref{dense-fuck-stable-dense} holds true if $\PP[\xi\vert x]\ne 0$ for all $x\in B_0\backslash \{0_\PP\}$.
\end{remark}
\begin{remark}\label{remark:contrast}
In case $\xi$ is sensitive --- meaning that $\PP[\xi\vert\sigma_{\mathrm{stb}}]=0$ ---  \ref{dense-stbl-b} (resp. \ref{dense-stbl-c}) precludes the possibility that $\land_{n\in \mathbb{N}}y_n'$  is included in the stable $\sigma$-field (resp. on any non-trivial intersection with a member of $B_0$), which contrasts  with Theorem~\ref{thm:is-in-closure}, according to which such decreasing intersections were found to approximate arbitrarily well  the stable $\sigma$-field in case $B_0$ is sufficient for (i.e. has the same first chaos \cite[top of p.~316]{tsirelson} and hence the same stable part as)  $B$. (Surprisingly, by a result of Tsirelson \cite[Theorem~1.13]{tsirelson} any atomless $B_0$ is sufficient.)
\end{remark}
The method of proof of Theorem~\ref{theorem:dense-fuck-stable} is very much similar to that of Theorem~\ref{thm:is-in-closure} so we will allow ourselves  ever so slightly more scarcity of detail.
\begin{proof}[Proof of Theorem~\ref{theorem:dense-fuck-stable}]
First, pick any sequence $(\zeta_n)_{n\in \mathbb{N}}$ in $(0,1)$ that is $\downarrow 0$ and is such that  $\sum_{n\in \mathbb{N}}(1-\zeta_n^{\zeta_n})<\infty$ (it is possible because $\lim_{\zeta\downarrow 0}\zeta^\zeta=1$), so that \begin{equation}\label{zeta-choice}\prod_{n\in \mathbb{N}} {\zeta_n}^{\zeta_n}>0.\end{equation}

Second, let $b=(b_n)_{n\in \mathbb{N}}$ be any $\uparrow$ sequence of finite noise Boolean subalgebras of $B$ such that $K_{b_n}\uparrow K$ as $n\to\infty$ a.e.-$\mu$; because $\mu_\xi$ is finite, by bounded convergence we have that for each $m\in \mathbb{N}$, $$\lim_{l\to\infty}\mu_\xi\left(\lim_{n\to\infty}\frac{K_{b_n}}{\vert \at(b_n)\vert}=0,\frac{K_{b_l}}{\vert \at(b_{l})\vert}> \zeta_m\right)=0$$ and therefore  there exists a $\uparrow\uparrow$ sequence $(\mathsf{n}_n)_{n\in\mathbb{N}}$ in $\mathbb{N}$ such that 
$$\sum_{m\in \mathbb{N}}\mu_\xi\left(\lim_{n\to\infty}\frac{K_{b_n}}{\vert \at(b_n)\vert}=0,\frac{K_{b_{\mathsf{n}_m}}}{\vert \at(b_{\mathsf{n}_m})\vert}> \zeta_m\right)<\infty;$$ by Borel-Cantelli we deduce that $\frac{K_{b_{\mathsf{n}_n}}}{\vert\at(b_{\mathsf{n}_n})\vert}\leq \zeta_n$ for all except finitely many $n\in \mathbb{N}$ a.e.-$\mu_\xi$ on $\left\{\lim_{n\to\infty}\frac{K_{b_n}}{\vert \at(b_n)\vert}=0\right\}$; replacing $b$ with $b_\mathsf{n}$ if necessary we may and do assume that 
\begin{equation}\label{eq:adiuvat}
\text{$\frac{K_{b_{n}}}{\vert \at(b_n)\vert}\leq \zeta_n$ for all except finitely many $n\in \mathbb{N}$ a.e.-$\mu_\xi$ on $\left\{\lim_{n\to\infty}\frac{K_{b_n}}{\vert \at(b_n)\vert}=0\right\}$}
\end{equation}
and resolve to pruning.

Setting $$p_n:=1-\sqrt[\vert\at(b_n)\vert ]{\zeta_n}\in (0,1),\quad n\in \mathbb{N},$$ construct under a probability $\QQ$ the random sequences $X=(X_n)_{n\in \mathbb{N}}$ and $Y=(Y_n)_{n\in \mathbb{N}}:=(Y_n^1)_{n\in \mathbb{N}_0}$  [no ``delay'']\phantomsection\label{no-delay} as in, and in the paragraph just before \eqref{eq:theYs}. To wit, the $X_n$, $n\in \mathbb{N}$, are independent random elements (under $\QQ$), $X_n$ taking values in $b_n$, $$\QQ(X_n=x)={\vert \at(b_n)\vert\choose \vert\at(b_n)\cap 2^x\vert}p_n^{\vert\at(b_n)\cap 2^x\vert}(1-p_n)^{\vert\at(b_n)\vert-\vert\at(b_n)\cap 2^x\vert},\quad x\in b_n, n\in \mathbb{N},$$  finally 
\begin{equation*}
Y_n:=X_1\lor \cdots \lor X_n,\quad n\in \mathbb{N}.
\end{equation*}
 (We shall eventually set $y_n:=Y_n(\omega)$ for all $n\in \mathbb{N}$ for some $\omega$ from the sample space of $\QQ$.)

Now, on the one hand, for all $x\in B_0\backslash \{0_\PP\}$, $$\cap_{n\in \mathbb{N}}\{Y_n\land x=0_\PP\}=\cap_{n\in \mathbb{N}}\{X_n\land x=0_\PP\}=\cap_{n\in \mathbb{N}}\cap_{a\in \at(b_n)\cap 2^x}\{a\not\subset X_n\},$$ the $\QQ$-probability of which is $$\leq\liminf_{n\to\infty}(1-p_n)^{\vert \at(b_n)\cap 2^x\vert}=\lim_{n\to\infty}\zeta_n^{\frac{\vert \at(b_n)\cap 2^x\vert}{\vert \at(b_n)\vert}}=0$$ due to \eqref{dense-fuck-stable-weird} and because $\lim_{n\to\infty}\zeta_n=0$. Since $B_0$ is countable, we deduce that 
\begin{equation}\label{first-hit}
\QQ\text{-a.s.:}\quad Y_n\land x\ne 0_\PP\text{ for some }n\in \mathbb{N}\text{ for all }x\in B_0\backslash \{0_\PP\}.
\end{equation}

On the other hand, for $\mu$-a.e. $s$, 
\begin{align*}
\left\{s\in \cap_{n\in \mathbb{N}}S_{Y_n'}\right\}&=\cap_{n\in \mathbb{N}}\{\underline{b_n}(s)\subset Y_n'\}=\cap_{n\in \mathbb{N}}\{\underline{b_n}(s)\subset X_n'\}\\
&=\cap_{n\in \mathbb{N}}\cap_{a\in \at(b_n)\cap 2^{\underline{b_n}(s)}}\{a\not\subset X_n\},
\end{align*} the $\QQ$-probability of which is $$ \prod_{n\in \mathbb{N}}(1-p_n)^{K_{b_n}(s)}=\prod_{n\in \mathbb{N}}\zeta_n^{\frac{K_{b_n}(s)}{\vert\at(b_n)\vert}}$$ and this is $>0$ thanks to \eqref{zeta-choice}-\eqref{eq:adiuvat} if further  $s\in \left\{\lim_{n\to\infty}\frac{K_{b_n}}{\vert \at(b_n)\vert}=0\right\}$, the latter having positive $\mu_\xi$-measure by  \eqref{dense-fuck-stable-main}. Via Tonelli we infer that with positive $\QQ$-probability $\mu_\xi(\cap_{n\in \mathbb{N}}S_{Y_n'})>0$, i.e. 
\begin{equation}\label{second-hit}
\QQ\left(\PP[\vert\PP[\xi\vert \land_{n\in \mathbb{N}}(Y_n)']\vert^2]>0\right)>0.
\end{equation}
Combining \eqref{first-hit}-\eqref{second-hit} it is now elementary to conclude \ref{dense-stbl-a}-\ref{dense-stbl-b} in the manner indicated paranthetically just above.

Suppose now that \eqref{dense-fuck-stable-dense} holds true. In order to get  \ref{dense-stbl-a}-\ref{dense-stbl-c} we have to be a little more careful.  

We begin by noting that the $\uparrow$ union $\cup_{m\in\mathbb{N}}\{\mu_\xi(\cap_{n\in \mathbb{N}_{\geq m}}S_{X_n'})>0\}$ is a tail event of the sequence $X$. By the above this event has positive $\QQ$-probability; by Kolmogorov's zero-one law it has $\QQ$-probability one, and so for any given $\epsilon\in (0,1)$ there is $m\in \mathbb{N}$ such that $\QQ(\mu_\xi(\cap_{n\in \mathbb{N}_{\geq m}}S_{X_n'})>0)\geq 1-\epsilon$. 

Let  $s=(s_n)_{n\in \mathbb{N}}$ be a summable sequence with values in $(0,1)$.  Inductively, for each $n\in \mathbb{N}$, applying the preceding discussion to $\PP[\xi\vert a]$, $a\in \at(b_{\mathsf{m}_{n-1}})$ ($\mathsf{b}_0:=\{0_\PP,1_\PP\}$), in lieu of $\xi$ and exploiting independence we get existence of $\mathsf{m}_n>\mathsf{m}_{n-1}$ ($\mathsf{m}_0:=0$) such that $$\QQ(\cap_{a\in \at(b_{\mathsf{m}_{n-1}})}\{\mu_\xi(S_a\cap (\cap_{k\in \mathbb{N}_{\geq \mathsf{m}_n}}S_{X_k'}))>0\})\geq 1-s_n.$$ By Borel-Cantelli,
\begin{align}\label{third-hit}
&\QQ\text{-a.s.:}\text{ for all but finitely many $n\in \mathbb{N}$, for all $a\in \at(b_{\mathsf{m}_{n-1}})$,} \\\nonumber
&\qquad \qquad\qquad \qquad \qquad\qquad\qquad \qquad \qquad \mu_\xi(S_a\cap (\cap_{k\in \mathbb{N}_{\geq \mathsf{m}_n}}S_{X_k'}))>0.
\end{align}
Of course we also have the following strengthening of \eqref{first-hit}:
\begin{equation}\label{fifth-hit}
\QQ\text{-a.s.:}\text{ for all $n\in \mathbb{N}$, for all $x\in B_0\backslash \{0_\PP\}$, for some $k\in \mathbb{N}_{\geq n}$, $X_{\mathsf{m}_k}\land x\ne 0_\PP$}.
\end{equation}

Now take any $\omega$ from the sample space of $\QQ$ on which the $\QQ$-a.s. events evidenced in \eqref{third-hit}-\eqref{fifth-hit} transpire. By \eqref{third-hit} there is $N\in \mathbb{N}$ such that for all $n\in \mathbb{N}_{\geq N}$, for all $a\in \at(b_{\mathsf{m}_{n-1}})$, 
\begin{equation}\label{app:last}
\mu_\xi(S_a\cap (\cap_{k\in \mathbb{N}_{\geq \mathsf{m}_n}}S_{X_k'(\omega)}))>0.
\end{equation}
 By \eqref{fifth-hit}, for all $x\in B_0\backslash \{0_\PP\}$, for some $k\in \mathbb{N}_{\geq N}$, $X_{\mathsf{m}_k}(\omega)\land x\ne 0_\PP$.  

Setting 
\begin{equation*}
y_l:=\underbrace{\lor_{k=N}^lX_{\mathsf{m}_k}(\omega)}_{=0_\PP\text{ if }l<N},\quad l\in \mathbb{N},
\end{equation*}
 certainly we have \ref{dense-stbl-a}. As for \ref{dense-stbl-c}, for sure $\PP[\xi\vert \land_{n\in \mathbb{N}}y_n']\ne 0$.  Suppose that for some  $a\in \cup_{l\in \mathbb{N}}\at(b_l)$ we have $a\land (\land_{l\in \mathbb{N}}y_l')\ne 0_\PP$. We seek to show $\PP[\xi\vert a\land (\land_{n\in \mathbb{N}}y_n')]\ne 0$, which will demonstrate \ref{dense-stbl-c}. We may and do assume $a\in \at(b_{\mathsf{m}_{n-1}})$ for some $n\in \mathbb{N}_{\geq N}$ that we fix. Now, since $a\land (\land_{l\in \mathbb{N}}y_l')\ne 0_\PP$, it must be that for every $k\in \mathbb{N}\cap [N,n-1]$, $$a\land X_{\mathsf{m}_k}'(\omega)\ne 0_\PP\quad  \therefore\quad  a\subset X_{\mathsf{m}_k}'(\omega)$$ ($\because$ every  atom of $b_{\mathsf{m}_k}$ is the join of some of the atoms of $b_{\mathsf{m}_n}$) and so $a\land  (\land_{l\in \mathbb{N}}y_l')=a\land (\land_{l\in \mathbb{N}_{\geq n}}X_{\mathsf{m}_l}'(\omega))$. But then by  \eqref{app:last}, $\PP[\PP[\xi\vert a\land  (\land_{l\in \mathbb{N}}y_l')]^2]=\mu_\xi(S_a\cap (\cap_{l\in \mathbb{N}_{\geq n}}S_{X_{\mathsf{m}_l}'(\omega)}))\geq \mu_\xi(S_a\cap (\cap_{k\in \mathbb{N}_{\geq \mathsf{m}_n}}S_{X_{k}'(\omega)}))>0$. 
\end{proof}

 \section{A coupling lemma}\label{appendix:coupling}
 Here we present an elementary but quite non-obvious and neat result involving conditioning that may also be of independent interest. Like Appendix~\ref{appendix:ntba} this section stands by itself except for the agreements of Subsection~\ref{miscellaneous} (those of~\ref{subsection:lattice} are not needed here). 
 \begin{lemma}\label{lemma:elementary}
On a probability space $(\Omega,\FF,\PP)$ let there be given random elements $X$, $X_1$, $X_2$ taking values in a countably separated measurable space $(E,\mathcal{E})$, also square-integrable complex-valued random variables $g$, $g_1$, $g_2$, and assume the existence of sub-$\sigma$-fields $\FF_1\subset \FF_2$ such that, for each $i\in \{1,2\}$, the following holds: $(g,X)$ is independent of and identically distributed as $(g_i,X_i)$, both of these \emph{given} $\FF_i$. Then $$0\leq \PP[\overline{g}g_1;X=X_1]\leq \PP[\overline{g}g_2;X=X_2].$$ 
\end{lemma}
\begin{proof}
Because $(E,\mathcal{E})$ is countably separated we have a dissecting sequence $(J_n)_{n\in \mathbb{N}}$, i.e. a sequence of ever finer measurable finite partitions of $E$, whose union separates the points of $E$. Then we compute:
\begin{align*}
&\PP[\overline{g}g_1;X=X_1]=\PP[\overline{g}g_1;\cap_{n\in \mathbb{N}}\cup_{j\in J_n}\{X\in j,X_1\in j\}]\\
&=\lim_{n\to\infty}\PP[\overline{g}g_1;\cup_{j\in J_n}\{X\in j,X_1\in j\}]=\lim_{n\to \infty}\sum_{j\in J_n}\PP[\overline{g}g_1;X\in j,X_1\in j]\\
&=\lim_{n\to \infty}\sum_{j\in J_n}\PP[\PP[\overline{g}g_1;X\in j,X_1\in j\vert\FF_1]]=\lim_{n\to \infty}\sum_{j\in J_n}\PP[\PP[\overline{g}\mathbbm{1}_{\{X\in j\}}\vert \FF_1]\PP[g_1\mathbbm{1}_{\{X_1\in j\}}\vert\FF_1]]\\
&\,\,\qquad \qquad\qquad \qquad \qquad\qquad \qquad \qquad  \text{ ($\because$ $(g,X)$ is independent of $(g_1,X_1)$ given $\FF_1$)}\\
&=\lim_{n\to \infty}\sum_{j\in J_n}\PP[\vert\PP[g\mathbbm{1}_{\{X\in j\}} \vert\FF_1]\vert^2]\geq 0 \,\text{ ($\because$ $(g,X)$ has the same law as $(g_1,X_1)$ given $\FF_1$)}.
\end{align*}
By the same token
$$\PP[\overline{g}g_2;X=X_2]=\lim_{n\to \infty}\sum_{j\in J_n}\PP[\vert\PP[g\mathbbm{1}_{\{X\in j\}} \vert\FF_2]\vert^2].$$
But $\FF_1\subset \FF_2$ and conditioning is a projection, therefore a contraction on $\LLL^2(\PP)$, so
\begin{equation*}
\PP[\vert\PP[g\mathbbm{1}_{\{X\in j\}} \vert\FF_2]\vert^2]\geq \PP[\vert\PP[\PP[g\mathbbm{1}_{\{X\in j\}} \vert\FF_2]\vert\FF_1]]\vert^2]=\PP[\vert\PP[g\mathbbm{1}_{\{X\in j\}} \vert\FF_1]\vert^2],
\end{equation*} for all $j\in \cup_{n\in \mathbb{N}}J_n$, wherefrom we conclude at once.
\end{proof} 
\end{appendix}

\bibliographystyle{plain}
\bibliography{Biblio_noise}
\end{document}